\documentclass[10pt]{amsart}
\usepackage{amssymb, amsthm, amsmath}
\usepackage[all]{xy}
\usepackage{graphicx}
\usepackage{psfrag}

\newtheorem{prop}{Proposition}[section]
\newtheorem{thm}{Theorem}[section]
\newtheorem{cor}{Corollary}[section]
\newtheorem{lemma}{Lemma}[section]
\theoremstyle{definition}
\newtheorem{defn}{Definition}[section]
\newtheorem{example}{Example}[section]
\newtheorem*{remark}{Remark}

\newcommand\A{{\mathbb A}}
\newcommand\C{{\mathbb C}}
\newcommand\Q{{\mathbb Q}}

\newcommand\bP{{\mathbb P}}
\newcommand\cO{{\mathcal O}}

\newcommand\Z{{\mathbb Z}}
\newcommand\cP{{\mathcal P}}
\newcommand\cQ{{\mathcal Q}}
\newcommand\cS{{\mathcal S}}

\newcommand\cV{{\mathcal V}}

\newcommand\la{\lambda}
\newcommand\s{{\sigma}}
\newcommand\ta{{\tau}}

\newcommand\ssm{\smallsetminus}
\newcommand\gequ{\geq}
\newcommand\lequ{\leq}
\newcommand\noin{\noindent}
\newcommand\bull{{\scriptscriptstyle \bullet}}
\newcommand\eqto{\stackrel{\lower1.5pt\hbox{$\scriptstyle\sim\,$}}\to}
\newcommand\ov{\overline}

\newcommand\wt{\widetilde}
\newcommand\dis{\displaystyle}
\newcommand\pp[2]{{({#1}\mathchoice{\,}{\,}{}{}|\mathchoice{\,}{\,}{}{}{#2})}}

\newcommand{\floor}[1]{\lfloor #1 \rfloor}
\DeclareMathOperator{\Ker}{Ker}
\DeclareMathOperator{\Span}{Span}

\DeclareMathOperator{\codim}{codim}

\DeclareMathOperator{\HH}{\mathrm{H}}
\DeclareMathOperator{\QH}{\mathrm{QH}}
\DeclareMathOperator{\Sp}{Sp}
\DeclareMathOperator{\SO}{SO}
\DeclareMathOperator{\SL}{SL}
\DeclareMathOperator{\GL}{GL}
\DeclareMathOperator{\LG}{LG}
\DeclareMathOperator{\IG}{IG}
\DeclareMathOperator{\IF}{IF}
\DeclareMathOperator{\OG}{OG}
\DeclareMathOperator{\G}{G}
\DeclareMathOperator{\OF}{OF}
\DeclareMathOperator{\type}{\mathrm{type}}
\DeclareMathOperator{\rank}{\mathrm{rank}}
\DeclareMathOperator{\point}{\mathrm{point}}
\newcommand{\gw}[2]{\langle #1 \rangle^{\mbox{}}_{#2}}
\newcommand{\pic}[2]{\includegraphics[scale=#1]{#2}}
\newcommand{\hmm}[1]{\hspace{#1mm}}

\psfrag{k}{$k$}
\psfrag{(r,c)}{$(r,c)$}
\psfrag{(r',c')}{$(r',c')$}

\begin{document}

\title[Quantum Pieri rules for isotropic Grassmannians]
{Quantum Pieri rules for isotropic Grassmannians}

\date{September 29, 2008}

\author{Anders Skovsted Buch}
\address{Department of Mathematics, Rutgers University, 110
  Frelinghuysen Road, Piscataway, NJ 08854, USA}
\email{asbuch@math.rutgers.edu}

\author{Andrew Kresch}
\address{Institut f\"ur Mathematik,
Universit\"at Z\"urich, Winterthurerstrasse 190,
CH-8057 Z\"urich, Switzerland}
\email{andrew.kresch@math.uzh.ch}

\author{Harry~Tamvakis} \address{University of Maryland, Department of
Mathematics, 1301 Mathematics Building, College Park, MD 20742, USA}
\email{harryt@math.umd.edu}

\subjclass[2000]{Primary 14N35; Secondary 14M15, 14N15}

\thanks{The authors were supported in part by NSF Grant DMS-0603822
  (Buch), the Swiss National Science Foundation (Kresch), and NSF
  Grants DMS-0401082 and DMS-0639033 (Tamvakis).}

\begin{abstract}
We study the three point genus zero Gromov-Witten invariants on the
Grassmannians which parametrize non-maximal isotropic subspaces in
a vector space equipped with a nondegenerate symmetric or
skew-symmetric form. We establish Pieri rules for the classical cohomology 
and the small quantum cohomology ring of these varieties, which give a 
combinatorial formula for the product of any Schubert class with 
certain special Schubert classes. We also give presentations of these
rings, with integer coefficients, in terms of special Schubert class
generators and relations.
\end{abstract}

\maketitle 

\setcounter{section}{-1}

\section{Introduction}

Let $G$ be a classical Lie group of type B, C, or D, and $P$ any
maximal parabolic subgroup of $G$. The homogeneous space $X=G/P$ is a
Grassmannian which parametrizes isotropic subspaces in a vector space
equipped with a nondegenerate symmetric or skew-symmetric bilinear
form.  The small quantum cohomology ring $\QH(X)$ is a deformation of
the cohomology ring $\HH^*(X,\Z)$, defined using structure constants
given by the (three point, genus zero) Gromov-Witten invariants of
$X$. The ring $\QH(X)$ is generated as an algebra over $\Z[q]$ by
certain special Schubert classes, in a similar fashion to the ordinary
cohomology ring of $X$, which is recovered when the formal variable $q$ is
set equal to zero. The main purpose of this paper is to formulate and
prove a {\em quantum Pieri rule}, a combinatorial formula which
describes the product of a special Schubert class with an arbitrary
one in the quantum cohomology ring.

A quantum Pieri rule for the usual type A Grassmannian was proved by
Bertram \cite{Bertram}, while corresponding rules for the
Grassmannians of maximal isotropic subspaces were established by
the second and third authors \cite{KTlg, KTog}. We deal here with the
remaining cases, where $X$ is a non-maximal isotropic symplectic or
orthogonal Grassmannian. Our results are applied to obtain
presentations of $\QH(X)$ with integer coefficients in terms of special
Schubert class generators and relations.

One notable feature is that the quantum Pieri rule in the
orthogonal cases involves quadratic $q$ terms, which is a new
phenomenon for Grassmannians. Algebraically, it is simplest to
understand the transition between the maximal and non-maximal
isotropic cases in type D, by considering the space $\OG(n-k,2n)$ of
isotropic $(n-k)$-dimensional subspaces of $\C^{2n}$, for $k\gequ
0$. When $k=1$, the Picard group of this orthogonal Grassmannian has
rank two. Therefore the quantum cohomology ring of $\OG(n-1,2n)$
involves two deformation parameters $q_1$ and $q_2$, whereas
$\OG(n,2n)$ and $\OG(n-k,2n)$ for $k>1$ have only one deformation
parameter $q$. Formally, the single quantum parameter on $\OG(n,2n)$
is replaced with two square roots on $\OG(n-1,2n)$, which in turn are
identified on $\OG(n-2,2n)$.

The proofs of the quantum Pieri rules in this article use the
kernel--span and dimension counting ideas of the first author
\cite{Buch} in an essential way; the earlier papers \cite{Bertram,
KTlg, KTog} used intersection theory on certain Quot scheme
compactifications of the moduli space of maps from the projective line
to $X$. Furthermore, we continue the program set out in \cite{BKT} of
equating the Gromov-Witten invariants on $X$ (of any degree) with
classical triple intersection numbers on related varieties
$X'$. However the analysis here is considerably more subtle, as the
parameter space $X'$ of kernel--span pairs will no longer be a
homogeneous space for the group $G$, in general.  For instance, the
parameter space of lines on an orthogonal Grassmannian is a two step
orthogonal flag variety, but the analogous statement is false in the
symplectic case, where additional geometric arguments are needed to
relate the degree one Gromov-Witten invariants to classical Schubert
structure constants (Proposition \ref{P:gwdegone}).

On the other hand, in degrees greater than one, the symplectic
Grassmannians are much better behaved than the orthogonal ones, from
this point of view.  We are able to show that the function $\phi$
which sends a rational map (or curve) counted by a Gromov-Witten
invariant to the pair consisting of its kernel and span is a bijection
in the symplectic case, but for orthogonal Grassmannians the map
$\phi$ is $N$ to $1$, where $N$ is an explicitly determined power of
two.  Additional complications stem from the fact that for any two
points on the maximal orthogonal Grassmannian $\OG(d,2d)$, the
corresponding subspaces of $\C^{2d}$ must intersect in at least a
line, if the integer $d$ is {\em odd}; this makes it impossible to
place three points of $\OG(d,2d)$ in ``pairwise general position'' in
the sense of \cite[Prop.\ 4]{BKT}.

Our Pieri rules and presentations are new even for the classical
cohomology ring $\HH^*(X,\Z)$.  Over each Grassmannian $X$, there is a
universal exact sequence of vector bundles
\[ 0 \to \cS \to \cV_X \to \cQ \to 0 \]
with $\cS$ the tautological subbundle of the trivial bundle $\cV_X$.
Pragacz and Ratajski \cite{PRpieri, PRpieri2} proved Pieri-type rules
for the product of an arbitrary Schubert class with the Chern classes
of $\cS^*$.  An important part of our work was the discovery of new
classical Pieri rules for multiplying with the Chern classes of $\cQ$.
Aside from being much simpler than the formulas in \cite{PRpieri,
PRpieri2}, the new Pieri rules are also essential for our applications
to quantum cohomology.  To be specific, we prove that any quantum
product of $c_p(\cQ)$ with another Schubert class on a symplectic
Grassmannian can be obtained from the product of these classes in the
ordinary cohomology ring of a larger Grassmannian.  However, we do not
know such a relation for all products involving the Chern classes of
$\cS^*$ (see Example \ref{E:lengthcond}).

Our method for proving the new classical Pieri rules is completely
different (and much shorter) than that of loc.\ cit. We use a
geometric approach, following Hodge's classical proof of
the Pieri rule for the usual Grassmannians via triple intersections
\cite{H}; see also \cite{Sertoz, Sottile}.  In type D, the
Chern classes of the tautological bundles $\cS$ and $\cQ$ do not
generate the cohomology ring $\HH^*(X,\Q)$, and thus a direct approach
by intersecting Schubert cells seems necessary to obtain the full
picture. We also note that the classical Borel presentation \cite{Bo}
of the cohomology ring of orthogonal Grassmannians requires rational
coefficients, and uses a different set of generators.

Another innovation of this article is our parametrization of the
Schubert varieties by {\em $k$-strict partitions}, a convention which
makes their codimension apparent.  This notation is most convenient
when working algebraically, in the classical or quantum cohomology
ring, and extends the traditional parametrization for type A and
maximal isotropic Grassmannians. An underlying reason which vindicates
this choice, but will be explained elsewhere, is our work on
Giambelli-type formulas for these varieties; in fact, the latter
project has affected the exposition here in more ways than one.  After
the first three sections we introduce a different parametrization of
the Schubert varieties by {\em index sets}, which is closer to their
geometric definition as the closures of Schubert cells. Our proof of
the classical Pieri rules is in this latter language, where they are
somewhat easier to state.

The aforementioned motivation for the new classical Pieri rules is the
reason behind the {\em in medias res} organization of this paper,
where we begin by studying the quantum cohomology rings and follow
this with a more detailed look at the Schubert varieties, culminating
with our proofs of the classical Pieri rules.  The various sections
are ordered according to the three different Lie types considered.
The quantum cohomology of isotropic Grassmannians of type C, B, and D
is examined in Sections \ref{QCIG}, \ref{QCOGodd}, and \ref{QCOGeven},
respectively. In Section \ref{schvars} we begin our study of the
Schubert varieties from scratch, parametrizing them by index sets and
relating these to $k$-strict partitions.  Section \ref{Pierirule}
contains short proofs of our classical Pieri rules, using the language
of index sets, and then translates these rules to the initial
statements in terms of $k$-strict partitions. The reader who is
willing to grant the classical Pieri rules can omit Section
\ref{Pierirule} entirely. Finally, we provide an appendix which
discusses the quantum cohomology of the orthogonal Grassmannian
$\OG(n,2n+2)$. Most of the results of this paper were announced in the
survey \cite{T}, which employed the older notation of Weyl group
elements and partition pairs.

The authors wish to thank Michael Brion, William Fulton, Laurent
Manivel, Piotr Pragacz, and Miles Reid for helpful discussions.
We also thank Sheldon Katz and Stein Arild Str{\o}mme, the authors
of the Maple package `Schubert', and Jan-Magnus {\O}kland for helpful
comments about this software.


\section{The Grassmannian $\IG(n-k,2n)$}
\label{QCIG}

\subsection{Schubert classes}
\label{igscs}

Fix a vector space $V \cong \C^{2n}$ with a non-degenerate
skew-symmetric bilinear form $(\ ,\,)$, and fix a non-negative integer
$m \lequ n$.  Throughout this section we let the symbol $\IG$, without
parameters, denote the Grassmannian $\IG(m,2n)$ which parametrizes
$m$-dimensional isotropic subspaces of $V$.  This algebraic variety
has dimension $2m(n-m)+m(m+1)/2$. The Schubert varieties in $\IG$ are
described below; our indexing conventions for these varieties appear
to be new, and we refer the reader to Section \ref{schvars} for
elementary proofs of the relevant facts.

A partition is a weakly decreasing sequence of non-negative integers
$\lambda = (\lambda_1 \gequ \dots \gequ \lambda_m \gequ 0)$.  We will
identify a partition $\lambda$ with its Young diagram of boxes, which
has $\lambda_i$ boxes in row $i$.  The weight of $\lambda$ is the sum
$|\lambda| = \sum \lambda_i$ of its parts, and the length
$\ell(\lambda)$ is the number of non-zero parts of $\lambda$.  We set
$\lambda_i = 0$ for $i > \ell(\lambda)$.  

\begin{defn}
Let $k$ be a non-negative integer. We say that the partition $\lambda$ 
is {\em $k$-strict\/} if no part greater than $k$ is repeated, 
i.e.,\ $\lambda_j > k \Rightarrow \lambda_{j+1}<\la_j$.
\end{defn}

Set $k = n-m$.  The Schubert varieties in $\IG$ are indexed by
$k$-strict partitions $\lambda$ which are contained in an $m \times
(n+k)$ rectangle; we denote the set of all such partitions by
$\cP(k,n)$. An isotropic flag in $V$ is a complete flag $0 =
F_0 \subsetneq F_1 \subsetneq \dots \subsetneq F_{2n} = V$ of subspaces
such that $F_{n+i} = F_{n-i}^\perp$ for each $0 \lequ i \lequ n$.  The
Schubert variety for any $\lambda\in \cP(k,n)$ relative to
the isotropic flag $F_\bull$ is defined by
\[ X_\lambda(F_\bull) = \{ \Sigma \in \IG \mid \dim(\Sigma \cap
   F_{p_j(\lambda)}) \gequ j \ \ \forall\, 1 \lequ j \lequ
   \ell(\lambda) \} \,,
\]
where 
\[
p_j(\lambda) = n+k+1-\lambda_j + \#\{i<j : \lambda_i+\lambda_j
\lequ 2k+j-i \}.
\]
The codimension of this variety is equal to $|\lambda|$.  We let
$\sigma_\lambda = [X_\lambda] \in \HH^{2|\lambda|}(\IG) =
\HH^{2|\lambda|}(\IG,\Z)$ denote the cohomology class Poincar\'e dual
to the cycle determined by $X_\lambda(F_\bull)$.  These Schubert
classes form a $\Z$-basis for the cohomology ring of $\IG$.

The $k$-strict partition $\lambda$ has a unique dual partition
$\lambda^\vee\in\cP(k,n)$, for which
$p_j(\lambda^\vee) = 2n+1-p_{m+1-j}(\lambda)$ for $1 \lequ j \lequ m$.
With this notation we have
\[ \int_{\IG} \sigma_\lambda \cdot \sigma_\mu =
   \delta_{\mu,\lambda^\vee} \,.
\]

\subsection{Classical Pieri rule}
\label{igclpieri}
The Schubert varieties $X_p(F_\bull) = X_{(p)}(F_\bull)$ for $1 \lequ p
\lequ n+k$ are called special, and are defined by the single Schubert
condition $\Sigma \cap F_{n+k+1-p} \neq 0$. 
Consider the exact sequence of vector bundles over $\IG$
\[
0 \to \cS \to \cV_{\IG} \to \cQ \to 0,
\]
where $\cV_{\IG}$ denotes the trivial bundle of rank $2n$ and $\cS$ is 
the tautological subbundle of rank $m$. The
special Schubert class $\sigma_p=[X_p(F_\bull)]$ 
is equal to the Chern class $c_p(\cQ)$
of the tautological quotient bundle on $\IG$.
These classes generate the cohomology ring of $\IG$. 
When $m<n$, one also has a set of special
classes $\sigma_{(1^p)}$, for $1 \lequ p \lequ m$, which are
the Chern classes $c_p(\cS^*)$. These latter classes form a different 
set of generators of $\HH^*(\IG)$.

When $m=n$ and $\IG = \LG(n,2n)$ is the Lagrangian Grassmannian of
maximal isotropic subspaces, Hiller and Boe proved a Pieri rule
for any product involving a special class $\sigma_p$ \cite{HB}.  For
$m<n$, Pragacz and Ratajski have proved a more general formula for
multiplying with the classes $\sigma_{(1^p)}$ \cite{PRpieri}.  Our
first result is a new Pieri rule for products involving the special
classes $\sigma_p = c_p(\cQ)$ for general isotropic Grassmannians.  This
rule is easier to state than the rule proved in \cite{PRpieri}.
Furthermore, our methods for obtaining quantum generalizations of the
Pieri rules only work for the special classes $\sigma_p$ (see
Proposition~\ref{P:gwdegone} and Example~\ref{E:lengthcond}).

\begin{defn}
Let $\lambda$ be a $k$-strict partition.  We will say that the box in
row $r$ and column $c$ of $\lambda$ is {\em $k$-related\/} to the
box in row $r'$ and column $c'$ if $|c-k-1|+r = |c'-k-1|+r'$.  For
example, the grey boxes in the following partition are $k$-related.
\[ \pic{0.65}{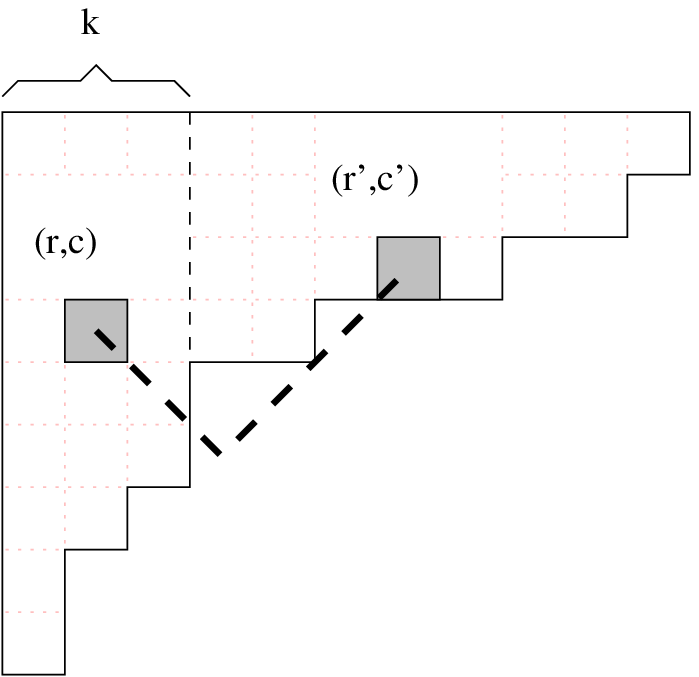} \]
\end{defn}

The notion of $k$-related boxes is one of the ingredients of Pragacz
and Rataj\-ski's formula for products involving the special classes
$\sigma_{(1^p)}$ \cite{PRpieri}. 

Given two Young diagrams $\la$ and $\mu$ with $\la\subset\mu$, the skew
diagram $\mu/\la$ is called a horizontal (resp.\ vertical) strip if
it does not contain two boxes in the same column (resp.\ row).

\begin{defn}
\label{D:pieriarrow}
For any two $k$-strict partitions $\lambda$ and $\mu$, we have a
relation $\lambda \to \mu$ if $\mu$ can be obtained by removing a
vertical strip from the first $k$ columns of $\lambda$ and adding a
horizontal strip to the result, so that

(1) if one of the first $k$ columns of $\mu$ has the same number of
boxes as the same column of $\lambda$, then the bottom box of this
column is $k$-related to at most one box of $\mu \smallsetminus
\lambda$; and

(2) if a column of $\mu$ has fewer boxes than the same column of
$\lambda$, then the removed boxes and the bottom box of $\mu$ in this
column must each be $k$-related to exactly one box of $\mu
\smallsetminus \lambda$, and these boxes of $\mu \smallsetminus
\lambda$ must all lie in the same row.

If $\lambda \to \mu$, we let $\A$ be the set of boxes of $\mu\ssm \la$
in columns $k+1$ through $k+n$ which are {\em not} mentioned in (1) or
(2). Then define $N(\lambda,\mu)$ to be the number of connected
components of $\A$ which do not have a box in column $k+1$.  Here two
boxes are connected if they share at least a vertex.
\end{defn}

\begin{thm}[Pieri rule for $\IG(m,2n)$] \label{T:pieriC}
For any $k$-strict partition $\lambda$ and integer $p \in [1,n+k]$,
we have
\[ \sigma_p \cdot \sigma_\lambda = \sum_{\lambda \to \mu, \,
  |\mu|=|\lambda|+p} 2^{N(\lambda,\mu)} \, \sigma_\mu \,.
\]
\end{thm}

This theorem will be proved in Section~\ref{Pierirule}.

\begin{example} \label{E:pieriC}
  For the Grassmannian $\IG(4,12)$ we have $k=2$.  For $\lambda =
  (5,3,2,2)$ we get the following shapes $\mu\in\cP(2,6)$ such that
  $\lambda \to \mu$ and $|\mu|=|\lambda|+4$:
\[
\includegraphics[scale=0.40,viewport=0 0 148 100]{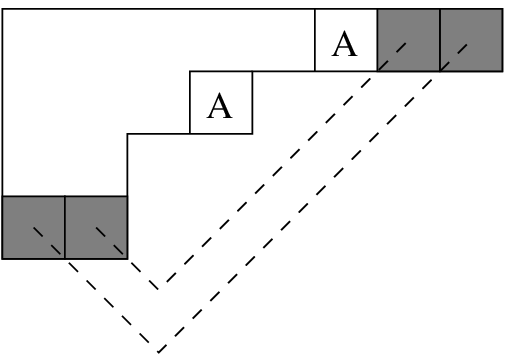}
\hmm{5}
\includegraphics[scale=0.40,viewport=0 -18 129 82]{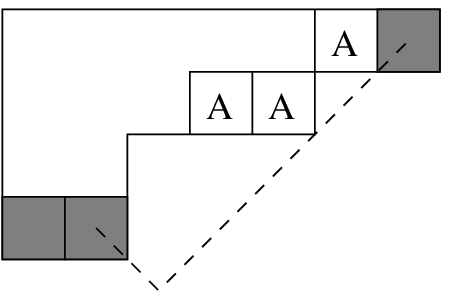}
\hmm{5}
\includegraphics[scale=0.40,viewport=0 -18 129 82]{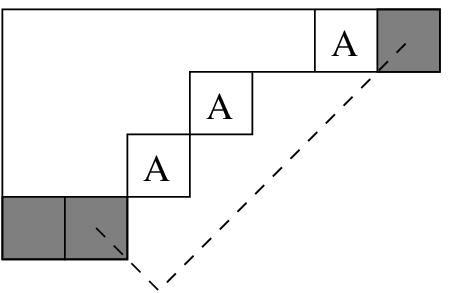}
\hmm{5}
\includegraphics[scale=0.40,viewport=0 -27 110 73]{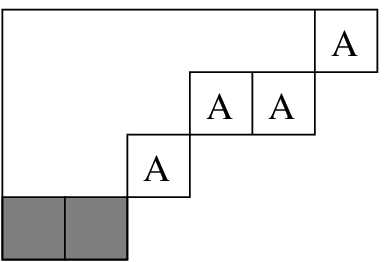}
\hmm{5}
\includegraphics[scale=0.40,viewport=0 0 148 100]{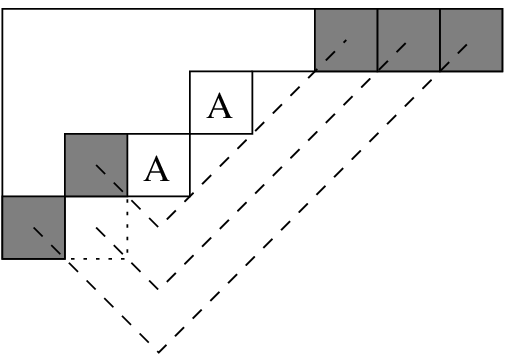}
\]
By Theorem~\ref{T:pieriC} we therefore obtain
\[ \sigma_4 \cdot \sigma_\lambda = 4 \sigma_{(8,4,2,2)} + 2
   \sigma_{(7,5,2,2)} + 2 \sigma_{(7,4,3,2)} + \sigma_{(6,5,3,2)} +
   \sigma_{(8,4,3,1)} \,.
\]
\end{example}

\subsection{Presentation of $\HH^*(\IG)$}
\label{igpresentation}

For every partition $\lambda$, define a monomial $\s^{\la}$ in the
special Schubert classes by $\s^{\la}=\prod_i\s_{\la_i}$.  We also use
the convention that $\sigma_0 = 1$ and $\sigma_p = 0$ for $p<0$ or
$p>n+k$.

\begin{thm} \label{presIG}
 {\em a)} The cohomology ring $\HH^*(\IG,\Z)$ is presented as a
 quotient of the polynomial ring $\Z[\sigma_1,\ldots,\sigma_{n+k}]$
 modulo the relations
\begin{equation}
\label{R1}
\det(\sigma_{1+j-i})_{1\lequ i,j \lequ r} = 0\,, \ \ \ \ 
n-k+1\lequ r \lequ n+k
\end{equation}
and
\begin{equation}
\label{R2}
\sigma_r^2 + 2\sum_{i=1}^{n+k-r}(-1)^i \sigma_{r+i}\sigma_{r-i}= 0\,,
 \ \  \ \ k+1\lequ r \lequ n.
\end{equation}

\medskip
\noin
{\em b)} The monomials $\s^{\la}$ with $\la\in\cP(k,n)$ 
form a $\Z$-basis for $\HH^*(\IG,\Z)$.
\end{thm}

We prove Theorem~\ref{presIG} below based on the following lemma,
which was shown to us by Miles Reid.  It provides a method to obtain
presentations of cohomology rings with integer coefficients which
simplified our earlier proofs of Theorems \ref{presIG},
\ref{presOGodd}, and \ref{presOGeven}.

\begin{lemma}[Reid]
\label{preslemma}
Let $R = \Z[a_1,\dots,a_d]$ be a free polynomial ring with homogeneous
generators $a_i$, let $I \subset R$ be an ideal generated by
homogeneous elements $b_1,\dots,b_d \in R$, and let $\phi : R/I \to H$
be a surjective ring homomorphism.  Assume that (i) $H$ is a free
$\Z$-module of rank $\prod_i (\deg b_i/\deg a_i)$, and (ii) for any
field $K$, the $K$-vector space $(R/I)\otimes_{\Z}K$ has finite
dimension. Then $\phi$ is an isomorphism.
\end{lemma}
\begin{proof}
  The second assumption implies that $(b_1,\ldots,b_d)$ is a regular
  sequence in $K[a_1,\ldots,a_d]$, for any field $K$. We deduce that
  $R/I$ is flat as a $\Z$-module; this follows from
  \cite[0.15.1.16]{EGA}, which reduces the result to an application
  of the local criterion for flatness, as explained in \cite[Lemma
  4.3.16]{L}.  Since $\phi$ is a surjection from $R/I$ onto a free
  $\Z$-module, the kernel of $\phi$ is also flat, so it is enough to
  show that $C = (R/I)\otimes_{\Z} \Q$ and $H\otimes_{\Z} \Q$ have the
  same dimension as $\Q$-vector spaces.  As the graded ring $C$ is a
  complete intersection, its Hilbert series \cite{Sta} is given by
  \[
  \HH(t)=\sum_{j=0}^{\infty}\dim(C_j)\, t^j = 
  \prod_{i=1}^d\frac{1-t^{\deg b_i}}{1-t^{\deg a_i}}.
  \]
  This is a polynomial which evaluates to $\prod_i (\deg b_i /\deg a_i)$
  at $t=1$.
\end{proof}

\begin{lemma}
\label{typeApres}
The quotient of the graded ring $\Z[a_1,\ldots, a_d]$ with 
$\deg a_i=i$ modulo the relations 
\[
\det(a_{1+j-i})_{1\lequ i,j \lequ r} = 0, \ \ \ \ 
m+1\lequ r \lequ m+d
\]
is a free $\Z$-module of rank $\binom{m+d}{d}$.
\end{lemma}
\begin{proof}
  The displayed quotient ring is one of the standard presentations of
  the cohomology ring of the usual (type A) Grassmannian of
  $m$-dimensional subspaces of a vector space of dimension $m+d$.
  Moreover, the Euler characteristic of this Grassmannian equals
  $\binom{m+d}{d}$.
\end{proof}

\begin{proof}[Proof of Theorem \ref{presIG}]
  We write $\mu \succ \la$ if the partition $\mu$ strictly dominates
  $\lambda$, i.e.\ $\mu \neq \lambda$ and $\mu_1 + \dots + \mu_i \geq
  \lambda_1 + \dots + \lambda_i$ for each $i \geq 1$.  It follows from
  the Pieri rule that $\sigma^\lambda = \sigma_\lambda + \sum_{\mu
  \succ \lambda} c_\mu \sigma_\mu$, where $c_\mu \in \Z$ and the sum
  is over partitions $\mu$ with $\mu \succ\la$. This implies (b).  In
  particular, the special Schubert classes $\sigma_1, \dots,
  \sigma_{n+k}$ generate the ring $\HH^*(\IG,\Z)$.
  
  Set $R = \Z[a_1,\dots,a_{n+k}]$, where $a_i$ is a homogeneous
  variable of degree $i$, and let $\phi : R \to \HH^*(\IG, \Z)$ be the
  surjective ring homomorphism defined by $\phi(a_i) = \sigma_i$.  We
  also write $a_0 = 1$ and $a_i = 0$ for $i<0$ or $i > n+k$.  For $r
  \geq 1$ we define 
\[
d_r = \det(a_{1+j-i})_{1 \lequ i,j \lequ r} \ \ \ \text{and} \ \ \ 
b_r = a_r^2 + 2 \sum_{i\geq 1} (-1)^i a_{r+i} a_{r-i}.
\]  
  
   Let $t$ be a formal variable.
  By expanding the determinant $d_r$ along the top row, we obtain $d_r
  = a_1 d_{r-1} - a_2 d_{r-2} + \dots + (-1)^{r-1} a_r$, which is
  equivalent to the identity of formal power series
  \begin{equation} \label{E:powerd}
     \Big( \sum_{i= 0}^{n+k} a_it^i \Big)
     \Big( \sum_{i \geq 0} (-1)^i d_i t^i \Big) = 1.
  \end{equation}
  The definition of $b_r$ implies that
  \begin{equation} \label{E:powerb}
   \sum_{i= 0}^{n+k} (-1)^i b_i t^{2i} = 
   \Big( \sum_{i= 0}^{n+k} a_it^i \Big)
   \Big( \sum_{i= 0}^{n+k} (-1)^i a_i t^i \Big) \,.
  \end{equation}
 
  Consider the ideal $I = (d_{m+1}, \dots, d_{n+k}, b_{k+1}, \dots,
  b_n) \subset R$.  We claim that $\phi(I) = 0$.  Indeed, since $\s_i
  = c_i(\cQ)$ for each $i$, the Whitney sum formula $c_t(\cS) c_t(\cQ)
  = 1$ and equation (\ref{E:powerd}) imply that $\phi(d_r) =
  c_r(\cS^*)$, so $\phi(d_r) = 0$ for $r > m$.  Notice that the
  symplectic form gives a pairing $\cS \otimes \cQ \to \cO$, which in
  turn produces an injection $\cS \hookrightarrow \cQ^*$.  It
  therefore follows from (\ref{E:powerb}) that
\[
(-1)^r \phi(b_r) = c_{2r}(\cQ \oplus \cQ^*) = 
c_{2r}(\cV_{\IG}/\cS \oplus \cQ^*) = c_{2r}(\cQ^*/\cS) = 0 \ \ \text{for} 
\ r > k,
\] 
which proves the claim.  Using the induced map $\phi : R/I \to
\HH^*(\IG)$ we are left with checking (i) and (ii) of
Lemma~\ref{preslemma}.  Property (i) follows because $\deg(d_r) = r$,
$\deg(b_r) = 2r$, and $\rank \HH^*(\IG) = \# \cP(k,n) = 2^m
\binom{n}{k}$.
  
  To prove (ii) it is enough to show that $d_r \in I$ for $n+k < r
  \leq 2n$, since this implies that $R/I$ is a quotient of
  $R/(d_{m+1},\dots,d_{2n})$, which is a free $\Z$-module of finite
  rank by Lemma~\ref{typeApres}.  It follows from (\ref{E:powerd}) and
  (\ref{E:powerb}) that 
\[
   \sum_{i= 0}^{n+k} a_it^i = \Big(\sum_{i= 0}^{n+k} (-1)^i b_i t^{2i}\Big)
   \Big(\sum_{i \geq 0} d_it^i\Big),
\]
which implies that
\[
\sum_{i=0}^{n+k}a_it^i \equiv
\Big(\sum_{i=0}^k(-1)^ib_it^{2i}+\sum_{i=n+1}^{n+k}(-1)^ib_it^{2i}\Big)
\Big(\sum_{i=0}^m d_it^i + \sum_{i\geq n+k+1} d_it^i \Big)
\]
  modulo the homogeneous ideal $I \subset R$.  By equating terms of
  equal degrees in this congruence, we obtain $d_{n+k+1} \equiv
  d_{n+k+2} \equiv \dots \equiv d_{2n} \equiv 0$ modulo $I$, as
  required.
\end{proof}

\subsection{Gromov-Witten invariants}
\label{iggwis}

A rational map of degree $d$ to $\IG$ is a morphism of varieties $f :
\bP^1 \to \IG$ such that
\[ \int_{\IG} f_*[\bP^1] \cdot \sigma_1 = d \,, \]
i.e.\ $d$ is the number of points in $f^{-1}(X_1(E_\bull))$ when the
isotropic flag $E_\bull$ is in general position.  All the
Gromov-Witten invariants considered in this paper are three-point and
genus zero.  Given a degree $d \gequ 0$ and $k$-strict partitions
$\lambda, \mu, \nu$ such that $|\lambda| + |\mu| + |\nu| = \dim(\IG) +
d(n+1+k)$, we define the Gromov-Witten invariant $\gw{\sigma_\lambda,
\sigma_\mu, \sigma_\nu}{d}$ to be the number of rational maps $f :
\bP^1 \to \IG$ of degree $d$ such that $f(0) \in X_\lambda(E_\bull)$,
$f(1) \in X_\mu(F_\bull)$, and $f(\infty) \in X_\nu(G_\bull)$, for
given isotropic flags $E_\bull$, $F_\bull$, and $G_\bull$ in general
position. As in \cite{BKT}, we will
show that these morphisms 
are maps to a
Lagrangian Grassmannian $\LG(d,2d)$ contained in $\IG$ whose image
curves pass through three general points of $\LG$. 

Following \cite{Buch}, we define the {\em kernel\/} of a rational map
$f : \bP^1 \to \IG$ as the intersection of all the subspaces $\Sigma_t
= f(t)$ in the image of $f$, and the {\em span\/} of $f$ as the linear
span of these subspaces in $V$:
\[ \Ker(f) = \bigcap \Sigma_t \subset V  \text{\ \ \ and\ \ \ }
   \Span(f) = \sum \Sigma_t \subset V \,.
\]
If $f$ has degree $d$ then $\dim \Ker(f) \gequ m-d$ and $\dim
\Span(f) \lequ m+d$ (see \cite[Lemma 1]{Buch}).  Notice that $\Ker(f)
\subset \Span(f) \subset \Ker(f)^\perp$.

For any integer $d \lequ m$, let $Y_d$ be the variety parametrizing
pairs $(A,B)$ of subspaces of $V$ such that $A \subset B \subset
A^\perp$, $\dim A = m-d$, and $\dim B = m+d$.  Since $Y_d$ is a
$\G(2d,2k+2d)$-bundle over $\IG(m-d,2n)$, we compute that $\dim(Y_d) =
\dim(\IG) + d(n+1+k) - 3d(d+1)/2$.  The algebraic group $\Sp_{2n}$
acts on the ambient vector space $V$, and hence also on $Y_d$, but
this action is not transitive. Let $T_d$ be the variety of triples 
$(A,\Sigma,B)$ of subspaces of $V$
such that $(A,B) \in Y_d$, $\Sigma \in \IG$, and $A \subset \Sigma
\subset B$.  Let $\pi : T_d \to Y_d$ be the projection.

For each $k$-strict partition $\lambda$, define a subvariety
$Y_\lambda \subset Y_d$ by the prescription
\[ Y_\lambda(E_\bull) = \{ (A,B) \in Y_d \mid \exists\, \Sigma \in
   X_\lambda(E_\bull) : A \subset \Sigma \subset B \} \,. 
\]
Let $T_\lambda(E_\bull) = \{(A,\Sigma,B) \in T_d \mid \Sigma \in
X_\lambda(E_\bull) \}$.  Then we have $\pi(T_\lambda(E_\bull)) =
Y_\lambda(E_\bull)$.  Since the general fibers of $\pi$ are isomorphic
to Lagrangian Grassmannians $\LG(d,2d)$, it follows from this that the
codimension of $Y_\lambda(E_\bull)$ in $Y_d$ is at least $|\lambda| -
d(d+1)/2$.  The following lemma implies that this codimension is
obtained if and only if $\lambda$ contains the staircase partition
$\rho_d = (d,d-1,\dots,2,1)$ with $d$ rows.

\begin{lemma} \label{L:fiberC}
  The restricted projection $\pi : T_\lambda(E_\bull) \to
  Y_\lambda(E_\bull)$ is generically one to one when $\rho_d \subset
  \lambda$, and has fibers of positive dimension when $\rho_d
  \not\subset \lambda$.
\end{lemma}
\begin{proof}
  We can assume that $E_r = \Span\{e_1,\dots,e_r\}$ where
  $\{e_1,\dots,e_{2n}\}$ is a standard symplectic basis for $V$, i.e.\
  $(e_i,e_j) = 0$ when $i+j \neq 2n+1$.  We will drop the flag
  $E_\bull$ from the notation and write simply $X_\la$, $Y_\lambda$,
  and $T_\lambda$.  Set $p_j = p_j(\lambda)$ for $1 \leq j \leq m$,
  $p_0=0$, $p_{m+1}=2n+1$, and $\#_j = \#\{i<j \mid p_i+p_j > 2n+1
  \}$.  Then we have $p_j = n+k+1-\lambda_j+\#_j$ by Proposition
  \ref{P:trnslC}.  We also define 
\[
s_j = \max(m+d+p_j-2n-j,0) = \max(d+1-j-\lambda_j+\#_j,0).  
\]
  For $(A,B)$ in $Y_\la$, we set
  $A_j=A\cap E_{p_j}$ and $B_j=B\cap E_{p_j}$.  By definition there
  exists $\Sigma$ in $X_\la$, with $A\subset \Sigma\subset B$.  We set
  $\Sigma_j=\Sigma\cap E_{p_j}$, so that $\dim(\Sigma_j)\ge j$.  Now,
\begin{gather}
\dim(A_j)\ge j-d,
\label{Cbd1} \\
\dim(B_j)\ge j+s_j,
\label{Cbd2}
\end{gather}
  since $A$ is of codimension $d$ in $\Sigma$, and $B_j$ is the
  intersection of spaces of dimension $m+d$ and $p_j$, while also
  $\Sigma_j\subset B_j$.  For $j$ such that equality holds in
  \eqref{Cbd2} we have $j\le \dim(\Sigma_{j-\varepsilon}^\perp\cap
  B_j)= j+s_j-\dim(\Sigma_{j-\varepsilon})
  +\dim(\Sigma_{j-\varepsilon}\cap B_j^\perp)$ for
  $\varepsilon\in\{0,1\}$, since $(\Sigma^\perp_{j-\varepsilon}\cap
  B_j)^\perp = \Sigma_{j-\varepsilon}+B_j^\perp$. Therefore, equality
  in \eqref{Cbd2} implies
\begin{gather}
\dim(B_j\cap B_j^\perp)
\ge j-s_j,
\label{Cbd3} \\
\dim(B_{j-1}\cap B_j^\perp)
\ge j-1-s_j.
\label{Cbd4}
\end{gather}

Define $U_\la$ to be the open subset of points $(A,B)$ in $Y_\lambda$
satisfying equality in \eqref{Cbd1} for $j\ge d$ and in \eqref{Cbd2},
\eqref{Cbd3}, \eqref{Cbd4} for $1\le j\le d$.  We will show that if
$\rho_d \subset \lambda$, then $U_\lambda \neq \emptyset$, and for
$(A,B) \in U_\lambda$ there is a unique point $\Sigma\in X_\la$ with
$A\subset\Sigma \subset B$.

For the nonemptiness, we have from $\la_j\ge d+1-j$ that $s_j \le
\#_j$.  For $j>d$ and $s_j>0$ we also have $s_j \leq d+1-j + \#_j \leq
\#\{ i \leq d \mid p_i+p_j > 2n+1\}$.  We can therefore choose a
permutation $\pi\in {\mathfrak S}_d$ such that for all $1\le j\le m+1$
and $i\le s_j$,
\[
\pi(i)<j\quad\text{and}\quad
p_{\pi(i)}+p_j\ge 2n+2+s_j-i \,.
\]
Now we define $A=\Span(e_{p_{d+1}},\ldots,e_{p_m})$ and
$B=\Span(e_{p_1},\ldots,e_{p_m},u_1,\ldots,u_d)$, where
$u_i=e_{2n+1-p_{\pi(i)}}+e_{p_j-1-s_j+i}$ for $s_{j-1}<i\le s_j$.
Equality holds in \eqref{Cbd1} for $j\ge d$.  Notice that
$2n+1-p_{\pi(i)}\le p_j-1-s_j+i$ and $p_{j-1}<p_j-1-s_j+i$ when
$s_{j-1}<i\le s_j$.  It follows that $B_j =
\Span(e_{p_1},\dots,e_{p_j},u_1,\dots,u_{s_j})$, so equality in
\eqref{Cbd2} holds for all $j$.  By the pairings
$(e_{p_{\pi(i)}},u_i)\ne 0$ and $(e_{p_{\pi(i)}},u_{i'})=0$ for
$i'>i$, we have that the restriction of the bilinear form to the span
of $e_{p_{\pi(1)}}$, $\ldots$, $e_{p_{\pi(s_j)}}$, $u_1$, $\ldots$,
$u_{s_j}$ is nondegenerate, for any $j$.  So $\dim(B_j\cap
B_j^\perp)\le \dim(B_j)-2s_j=j-s_j$.  Since $B_{j-1}$ contains the
first $s_j+s_{j-1}$ of these vectors, $\dim(B_{j-1}\cap B_j^\perp)\le
\dim(B_{j-1})-(s_j+s_{j-1})=j-1-s_j$.  Equality in \eqref{Cbd3} and
\eqref{Cbd4} is established.

Let $(A,B) \in U_\lambda$ and $0 \leq j \leq d$.  We claim that there
exists a unique isotropic subspace $\Sigma_j \subset B_j$ of dimension
$j$, such that 
\begin{equation} \label{CSigmajineq}
  \dim(\Sigma_j\cap B_i) \ge i
\end{equation}
for $1\le i\le j$.  We prove this by induction on $j$, the base case
$j=0$ being clear.  The equalities \eqref{Cbd2} and \eqref{Cbd3} imply
that a maximal isotropic subspace of $B_j$ has dimension $j$.  If
$\Sigma_j$ satisfies \eqref{CSigmajineq} then the induction hypothesis
implies that $\Sigma_j \cap B_{j-1} = \Sigma_{j-1}$, so it is enough
to show that $\Sigma_{j-1}$ extends to a unique isotropic subspace of
dimension $j$ in $B_j$, i.e.\ $\dim(\Sigma_{j-1}^\perp \cap B_j) = j$.
This follows from equality \eqref{Cbd4} because
$\dim(\Sigma_{j-1}^\perp \cap B_j) = s_j+1+\dim(\Sigma_{j-1}\cap
B_j^\perp) \leq j$.  Now $\Sigma_d \cap A = 0$ by the equality
\eqref{Cbd1} for $j=d$, hence, if we set $\Sigma = \Sigma_d + A$, then
$\Sigma \in X_\lambda$.  Conversely, for any $\Sigma \in X_\lambda$
with $A \subset \Sigma \subset B$ we have $\dim(\Sigma \cap B_i) \geq
i$ for all $i$, therefore $\Sigma \cap B_d = \Sigma_d$ and $\Sigma =
\Sigma_d + A$.

We finally show that when $\rho_d \not\subset \lambda$ we have
$\dim(Y_\lambda) < \dim(T_\lambda)$.  Let $X_\lambda^\circ \subset
X_\lambda$ be the Schubert cell.  It is enough to show that if $(A,B)
\in Y_\lambda$ satisfies $A_d = 0$, then there exists $\Sigma' \in
X_\lambda \smallsetminus X_\lambda^\circ$ with $A \subset \Sigma'
\subset B$.  If this is false, then we can choose $\Sigma \in
X^\circ_\lambda$ with $A \subset \Sigma \subset B$.  Set $\Sigma_i =
\Sigma \cap E_{p_i}$ for $1 \leq i \leq m$.  For some $j$ we have
$\lambda_j \leq d-j$ which implies that $s_j > \#_j$, hence
$\dim(\Sigma_{j-1}^\perp \cap E_{p_j}) \geq p_j - \#_j > p_j-s_j$.
Using \eqref{Cbd2} we obtain $\dim(\Sigma_{j-1}^\perp \cap B_j) > j$,
so there exists a $j$-dimensional isotropic extension $\Sigma'_j$ of
$\Sigma_{j-1}$ contained in $B \cap E_{p_j-1}$.  For $i=j+1, \dots, d$
we now choose an $i$-dimensional isotropic extension $\Sigma'_i$ of
$\Sigma'_{i-1}$ contained in $B_i$.  This is possible because $B_i$
already contains the $i$-dimensional isotropic subspace $\Sigma_i$.
The subspace $\Sigma' = \Sigma'_d + A$ then satisfies $\Sigma' \in
X_\lambda \smallsetminus X_\lambda^\circ$ and $A \subset \Sigma'
\subset B$.
\end{proof}

The next proposition is used to
reconstruct rational maps from their kernel and span.  Special cases
are equivalent to computing the Gromov-Witten invariant counting
curves through three general points in maximal isotropic Grassmannians,
which was first done in \cite[Cor.~8]{KTlg} and \cite[Cor.~7]{KTog}.
Although these results suffice to handle the case of $\IG$, for the
non maximal orthogonal Grassmannians we will need to work more generally.

Consider a vector space $W\cong \C^{2d}$ with an 
arbitrary bilinear form $(\ ,\, )$, and let
$\ov{\G}(d,W)$ denote the Grassmannian of dimension $d$ subspaces of $W$
which are isotropic with respect to $(\ ,\, )$.
There is a closed embedding $\iota$ of
$\ov{\G}(d,W)$ into the type A Grassmannian $\G(d,W)$, and we say
that a morphism $f:\bP^1 \to \ov{\G}(d,W)$ has degree $d$ if the
composite $\iota\circ f$ has degree $d$. We say that two points $U_1$,
$U_2$ of $\ov{\G}(d,W)$ are in general position if the intersection
$U_1\cap U_2$ is trivial.

\begin{prop}
\label{Gd2d}
Let $U_1$, $U_2$, and $U_3$ be three points of $\ov{\G}(d,W)$ which 
are pairwise in general position.  Then there is a unique morphism 
$f:\bP^1\to\ov{\G}(d,W)$ of degree $d$ such that $f(0)=U_1$, $f(1)=
U_2$, and $f(\infty)=U_3$. 
\end{prop}
\begin{proof}
  For any basis $\{v_1,\ldots,v_d\}$ of $U_2$, write $v_i=u_i+w_i$
  with $u_i\in U_1$ and $w_i\in U_3$. As in the proof of
  \cite[Prop.~1]{BKT}, one shows that the map $f(s:t) =
  \text{Span}(su_1+tw_1,\ldots, su_d+tw_d\}$ is the unique one
  satisfying the condition in the proposition. Here $(s:t)$ are the
  homogeneous coordinates on $\bP^1$. Since any quadratic form
  $q(s,t)$ which vanishes at $(1,0)$, $(1,1)$, and $(0,1)$ must vanish
  identically, the subspaces of $W$ which lie in the image of
  $f$ must all be isotropic.
\end{proof}

\begin{example}
  We have $\gw{\sigma_{\rho_m}, [\point], [\point]}{m} = 1$ for all
  isotropic Grassmannians $\IG = \IG(m,2n)$.  In fact, it follows from
  Lemma~\ref{L:fiberC} that $Y_{\rho_m}$ has codimension zero in
  $Y_m$, so $Y_{\rho_m} = Y_m$.  Let $U,V \in \IG$ be general points.
  By setting $(A,B)=(0,U\oplus V) \in Y_m$, the lemma implies that
  exactly one point $\Sigma \in X_{\rho_m}$ satisfies $\Sigma \subset
  U \oplus V$.  Note that the span of any curve counted by the
  Gromov-Witten invariant must be contained in $U\oplus V$, and
  therefore the curve is itself contained in $\LG(m,U\oplus V) \subset
  \IG$.  The claim now follows
  from Proposition \ref{Gd2d}.
\end{example}

The following theorem generalizes \cite[Thm.~2]{BKT}.  The vanishing
statement in part (d) was proved for Lagrangian Grassmannians in
\cite{KTlg}.

\begin{thm} \label{T:qclasC}
  Let $d \gequ 0$ and choose $\lambda,\mu,\nu \in \cP(k,n)$ such that
  $|\lambda|+|\mu|+|\nu| = \dim(\IG) + d(n+1+k)$.  Let $X_\lambda$,
  $X_\mu$, and $X_\nu$ be Schubert varieties of $\IG(m,2n)$ in general
  position, and let $Y_\lambda,Y_\mu,Y_\nu$ be the associated
  subvarieties of $Y_d$.
  
  \smallskip \noin {\rm (a)} The subvarieties $Y_\lambda$, $Y_\mu$,
  and $Y_\nu$ intersect transversally in $Y_d$, and $Y_\lambda \cap
  Y_\mu \cap Y_\nu$ is finite.  For each point $(A,B) \in Y_\lambda
  \cap Y_\mu \cap Y_\nu$ we have $A = B \cap B^\perp$.
  
  \smallskip \noin {\rm (b)} The assignment $f \mapsto
  (\Ker(f),\Span(f))$ gives a bijection of the set of rational maps
  $f:\bP^1\to\IG$ of degree $d$ such that $f(0)\in X_\lambda$,
  $f(1)\in X_\mu$, $f(\infty)\in X_\nu$, with the points of the
  intersection $Y_\lambda \cap Y_\mu \cap Y_\nu$ in $Y_d$.
  
  \smallskip \noin {\rm (c)} We have
  $\dis \gw{\sigma_\lambda,\sigma_\mu,\sigma_\nu}{d} = \int_{Y_d} [Y_\lambda]
  \cdot [Y_\mu] \cdot [Y_\nu]$.
  
  \smallskip \noin {\rm (d)} If $\lambda$ does not contain the
  staircase partition $\rho_d$, then $\gw{\sigma_\lambda, \sigma_\mu,
    \sigma_\nu}{d} = 0$. 
\end{thm}
\begin{proof}
  For integers $0 \lequ e_1,e_2 \lequ d$ and $r \gequ 0$ we let
  $Y^r_{e_1,e_2}$ be the variety of pairs $(A,B)$ such that $A \subset
  B \subset A^\perp \subset V$, $\dim(A) = m-e_1$, $\dim(B) =
  m+e_2$, and $\dim(B \cap B^\perp) = m-e_1+r$.  These varieties
  have a transitive action of $\Sp_{2n}$, and $Y_d$ is the union of
  the varieties $Y^r_{d,d}$ for even numbers $r \lequ \min(2d,2k)$.  In
  general, $Y^r_{e_1,e_2}$ is empty unless $r \lequ
  \min(e_1+e_2,2k+e_1-e_2)$ and $e_1+e_2-r$ is even.  We also set
\[ Y^r_\lambda = \{ (A,B) \in Y^r_{e_1,e_2} \mid \exists\, \Sigma \in 
   X_\lambda : A \subset \Sigma \subset B \} \,.
\]
Our first goal is to prove that $Y^r_\lambda \cap Y^r_\mu \cap Y^r_\nu$ is
empty unless $e_1=e_2=d$ and $r=0$, in which case the intersection is
a finite set of points.

The smooth map $Y^r_{e_1,e_2} \to \IF(m-e_1,m-e_1+r;2n)$ given by
$(A,B) \mapsto (A,B\cap B^\perp)$ has fibers isomorphic to open
subsets of the type A Grassmannian $\G(e_1+e_2-r,2k+2e_1-2r)$, so we
obtain
\[ \dim(Y^r_{e_1,e_2}) = \frac{1}{2}
   (n^2 + 2nk -3k^2 + m + 2 m e_1 - e_1^2+4e_2k-2e_2^2-e_1+r-r^2) \,.
\]

Define the variety $T^r = \{ (A,\Sigma,B) \mid (A,B) \in
Y^r_{e_1,e_2}, \Sigma \in \IG, A \subset \Sigma \subset B \}$, and
let $\pi_1 : T^r \to \IG$ and $\pi_2 : T^r \to Y^r_{e_1,e_2}$ be the
projections.  Then $Y^r_\lambda = \pi_2(\pi_1^{-1}(X_\lambda))$.
Given a point $(A,B) \in Y^r_{e_1,e_2}$, the fiber $\pi_2^{-1}((A,B))$
is the union of the varieties
\[ Q_s = \{ \Sigma \in \IG \mid A \subset \Sigma \subset B,\ 
   \dim(\Sigma\cap B \cap B^\perp) = m-e_1+s \}
\]
for all integers $s$.  Notice that the dimension of an isotropic
subspace of $B$ is at most equal to $m+(r-e_1+e_2)/2$, and if
$\Sigma \in Q_s$ then $\Sigma + (B\cap B^\perp)$ is such a subspace of
dimension $m+r-s$.  This implies that $Q_s$ is empty unless $2s \gequ
r+e_1-e_2$.  Since the map $Q_s \to \IG(e_1-s, B/(B\cap B^\perp))$
given by $\Sigma \mapsto (\Sigma + (B\cap B^\perp))/(B\cap B^\perp)$
has fibers isomorphic to open subsets of $\G(e_1,e_1+r-s)$, it follows
that when $Q_s$ is not empty, we have
\[ \dim Q_s = \frac{1}{2}
   (2e_1 e_2 -e_1^2 + e_1 + 2e_1 s - 2e_2 s  + 2rs - 3s^2 - s) \,.
\]

Choose $s$ such that $Q_s$ has the same dimension as the fibers of
$\pi_2$.  Then the codimension of $Y^r_\lambda$ in $Y^r_{e_1,e_2}$ is
at least $|\lambda| - \dim Q_s$.  It follows that
$\codim(Y^r_\lambda)+\codim(Y^r_\mu)+\codim(Y^r_\nu) -
\dim(Y^r_{e_1,e_2})$ is greater than or equal to the number
$\Delta(e_1,e_2,r,s) := \dim \IG + d(n+1+k) - \dim(Y^r_{e_1,e_2}) -
3\dim Q_s$.  A computation shows that $\Delta(e_1,e_2,r,s)$ is equal
to
\[ (r-3s)(r-3s-1)/2 + (d-e_1) + (n+k)(d-e_2) + 
   (m+e_2-2e_1+3s)(e_2-e_1) \,. 
\]

Suppose $Y^r_\lambda \cap Y^r_\mu \cap Y^r_\nu \neq \emptyset$.  Since the
action of $\Sp_{2n}$ on $Y^r_{e_1,e_2}$ is transitive, it follows from
the Kleiman-Bertini theorem that $\Delta(e_1,e_2,r,s) \lequ 0$.  If
$e_2 \gequ e_1$ then this immediately implies that $e_1=e_2=d$, so we
may assume that $e_2 < e_1$.  If $(A,B) \in Y^r_\lambda \cap Y^r_\mu
\cap Y^r_\nu$ then we can find a subspace $B' \subset A^\perp$
containing $B$, such that $\dim B' = m+e_1$.  Since $(A,B')$ must be
a point in the intersection of modified Schubert varieties in some
variety $Y^{r'}_{e_1,e_1}$, it follows that $\Delta(e_1,e_1,r',s') \leq
0$, so $e_1=d$.  Still assuming $e_2 < e_1$, we can take a subspace
$B'' \subset A^\perp$ containing $B$, such that $\dim B'' = m+d-1$.
Since $(A,B'')$ lies in the intersection of modified Schubert
varieties in some space $Y^{r''}_{d,d-1}$, it follows that
\[ \Delta(d,d-1,r'',s'') = (r''-3s'')(r''-3s''-1)/2 + 2k + 1 - 3s'' + d \]
is non-positive, and since $r'' \lequ 2k+1$, this number is also
greater than or equal to $(r''-3s'')(r''-3s''+1)/2 + d$.  But this number
is positive, which gives a contradiction.  We conclude that
$e_1=e_2=d$, so $\Delta(d,d,r,s) = (r-3s)(r-3s-1)/2 \lequ 0$.  This is
possible only if $r=3s$ or $r=3s+1$.  Since $2s \gequ r+e_1-e_2 = r$, we
finally obtain $r=s=0$ as required.

It is now clear that part (a) is true.  In fact, the intersection
$Y_\lambda \cap Y_\mu \cap Y_\nu$ is contained in the open
$\Sp_{2n}$-orbit $Y^0_{d,d}$ of $Y_d$, so this intersection is
transverse by Kleiman and Bertini's theorem.  Furthermore, since
$\codim(Y_\lambda) \gequ |\lambda| - \dim Q_0 = |\lambda| - d(d+1)/2$,
we deduce that the intersection is finite.

Let $f : \bP^1 \to \IG$ be a rational map of degree $d$ such that
$f(0) \in X_\lambda$, $f(1) \in X_\mu$, and $f(\infty) \in X_\nu$.  Then
$(\Ker(f),\Span(f))$ must be a point of an intersection $Y^r_\lambda
\cap Y^r_\mu \cap Y^r_\nu$ in some space $Y^r_{e_1,e_2}$, necessarily with
$e_1=e_2=d$, so we must have $\dim \Ker(f) = m-d$ and $\dim \Span(f)
= m+d$.  This shows that the map $f \mapsto (\Ker(f),\Span(f))$ of
part (b) is well defined.  In particular we have $d \lequ m$.  On the
other hand, let $(A_0,B_0) \in Y_\lambda \cap Y_\mu \cap Y_\nu \subset
Y_d$ be any point.  To establish (b), we must show that there is a
unique rational map $f : \bP^1 \to \IG$ of degree $d$ such that
$f(0)\in X_\lambda$, $f(1) \in X_\mu$, $f(\infty) \in X_\nu$, and
$(\Ker(f),\Span(f)) = (A_0,B_0)$.

Setting $W=B_0/A_0$ we can identify $\ov{\G}(d,W)$ with $\{ \Sigma \in \IG \mid
A_0 \subset \Sigma \subset B_0 \}$. Since $B_0 \cap B_0^\perp = A_0$ by
part (a), we have $\ov{\G}(d,W)\cong \LG(d,2d)$. Let $P\in X_\lambda \cap
\ov{\G}(d,W)$, $Q\in X_\mu \cap \ov{\G}(d,W)$, and $R\in X_\nu \cap \ov{\G}(d,W)$.
We claim that $P \cap Q = A_0$.
Otherwise there exists $A_0 \subset A'_0 \subset P\cap Q$ with $\dim
A'_0 = m-d+1$.  Now define
\[ Y' = \{ (A,A',B) \mid (A,B) \in Y^0_{d,d},\ A \subset A' \subset B,\ 
   \dim A' = m-d+1 \} \,.
\]
Notice that the action of $\Sp_{2n}$ on this variety is transitive,
and that $\dim Y' = \dim Y_d + 2d-1$.  We also set $Y'_\lambda = \{
(A,A',B) \in Y' \mid \exists\, \Sigma \in X_\lambda: A' \subset \Sigma
\subset B \}$.  Our assumptions show that $(A_0,A'_0,B_0) \in
Y'_\lambda \cap Y'_\mu \cap \pi^{-1}(Y_\nu)$ where $\pi : Y' \to Y_d$
is the projection.  Set $T' = \{(A,A',\Sigma,B) \mid (A,A',B) \in Y',
\Sigma\in \IG, A' \subset \Sigma \subset B \}$ and let $\pi'_1 : T' \to
\IG$ and $\pi'_2 : T' \to Y'$ be the projections.  The fibers of
$\pi'_2$ are isomorphic to $\LG(d-1,2d-2)$, so they have dimension
$d(d-1)/2$.  Since $Y'_\lambda = \pi'_2({\pi'_1}^{-1}(X_\lambda))$, this
implies that $\codim(Y'_\lambda) \gequ |\lambda| - d(d-1)/2$.  It
follows that
\begin{multline*}
  \codim(Y'_\lambda) + \codim(Y'_\mu) + \codim(\pi^{-1}(Y_\nu)) \gequ \\
  |\lambda| + |\mu| + |\nu| - 2d(d-1)/2 - d(d+1)/2 
  = \dim(Y_d) + 2d > \dim(Y') \,.
\end{multline*}
Since this implies that $Y'_\lambda \cap Y'_\mu \cap \pi^{-1}(Y_\nu) =
\emptyset$, we conclude that $P \cap Q = A_0$ as claimed.

Notice that the intersection $X_\lambda \cap \ov{\G}(d,W)$ must have
dimension zero, since otherwise one could choose the point $P \in
X_\lambda \cap \ov{\G}(d,W)$ such that $P \cap Q \varsupsetneq A_0$.
By Lemma~\ref{L:fiberC} this implies that $\rho_d \subset \lambda$
(which proves (d)), and $X_\lambda \cap \ov{\G}(d,W)$ must furthermore
be a single point.  We conclude that $X_\lambda \cap \ov{\G}(d,W) =
\{P\}$, $X_\mu \cap \ov{\G}(d,W) = \{Q\}$, $X_\nu \cap \ov{\G}(d,W) =
\{R\}$, and $P\cap Q = Q\cap R= R\cap P = A_0$.  Now by Proposition
\ref{Gd2d} there is a unique rational map $f : \bP^1 \to Z$ of degree
$d$ such that $f(0) \in X_\lambda$, $f(1) \in X_\mu$, and $f(\infty)
\in X_\nu$.  This proves (b), and (a) and (b) together imply (c).
\end{proof}

Set $X^+ = \IG(m+1, 2n+2)$.  The following proposition was proved in
\cite[Prop.~4]{KTlg} for Lagrangian Grassmannians.  As in loc.\ cit.,
the proof is based on an explicit correspondence between lines in $X =
\IG(m,2n)$ and points in $X^+$, but the underlying geometry is more
challenging in the general case.

\begin{prop} \label{P:gwdegone}
  Let $\lambda,\mu,\nu$ be $k$-strict partitions such that
  $|\lambda|+|\mu|+|\nu| = \dim X + (n+k+1)$ and
  $\ell(\lambda)+\ell(\mu)+\ell(\nu) \lequ 2m+1$.  Then
\[ \gw{\sigma_\lambda, \sigma_\mu, \sigma_\nu}{1} = 
   \frac{1}{2} \int_{\IG(m+1,2n+2)} 
   [X^+_\lambda] \cdot [X^+_\mu] \cdot [X^+_\nu] \,.
\]
\end{prop}
\begin{proof}
  Let $E_\bull, F_\bull, G_\bull \subset V = \C^{2n}$ be isotropic
  flags in general position.  Let $H = \C^2$ be a two-dimensional
  symplectic vector space, and choose generic elements $e,f,g \in H$.
  Set $V^+ = V \oplus H$ and let $E^+_\bull, F^+_\bull, G^+_\bull
  \subset V^+$ be the augmented flags given by $E^+_p = E_{p-1} \oplus
  \C e$, $F^+_p = F_{p-1} \oplus \C f$, and $G^+_p = G_{p-1} \oplus \C
  g$.  We will prove that the assignment $\Sigma^+ \mapsto (\Sigma^+
  \cap V, (\Sigma^+ + H)\cap V)$ gives a well defined 2-1 map from
  $X^+_\lambda(E^+_\bull) \cap X^+_\mu(F^+_\bull) \cap
  X^+_\nu(G^+_\bull)$ onto the intersection $Y_\lambda(E_\bull) \cap
  Y_\mu(F_\bull) \cap Y_\nu(G_\bull) \subset Y_1$.  Notice that
  $X^+_\lambda(E^+_\bull) = \{ \Sigma^+ \in X^+ \mid \dim(\Sigma^+
  \cap (E_{p_j(\lambda)}+\C e)) \gequ j, \ \forall \, 1 \lequ j \leq
  \ell(\lambda) \}$.
  
  The group $G = \Sp(V) \times \SL(H)$ acts on $X^+$ and our choices
  imply that the varieties $X^+_\lambda, X^+_\mu, X^+_\nu$ are in
  general position for this action.  It is easily seen that the
  $G$-action on $X^+$ has the following four orbits.
  
$O_1 = \IG(m+1,V)$

$O_2 = \IG(m,V) \times \IG(1,H)$

$O_3 = \{ \Sigma^+ \in X^+ \mid \dim(\Sigma^+ \cap V) = m, \Sigma^+ \cap H =
0 \}$

$O_4 = \{ \Sigma^+ \in X^+ \mid \dim(\Sigma^+ \cap V) = m-1 \}$

We start by showing that the intersection $X^+_\lambda \cap X^+_\mu
\cap X^+_\nu$ is contained in the open orbit $O_4$.  Notice at first
that $X^+_\lambda(E_\bull) \cap \IG(m+1,V)$ is the Schubert variety in
$\IG(m+1,V)$ given by the partition $\overline \lambda =
(\lambda_1-1,\dots,\lambda_{\ell(\lambda)}-1)$ obtained by deleting
the first column of $\lambda$.  Since $|\overline\lambda| +
|\overline\mu| + |\overline\nu| > |\lambda|+|\mu|+|\nu| - 2m-2 = \dim
\IG(m+1,V)$, it follows that $X^+_\lambda \cap X^+_\mu \cap X^+_\nu
\cap \IG(m+1,V) = \emptyset$.

Now assume $\Sigma^+ \in X^+_\lambda \cap X^+_\mu \cap X^+_\nu \cap
O_2$.  Then $\Sigma^+ = \Sigma \oplus \C h$ where $\Sigma \in X$ and
$0 \neq h \in H$.  We may assume that $h \not \in \C f$ and $h \not\in
\C g$, which implies that $\Sigma \in X_\mu(F_\bull) \cap
X_\nu(G_\bull)$.  Define the $k$-strict partition $\eta$ by $\eta_j =
\lambda_{j+1}$ for $j < \lambda_1-2k$; $\eta_j = \lambda_{j+1}-1$ for
$\lambda_1-2k \leq j < \ell(\lambda)$; and $\eta_j = 0$ for $j \geq
\ell(\lambda)$.  Since $\lambda_1+\lambda_j \lequ 2k+j-1$ implies that
$j > \lambda_1-2k$, it follows that $p_{j+1}(\lambda) \lequ p_j(\eta)$
for each $j < \ell(\lambda)$.  Since $\dim(\Sigma \cap
E_{p_{j+1}(\lambda)}) \gequ j$ for all $j$, this implies that $\Sigma
\in X_\eta(E_\bull)$.  But $|\eta|+|\mu|+|\nu| \gequ
|\lambda|+|\mu|+|\nu|-n-k > \dim X$, so $X_\eta \cap X_\mu \cap X_\nu
= \emptyset$, a contradiction.

We finally check that no point of $X^+_\lambda \cap X^+_\mu \cap
X^+_\nu$ is contained in $O_3$.  We have a smooth map $\pi : O_3 \to
X$ given by $\pi(\Sigma^+) = \Sigma^+ \cap V$.  The fiber over a point
$\Sigma \in X$ is an open subset of $\bP(H \oplus
\Sigma^\perp/\Sigma)$, so $\dim O_3 = \dim X + 2k+1 = \dim X^+ - m$.
Now suppose $\Sigma^+ \in X^+_\lambda \cap X^+_\mu \cap X^+_\nu \cap
O_3$.  If $\Sigma^+ \not\subset V\oplus \C e$, then the dimension of
$\Sigma^+ \cap (V \oplus \C e)$ is at most $m$.  But $\Sigma^+ \cap V$
has dimension $m$, so we must have $\Sigma^+ \cap (V \oplus \C e) =
\Sigma^+ \cap V$.  In particular we obtain $\Sigma^+ \cap (E_i \oplus
\C e) = (\Sigma^+ \cap V) \cap E_i$ for each $i$, so $\Sigma^+ \cap V
\in X_\lambda$ and $\Sigma^+ \in \pi^{-1}(X_\lambda)$.  Now notice
that $\Sigma^+$ cannot be contained in both $V \oplus \C e$ and $V
\oplus \C f$, so by permuting $\lambda, \mu, \nu$, we may assume that
$\Sigma^+ \not\subset V \oplus \C f$ and $\Sigma^+ \not\subset V
\oplus \C g$.  This implies that $\Sigma^+ \in \pi^{-1}(X_\mu) \cap
\pi^{-1}(X_\nu)$.  Since $\pi^{-1}(X_\lambda \cap X_\mu \cap X_\nu) =
\emptyset$, we then deduce that $\Sigma^+ \subset V \oplus \C e$.
However, all components of the variety $Z = \{ \Sigma^+ \in
X^+_\lambda \mid \Sigma^+ \subset V \oplus \C e \}$ have dimension
strictly smaller than the dimension of $X^+_\lambda$, so the
codimension in $O_3$ of each of the components of $Z \cap O_3$ is at
least $|\lambda|-m+1$.  Since $|\lambda|-m+1 + |\mu|+|\nu| = \dim X +
2k+2 > \dim O_3$ we conclude that $\pi^{-1}(X_\mu \cap X_\nu) \cap Z
\cap O_3 = \emptyset$, a contradiction.  This verifies that
$X^+_\lambda \cap X^+_\mu \cap X^+_\nu \subset O_4$.

Notice that for $\Sigma^+ \in O_4$ we must have $\Sigma^+ \cap H = 0$,
since otherwise $\Sigma^+ \cap H = \C h$ for some $0 \neq h \in H$,
and $\Sigma^+ \subset (\C h)^\perp = V \oplus \C h$.  This implies
that $\dim(\Sigma^+ \cap V) \gequ m$, a contradiction.  Since
$\dim(\Sigma^+ + H) = m+3$ and $(\Sigma^+ + H) + V = V\oplus H$, it
follows that $\dim((\Sigma^+ + H)\cap V) = m+1$.  Thus we have a well
defined map $\phi : O_4 \to Y_1$ given by $\Sigma^+ \mapsto (\Sigma^+
\cap V, (\Sigma^+ + H)\cap V)$.

If $\Sigma^+ \in X^+_\lambda \cap O_4$ then $\dim(\Sigma^+ \cap
(E_{p_j(\lambda)} \oplus \C e)) \gequ j$ implies that $\dim((\Sigma^+
\cap V) \cap E_{p_j(\lambda)}) \gequ j-1$ for each $j$.  Since $e
\not\in \Sigma^+$ we also have $\dim((\Sigma^+ + H) \cap
(E_{p_j(\lambda)} \oplus \C e)) \gequ j+1$, which implies that
$\dim((\Sigma^+ + H) \cap E_{p_j(\lambda)}) \gequ j$.  It follows from
this that $\phi(X^+_\lambda \cap O_4) \subset Y_\lambda$, and we
conclude that $\phi$ gives a well defined map $X^+_\lambda \cap
X^+_\mu \cap X^+_\nu \to Y_\lambda \cap Y_\mu \cap Y_\nu \subset Y_1$.

Now let $(A,B) \in Y_\lambda(E_\bull) \cap Y_\mu(F_\bull) \cap
Y_\nu(G_\bull)$ be given.  We must prove that $\phi^{-1}((A,B)) \cap
X^+_\lambda \cap X^+_\mu \cap X^+_\nu$ contains exactly two points.
We know from Theorem~\ref{T:qclasC} (a) that $A = B \cap B^\perp$, so
$W = B/A \oplus H$ is a symplectic space of dimension 4, and
$\LG(2,W)$ is identified with the subset of $\Sigma^+ \in X^+$ such
that $A \subset \Sigma^+ \subset B\oplus H$.  Since $(A,B) \in
Y_\lambda(E_\bull)$ we have $\dim(A \cap E_{p_j(\lambda)}) \gequ j-1$
and $\dim(B \cap E_{p_j(\lambda)}) \gequ j$ for all $j$.  By the proof
of Theorem~\ref{T:qclasC}, there exist unique points $\Sigma_1 \in
X_\lambda(E_\bull)$, $\Sigma_2 \in X_\mu(F_\bull)$, $\Sigma_3 \in
X_\nu(G_\bull)$ such that $A \subset \Sigma_i \subset B$ for each $i$,
and these points are pairwise distinct.  Set $\wt E = \Sigma_1 \oplus
\C e$, $\wt F = \Sigma_2 \oplus \C f$, and $\wt G = \Sigma_3 \oplus \C
g$.  Then there are exactly two points $\Sigma^+ \in X^+$, such that
$A \subset \Sigma^+ \subset B \oplus H$ and $\dim(\Sigma^+ \cap \wt E)
\gequ m$, $\dim(\Sigma^+ \cap \wt F) \gequ m$, $\dim(\Sigma^+ \cap \wt
G) \gequ m$.  This is true because exactly two isotropic planes of
$\LG(2,W)$ are incident to all of the subspaces $\Sigma_1/A \oplus \C
e$, $\Sigma_2/A \oplus \C f$, and $\Sigma_3/A \oplus \C g$.  Since
$\dim(\wt E \cap (E_{p_j(\lambda)}\oplus \C e)) \gequ j+1$ and
$\dim(\Sigma^+ \cap \wt E) \gequ m$, we obtain $\dim(\Sigma^+ \cap
(E_{p_j(\lambda)} + \C e)) \gequ j$, so $\Sigma^+ \in X^+_\lambda$.
Similarly we have $\Sigma^+ \in X^+_\mu \cap X^+_\nu$.

On the other hand, if $\Sigma^+ \in X^+_\lambda \cap X^+_\mu \cap
X^+_\nu$ is such that $A \subset \Sigma^+ \subset B \oplus H$, then $A
= \Sigma^+ \cap V$ and $B = (\Sigma^+ + H) \cap V$.  We claim that
$\Sigma_1 = (\Sigma^+ + \C e) \cap V$.  In fact, the isotropic
$m$-plane $\Sigma'_1 = (\Sigma^+ + \C e) \cap V$ satisfies $A \subset
\Sigma'_1 \subset B$.  Since $\dim((\Sigma^+ + \C e) \cap
(E_{p_j(\lambda)} \cap \C e)) \gequ j+1$ we get $\dim(\Sigma'_1 \cap
E_{p_j(\lambda)}) \gequ j$ for each $j$, so $\Sigma'_1 = \Sigma_1$ must
be the unique isotropic $m$-plane from $X_\lambda(E_\bull)$ which lies
between $A$ and $B$.  It follows from the claim that $\Sigma^+ + \wt E
\subset \Sigma^+ + \C e$, so $\dim(\Sigma^+ \cap \wt E) \gequ m$.
Similarly we see that $\dim(\Sigma^+ \cap \wt F) \gequ m$ and
$\dim(\Sigma^+ \cap \wt G) \gequ m$, which finally establishes the
required 2 to 1 map.
\end{proof}

\begin{example} \label{E:lengthcond}
  The length condition in Proposition~\ref{P:gwdegone} is essential.  For
  example, for $X = \IG(2,8)$ and $X^+ = \IG(3,10)$ we have
  $\gw{\sigma_{1,1}, \sigma_{4,1}, \sigma_{6,5}}{1} = 0$ and
\[
\int_{X^+} [X^+_{1,1}] \cdot [X^+_{4,1}] \cdot [X^+_{6,5}] = 1.
\]
\end{example}

\subsection{Quantum cohomology}
\label{qcIG}

The (small) quantum cohomology ring $\QH^*(\IG)$ is a $\Z[q]$-algebra
which is isomorphic to $\HH^*(\IG,\Z)\otimes_{\Z}\Z[q]$ as a module over
$\Z[q]$. The degree of the formal variable $q$ here is given by
\[ \deg(q)=\int_{\IG} c_1(T_{\IG})\cdot \s_{(1)^\vee} = n+k+1 \,. \]
The ring structure on $\QH^*(\IG)$ is determined by the relation
\[
  \s_\lambda \cdot \s_\mu = 
  \sum \gw{\s_\lambda, \s_\mu, \s_{\nu^\vee}}{d} \, \s_\nu \, q^d,
\]
the sum over $d\gequ 0$ and $k$-strict partitions $\nu$ with
$|\nu| = |\lambda| + |\mu| - (n+k+1)d$. For any partition $\nu$,
define $\nu^*$ by removing the first row of $\nu$, that is, 
$\nu^*=(\nu_2,\nu_3,\ldots)$.

\begin{thm}[Quantum Pieri rule for $\IG$]
\label{igqupieri}
For any $k$-strict partition $\lambda\in\cP(k,n)$
and integer $p \in [1,n+k]$, we have
\[
  \s_p \cdot \s_\lambda = \sum_{\lambda \to \mu} 2^{N(\lambda,\mu)}\,\s_\mu +
  \sum_{\lambda \to \nu} 2^{N(\lambda,\nu)-1} \,\s_{\nu^*}\, q
\]
in the quantum cohomology ring of $\IG(n-k,2n)$.  The first sum is over
partitions $\mu\in\cP(k,n)$
such that $|\mu|=|\lambda|+p$, and the second sum is over 
partitions $\nu\in\cP(k,n+1)$ with $|\nu|=|\lambda|+p$ and $\nu_1 = n+k+1$.
\end{thm}
\begin{proof}
  The first sum is dictated by the classical Pieri rule
  (Theorem~\ref{T:pieriC}).  Theorem~\ref{T:qclasC} (d) implies that
  the coefficient of $q^d$ in the product $\s_p \, \s_\la$ vanishes
  for all $d \gequ 2$.  By Proposition~\ref{P:gwdegone}, the
  coefficient of $\sigma_{\nu}\, q$ in the product is equal to
\[ \gw{\sigma_p, \sigma_\lambda, \sigma_{\nu^\vee}}{1} =
   \frac{1}{2} \int_{X^+} [X^+_p]\cdot [X^+_\lambda]\cdot 
     [X^+_{\nu^\vee}] \,,
\]
where $\nu^\vee$ denotes the dual partition of $\nu$ with respect to
the Grassmannian $\IG = \IG(m,2n)$.  One checks that the result of
dualizing $\nu^\vee$ with respect to $X^+ = \IG(m+1,2n+2)$ is $\nu^+ =
(n+k+1,\nu_1,\dots,\nu_m)$ (information on dual partitions is
given in Section \ref{dualpar}).  Finally, it follows from
Theorem~\ref{T:pieriC} that the coefficient of $\sigma_\nu\, q$ is
equal to $2^{N(\lambda,\nu^+)-1}$ if $\lambda \to \nu^+$, and
otherwise this coefficient is zero.
\end{proof}

\begin{example} \label{E:qpieriC}
  In the quantum ring of $\IG(4,12)$ we have $\sigma_4 \cdot
  \sigma_{(5,3,2,2)} = 4 \sigma_{(8,4,2,2)} + 2 \sigma_{(7,5,2,2)} + 2
  \sigma_{(7,4,3,2)} + \sigma_{(6,5,3,2)} + \sigma_{(8,4,3,1)} + 2
  \sigma_{(4,2,1)}\, q + 2 \sigma_{(3,2,2)}\, q + \sigma_{(3,2,1,1)}\, q$.
  The $q$-terms can be found by expanding the product $[X^+_4] \cdot
  [X^+_{(5,3,2,2)}]$ in $\HH^*(\IG(5,14))$.
\end{example}

\begin{thm}[Ring presentation] 
\label{igqpres}
The quantum cohomology ring $\QH^*(\IG)$ is presented as a quotient of the
polynomial ring $\Z[\sigma_1,\ldots,\sigma_{n+k},q]$ modulo the relations
\begin{equation}
\label{QR1}
\det(\sigma_{1+j-i})_{1\lequ i,j \lequ r} = 0, \ \ \ \ 
n-k+1\lequ r \lequ n+k
\end{equation}
and
\begin{equation}
\label{QR2}
\sigma_r^2 + 2\sum_{i=1}^{n+k-r}(-1)^i \sigma_{r+i}\sigma_{r-i}=
 (-1)^{n+k-r}\s_{2r-n-k-1}\, q,
 \ \  \ \ k+1\lequ r \lequ n.
\end{equation}
\end{thm}
\begin{proof}
  According to \cite{ST} (see also \cite[Sec.~10]{FPa}), we only need
  to check how the classical relations in Theorem \ref{presIG} deform
  under the quantum product.  The relations (\ref{QR1}) are true in
  $\QH^*(\IG)$ because the degree of $q$ is greater than $n+k$.  Using
  (\ref{R2}) for $X^+ = \IG(m+1,2n+2)$ we obtain the identity
  $[X^+_r]^2 + \sum_{i=1}^{n+k-r} (-1)^i [X^+_{r+i}] [X^+_{r-i}] =
  (-1)^{n+k-r} [X^+_{n+k+1}] [X^+_{2r-n-k-1}] = (-1)^{n+k-r}
  [X^+_{(n+k+1,2r-n-k-1)}]$ in $\HH^*(X^+)$.  If $2r < n+k+1$ then the
  right hand side is understood to be zero.  The quantum Pieri rule
  for $\IG$ now implies that the $q$-terms of (\ref{QR2}) agree.
\end{proof}

\subsection{Computing Gromov-Witten invariants}
\label{S:computing}

For any partition $\la = (\la_1,\ldots,\la_\ell)$ in $\cP(k,n)$, iterating
the quantum Pieri rule as in the proof of Theorem \ref{presIG}(b) gives
\[
  \s_{\la_1} \cdots \s_{\la_{\ell}} = \s_\la + \sum_{d,\mu} c_{d,\mu}
  \s_\mu  q^d
\]
in $\QH^* (\IG)$, where $c_{d,\mu}\in\Z$ and the partitions $\mu$ in the
sum satisfy $\mu \succ \la$ or $|\mu| < |\la|$. Therefore our quantum
Pieri formula can be used recursively to identify a given Schubert
class with a polynomial in the special Schubert classes and $q$.
These expressions together with the quantum Pieri rule can then be
used to evaluate any quantum product $\s_\la \cdot \s_\mu$, and hence
any Gromov-Witten invariant $\langle \s_\la, \s_\mu, \s_\nu \rangle_d$
on $\IG$, as in the following example.

\begin{example}
  In $\QH^* (\IG(3,10))$ we have
  \[ \s_4 \s_2^2 =
     \s_{(4,2,2)} + \s_{(4,3,1)} + 3\s_{(5,2,1)} + 4\s_{(6,1,1)}
         + 3\s_{(5,3)} + 5\s_{(6,2)} + 8\s_{(7,1)} + 2q.          \]
  This, combined with other quantum Pieri products, leads to the
  expression
  \[ \s_{(4,2,2)} = \s_4 \s_2^2 - \s_4 \s_3 \s_1
                     - \s_5 \s_2 \s_1 + \s_5 \s_3
                     + 2 \s_7 \s_1 - q.               \]
  Using this identity we can evaluate the product
  \begin{gather*}
  \sigma_{(4,2,2)} \cdot \sigma_{(5,3,1)} = 
  (\s_4 \s_2^2 - \s_4 \s_3 \s_1- \s_5 \s_2 \s_1 + \s_5 \s_3
  + 2 \s_7 \s_1 - q) \cdot \sigma_{(5,3,1)} \\
  = \sigma_{(7,6,4)}  + 4 \sigma_{(7,2)}\, q +  \sigma_{(5,3,1)}\, q +
  \sigma_{(7,1,1)}\, q + 2 \sigma_{(6,2,1)}\, q + 3 \sigma_{(6,3)}\, q +
  \sigma_{(5,4)} \, q +  \sigma_1\, q^2.
  \end{gather*}
  Since the Poincar\'e dual of $\sigma_1$ is $\sigma_{(7,6,4)}$, we
  obtain $\gw{\sigma_{(4,2,2)}, \sigma_{(5,3,1)}, \sigma_{(7,6,4)}}{2} =
  1$.
\end{example}


\section{The Grassmannian $\OG(n-k,2n+1)$}
\label{QCOGodd}

\subsection{Schubert classes}
\label{ogoddscs}

Consider a vector space $V\cong \C^{2n+1}$ and a nondegenerate
symmetric bilinear form on $V$.  For each $m=n-k < n$, the odd 
orthogonal Grassmannian
$\OG=\OG(m,2n+1)$ parametrizes the $m$-dimensional isotropic
subspaces in $V$.  The algebraic variety $\OG$ has dimension
$2m(n-m)+m(m+1)/2$, the same as the dimension of $\IG(m,2n)$.
Moreover, the Schubert varieties in $\OG$ are indexed by the 
same set of $k$-strict partitions $\cP(k,n)$ as for $\IG$.

An isotropic flag $F_\bull$ of $V$ is a complete flag $0=F_0\subsetneq
F_1 \subsetneq \dots \subsetneq F_{2n+1} = V$ such that $F_{n+i} =
F_{n+1-i}^\perp$ for all $1 \lequ i \lequ n+1$.  For each such flag
and $\la\in\cP(k,n)$, define the Schubert variety
\[ 
   X_\lambda(F_\bull) = \{ \Sigma \in \OG \mid \dim(\Sigma \cap
   F_{\ov{p}_j(\lambda)}) \gequ j \ \ \forall\, 1 \lequ j \lequ 
   \ell(\lambda) \} \,,
\]
where 
\[
\ov{p}_j(\lambda) = n+k-\lambda_j + \#\{i<j : \lambda_i+\lambda_j
\lequ 2k+j-i \} + \begin{cases}
             2 & \text{if $\la_j \lequ k$}, \\
             1 & \text{if $\la_j > k$}.
\end{cases}
\]
The codimension of this variety is equal to $|\lambda|$.  We define
$\ta_\lambda \in \HH^{2|\lambda|}(\OG) = \HH^{2|\lambda|}(\OG,\Z)$ to
be the cohomology class dual to the cycle given by
$X_\lambda(F_\bull)$.  These Schubert classes form a $\Z$-basis for
the cohomology ring of $\OG$.

\subsection{Classical Pieri rule}
\label{ogoddclpieri}

The following comparison of structure constants between $\IG$ and 
$\OG$ is well known; in fact this extends to the complete flag 
varieties in types B and C; see \cite[Sec.\ 3.1]{BS}. For each 
$k$-strict partition $\la$, let $\ell_k(\la)$ denote the number
of parts $\la_i$ which are strictly greater than $k$. 
Consider $\la$ and $\mu$
in $\cP(k,n)$ and suppose that
\[
\s_{\la}\s_{\mu}=\sum_{\nu}e_{\la,\mu}^{\nu}\,\s_{\nu}
\]
in $\HH^*(\IG(n-k,2n))$.
Then for the corresponding formula
\[
\ta_\lambda\ta_\mu=\sum_{\nu}f_{\la,\mu}^\nu\,\ta_\nu.
\]
in $\HH^*(\OG(n-k,2n+1))$, we have
\begin{equation}
\label{igogrel}
f_{\lambda,\mu}^{\nu}=2^{\ell_k(\nu)-\ell_k(\la)-\ell_k(\mu)}
\,e_{\lambda,\mu}^\nu.
\end{equation}
It follows that the map $\HH^*(\OG,\Z)\to \HH^*(\IG,\Q)$ sending
$\tau_{\lambda}$ to $2^{-\ell_k(\lambda)}\sigma_{\lambda}$ is an
isomorphism onto its image.

The {\em special} Schubert varieties for $\OG$
are the varieties $X_p=X_{(p)}(F_{\bull})$, for $1\lequ p\lequ n+k$.
These are defined by a single Schubert condition as follows: let
\begin{equation}
\label{vare}
\varepsilon(p)=n+k-p+
\begin{cases}
2&\text{if $p \lequ k$},\\
1&\text{if $p > k$}.
\end{cases}
\end{equation}
Then $X_p(F_\bull)= \{\, \Sigma\in \OG\,|\,\Sigma\cap
F_{\varepsilon(p)}\ne 0\,\}.$ We let $\ta_p$ denote the cohomology
class of $X_p$.  The classical Pieri rule for $\OG$ will involve the
same relation $\la\to\mu$ and set $\A\subset\mu\ssm\lambda$ that
appeared in Definition \ref{D:pieriarrow}. For each $\la$ and $\mu$
with $\la\to\mu$, we let $N'(\lambda,\mu)$ equal the number
(respectively, one less than the number) of components of $\A$, if
$p\lequ k$ (respectively, if $p>k$).

\begin{thm}[Pieri rule for $\OG(m,2n+1)$]
\label{thm.ogoddpieri}
For any $k$-strict partition $\la$ and integer $p \in [1,n+k]$,
we have
\begin{equation}
\label{ogclass}
\ta_p\cdot\ta_{\lambda} = 
\sum_{\lambda \to \mu, \,
  |\mu|=|\lambda|+p} 2^{N'(\lambda,\mu)} \, \ta_\mu.
\end{equation}
\end{thm}
\begin{proof}
The result follows immediately from Theorem \ref{T:pieriC} and 
(\ref{igogrel}).
\end{proof}

\begin{example} \label{E:pieriB}
  For $\OG(4,13)$ we have $k=2$ and $n=6$ as in
  Example~\ref{E:pieriC}.  Using the diagrams displayed there, we
  obtain $$\tau_4 \cdot \tau_{(5,3,2,2)} = 2 \tau_{(8,4,2,2)} +
  \tau_{(7,5,2,2)} + 2 \tau_{(7,4,3,2)} + \tau_{(6,5,3,2)} +
  \tau_{(8,4,3,1)}.$$
\end{example}

\subsection{Presentation of $\HH^*(\OG,\Z)$}
\label{ogoddpresentation}

If $\cS$ (respectively $\cQ$) denotes the tautological 
subbundle (respectively, quotient bundle) over $\OG(n-k,2n+1)$,
then one has that
\begin{equation}
\label{c-to-t}
c_p(\cQ)=
\begin{cases}
\ta_p & \text{if $p\lequ k$},\\
2\ta_p & \text{if $p> k$}.
\end{cases}
\end{equation}
Let $\delta_p=1$, if $p\lequ k$, and $\delta_p=2$, otherwise.

\begin{thm}
\label{presOGodd}
{\em a)} The cohomology ring $\HH^*(\OG(n-k,2n+1),\Z)$ is presented as a
quotient of the polynomial ring $\Z[\ta_1,\ldots,\ta_{n+k}]$ modulo
the relations
\begin{gather}
\label{ogoddR1}
\det(\delta_{1+j-i}\ta_{1+j-i})_{1\lequ i,j \lequ r} = 0, \ \ \ \ 
n-k+1\lequ r \lequ n, \\
\label{ogoddR1'}
\sum_{p=k+1}^r(-1)^p\tau_p
\det(\delta_{1+j-i}\ta_{1+j-i})_{1\lequ i,j \lequ r-p}=0, \ \ \ \
n+1\lequ r \lequ n+k,
\end{gather}
and
\begin{equation}
\label{ogoddR2}
\ta_r^2 + \sum_{i=1}^r(-1)^i \delta_{r-i}\ta_{r+i}\ta_{r-i}= 0,
 \ \  \ \ k+1\lequ r \lequ n.
\end{equation}

\medskip
\noin {\em b)} The monomials $\ta^{\la}=\prod_i\ta_{\la_i}$ with
$\la\in\cP(k,n)$ form a $\Z$-basis for $\HH^*(\OG,\Z)$.
\end{thm}
\begin{proof} 
  The statement (b) is proved exactly as the corresponding part of
  Theorem \ref{presIG}, in particular the special Schubert classes
  $\ta_p$ generate $\HH^*(\OG)$.  We will use Lemma \ref{preslemma}, and
  begin by noting that the rank of $\HH^*(\OG)$ as a $\Z$-module is the
  same as that of $\HH^*(\IG(n-k,2n))$. The Whitney sum formula shows
  that the relations (\ref{ogoddR1}) hold in the cohomology ring of
  $\OG$. The remaining relations follow easily from the Pieri rule.
  For (\ref{ogoddR1'}), one observes that for each $r$,
\[
\det(\delta_{1+j-i}\ta_{1+j-i})_{1\lequ i,j \lequ r} = \ta_{(1^r)}.
\]
This proves that condition (i) of Lemma \ref{preslemma} holds.

To prove (ii), let $K$ be any field of characteristic different than
$2$.  For each $r$, set $h_r = \det(c_{1+j-i})_{1\lequ i,j \lequ
r}$. Using the change of variables in (\ref{c-to-t}), we see that it
suffices to show that the quotient of the ring $K[c_1,\ldots,c_{n+k}]$
modulo the relations
\[
h_r = 0, \ \ \ n-k+1\lequ r \lequ n, \ \ \ 
\sum_{p=k+1}^r(-1)^p c_p h_{r-p}=0,  \ \ \
n+1\lequ r \lequ n+k,
\]
and
\[
c_r^2 + 2\sum_{i=1}^r(-1)^i c_{r+i}c_{r-i}= 0,
 \ \  \ \ k+1\lequ r \lequ n
\]
is finite dimensional. By Lemma \ref{typeApres}, it will suffice to
show that the relations $h_r=0$ for $n+1\lequ r \lequ n+k$ are also
true in this ring. The formal identity
\[
h_{n+1} = \left(\sum_{p=1}^k(-1)^{p+1}c_p h_{n+1-p}\right)+
\left(\sum_{p=k+1}^{n+1}(-1)^{p+1} c_p h_{n+1-p}\right)
\]
and the first line of relations prove that $h_{n+1}=0$; the vanishing
of $h_r$ for $r>n+1$ is established similarly.

Next, suppose that $\text{char}(K)=2$. We will show that each generator 
$\ta_p$ is a nilpotent element of the quotient of $K[\ta_1,\ldots,
\ta_{n+k}]$ modulo relations (\ref{ogoddR1}), (\ref{ogoddR1'}), and
(\ref{ogoddR2}). In characteristic 2, the equations (\ref{ogoddR1})
 assert the vanishing of Schur determinants in $\ta_1,\ldots,\ta_k$
for $r=n-k+1,\ldots,n$, and we know by Lemma \ref{typeApres} that 
this implies $\ta_1,\ldots,\ta_k$ are nilpotent. Now rewrite 
(\ref{ogoddR1'}) as 
\[
\ta_r=\sum_{p=k+1}^{r-1}(-1)^{p+1}\tau_p
\det(\delta_{1+j-i}\ta_{1+j-i})_{1\lequ i,j \lequ r-p}, \ \ \ \
n+1\lequ r \lequ n+k.
\]
As all determinants on the right hand side are nilpotent, so is each
$\ta_r$ for $r=n+1,\ldots,n+k$. Finally, the relations (\ref{ogoddR2})
are used to complete the proof of (a). 
\end{proof}

\subsection{Gromov-Witten invariants}
\label{ogoddgwis}

In this and the following section we assume that $k>0$, or
equivalently $m<n$. Given a degree $d\gequ 0$ and partitions
$\lambda,\mu,\nu\in \cP(k,n)$ such that $|\lambda| + |\mu| + |\nu| =
\dim(\OG)+d(n+k)$, we define the Gromov-Witten invariant $\langle
\ta_\lambda, \ta_\mu, \ta_\nu \rangle_d$ to be the number of rational
maps $f\colon \bP^1\to \OG$ of degree $d$ such that $f(0)\in
X_\lambda(E_\bull)$, $f(1)\in X_\mu(F_\bull)$, and $f(\infty)\in
X_\nu(G_\bull)$, for given isotropic flags $E_\bull$, $F_\bull$, and
$G_\bull$ in general position. Here a rational map of degree $d$ to
$\OG$ is a morphism $f : \bP^1 \to \OG$ such that $\int f_*[\bP^1]
\cdot \ta_1 = d$.  We define the kernel and span of such a map as in
Section \ref{iggwis}; again we have $\dim \Ker(f) \gequ m-d$ and
$\dim\Span(f) \lequ m+d$ for any map $f : \bP^1\to\OG$ of degree $d$.

In the orthogonal Lie types, the basic model for degree $d$ maps to
$\OG$ will have target $\OG(d,2d)$, when $d$ is even, but a different
space $\ov{\OG}=\ov{\OG}(d,2d)$ when $d$ is odd. The variety
$\ov{\OG}$ parametrizes isotropic subspaces of dimension $d$ in a
vector space $W\cong \C^{2d}$ equipped with a degenerate symmetric
bilinear form such that $\Ker(W\to W^*)$ has dimension $2$.

For any integer $d\lequ m$ the variety $Y_d$ parametrizes pairs
$(A,B)$ of subspaces of $V$ with $A\subset B \subset A^{\perp}$, $\dim
A = m-d$, and $\dim B = m+d$, as in Section \ref{iggwis}. Moreover,
for any Schubert variety $X_\la$ in $\OG$, the subvariety $Y_\la$ of
$Y_d$ consists of the set of pairs $(A,B)\in Y_d$ such that there
exists a $\Sigma\in X_{\la}$ with $ A\subset \Sigma \subset B$.  Here
we compute that $\dim(Y_d) = \dim(\OG)+ d(n+k)-3d(d-1)/2$, and we'll
show that the codimension of $Y_\la$ in $Y_d$ is at least equal to
$|\la|-d(d-1)/2$. When $d=m+1$, the role of $Y_d$ is played by the
Grassmannian $Y_{m+1}:=\G(2m+1, V)$ of all subspaces $B\subset V$ with
$\dim B = 2m+1$, and the varieties $Y_{\lambda}\subset Y_{m+1}$ are
defined in the same way.

Let $T_d$ be the variety of triples $(A,\Sigma,B)$
such that $(A,B) \in Y_d$, $\Sigma \in \OG$, and $A\subset \Sigma\subset B$.
Let $\pi : T_d \to Y_d$ be the projection.
For $\la \in \cP(k,n)$ we define
$T_\lambda(E_\bull) = \{(A,\Sigma,B) \in T_d \mid \Sigma \in
X_\lambda(E_\bull) \}$.
We have $\pi(T_\lambda(E_\bull)) = Y_\lambda(E_\bull)$.

When $d\le m$, we
let $Y'_d$ denote the variety parametrizing triples $(A,A',B)$ of
subspaces of $V$ with $(A,B)\in Y_d$ and $A'$ isotropic with $A
\subset A'\subset B$ and $\dim A' = \dim A+1$. There is a
projection map $\varphi : Y'_d\to Y_d$. Moreover, to each Schubert
variety $X_\la$ in $\OG$ there corresponds subvariety $Y'_\la$ of
$Y'_d$, defined as the locus of $(A,A',B)$ such that there exists a
$\Sigma\in X_\la$ with $A'\subset \Sigma\subset B$.
Let $T'_d$ denote the variety of nested spaces $(A,A',\Sigma,B)$
with $\Sigma\in \OG$ and $B\subset A^\perp$.
There is a projection map $\pi'\colon T'_d\to Y'_d$.
We define $T'_\la$ to consist of $(A,A',\Sigma,B)$ in $T'_d$ with
$\Sigma\in X_\la$.
Then $\pi'(T'_\la)=Y'_\la$.

Attached to each $(A,B)\in Y_d$ there is an invariant
$r=\dim(B\cap B^\perp)-\dim(A)$.
The quantity $r$ measures the degeneracy of the induced bilinear form on $B/A$.
The general point of $Y_d$ has $r=0$, while the $r=2$ locus of $Y_d$ is a
locally closed subvariety of $Y_d$.
For $\la\in \cP(k,n)$ we define $N_d=N_d(\la)=\#\{j\le m\,|\,\la_j=d-j\le k\}$.

\begin{lemma} \label{L:fiberB}
{\rm (a)} The restricted projection
$\pi : T_\la(E_\bull) \to Y_\la(E_\bull)$ is generically
$2^{N_d}$-to-$1$ when $\rho_{d-1}\subset\la$ and has fibers of positive
dimension when $\rho_{d-1}\not\subset\la$.
\smallskip \\
\noin When $d\le m$, we furthermore have:\\
\noin {\rm (b)} The restriction of $\pi$ over the
$r=2$ locus of $Y_\la(E_\bull)$ is
generically unramified $2^{N_d}$-to-$1$ when $(\rho_{d-1},1)\subset\la$
and has fibers of positive dimension when
$(\rho_{d-1},1)\not\subset\la$.
\smallskip \\
\noin {\rm (c)} The map $\pi'\colon T'_\la(E_\bull)\to
Y'_\la(E_\bull)$ is generically $1$-to-$1$ when $\rho_{d-1}\subset\la$.
\smallskip \\
\noin {\rm (d)} The restriction of $\pi'$ to the
$r=2$ locus is generically $1$-to-$1$ when
$(\rho_{d-1},1)\subset\la$.
\end{lemma}
\begin{proof}
We can assume that $E_r=\Span\{e_1,\ldots,e_r\}$ where
$\{e_1,\ldots,e_{2n+1}\}$ is a standard orthogonal basis for $V$,
i.e.\ $(e_i,e_j)=0$ when $i+j\ne 2n+2$.
We write
$X_\la$ for $X_\la(E_\bull)$,
$Y_\la$ for $Y_\la(E_\bull)$, and $T_\la$ for $T_\la(E_\bull)$.
Set $p_j=\bar p_j(\la)$ for $1\le j\le m$,
$p_0=0$, $p_{m+1}=2n+2$, $\#_j=\#\{i<j\,|\,p_i+p_j>2n+2\}$,
and $s_j=\max(m+d+p_j-2n-1-j,0)$.
We first prove statements (a)--(d) in the case $d\le m$, then we state
the modifications necessary to obtain (a) in case $d=m+1$.

For $(A,B)$ in $Y_\la$, we set $A_j=A\cap E_{p_j}$ and
$B_j=B\cap E_{p_j}$.
By definition there exists $\Sigma$ in $X_\la$ with
$A\subset\Sigma\subset B$.
We set $\Sigma_j=\Sigma\cap E_{p_j}$, so that
$\dim(\Sigma_j)\ge j$.
Now,
\begin{gather}
\dim(A_j)\ge j-d,
\label{Bbd1} \\
\dim(B_j)\ge j+s_j,
\label{Bbd2}
\end{gather}
since $A$ is of codimension $d$ in $\Sigma$, and
$B_j$ is the intersection of spaces of dimension $m+d$ and $p_j$,
while also $\Sigma_j\subset B_j$.
For $j$ such that equality holds in \eqref{Bbd2} we have
$j\le \dim(\Sigma_{j-\varepsilon}^\perp\cap B_j)=
j+s_j-\dim(\Sigma_{j-\varepsilon})
+\dim(\Sigma_{j-\varepsilon}\cap B_j^\perp)$ for $\varepsilon\in\{0,1\}$,
since $(\Sigma_{j-\varepsilon}^\perp\cap B_j)^\perp=\Sigma_{j-\varepsilon}
+B_j^\perp$.
So, equality in \eqref{Bbd2} implies
\begin{gather}
\dim(B_j\cap B_j^\perp) \ge j-s_j,
\label{Bbd3} \\
\dim(B_{j-1}\cap B_j^\perp) \ge j-1-s_j.
\label{Bbd4a}
\end{gather}
When $\la_j=d-j\le k$ we have
$s_j=1+\#_j$, and hence
$\dim(\Sigma_{j-1}\cap B_j^\perp)\ge
j-1-\#_j=j-s_j$, which gives us the following strengthening of \eqref{Bbd4a}:
\begin{equation}
\label{Bbd4b}
\dim(B_{j-1}\cap B_j^\perp)\ge
\left\{
\begin{array}{ll}
j-s_j,&\text{when $\la_j=d-j\le k$},\\
j-1-s_j,&\text{otherwise}.
\end{array}\right.
\end{equation}

Define $U_\la$ to be the open subset of $Y_\la$ of $(A,B)$ satisfying
equality in \eqref{Bbd1} for $j\ge d$ and in \eqref{Bbd2},
\eqref{Bbd3}, \eqref{Bbd4b} for $1\le j\le d$.
We show: $U_\la$ contains a point in the $r=0$ locus
when $\rho_{d-1}\subset\la$ and contains a point in the
$r=2$ locus when $(\rho_{d-1},1)\subset\la$.
Then we show for $(A,B)$ in $U_\la$ that there are precisely
$2^{N_d}$ possible $\Sigma\in X_\la$ satisfying $A\subset\Sigma \subset B$.

The condition $\la_j\ge d-j$ implies
$s_j=1+\#_j$ when $\la_j=d-j\le k$, and
$s_j\le\#_j$ otherwise.
For $j>d$ we also have
$s_j\le \#\{i\le d\,|\,p_i+p_j>2n+2\}$.
We can therefore choose $\pi\in {\mathfrak S}_d$ such that
for all $1\le j\le m+1$ and $i\le s_j$ we have
\begin{gather}
\pi(i)\le\left\{
\begin{array}{ll}
j,& \text{when $\la_j=d-j\le k$}, \\
j-1, & \text{otherwise},
\end{array}\right.
\label{Bpiineq1} \\
p_{\pi(i)}+p_j\ge 2n+3+s_j-i.
\label{Bpiineq2}
\end{gather}
Further, we dictate that among all permutations satisfying
\eqref{Bpiineq1}, \eqref{Bpiineq2}, $\pi$ is chosen so that
$\pi^{-1}(d)$ is as large as possible.

We exhibit a point in the $r=0$ locus of $U_\la$.
Set
$A=\Span(e_{p_{d+1}},\ldots,e_{p_m})$
and
$B=\Span(e_{p_1},\ldots,e_{p_m},u_1,\ldots,u_d)$
where
$u_i=e_{2n+2-p_{\pi(i)}}+e_{p_j-1-s_j+i}$
for $s_{j-1}<i\le s_j$.
Equality holds in \eqref{Bbd1} for $j\ge d$ and, by \eqref{Bpiineq2}
equality holds in \eqref{Bbd2} for all $j$.
The pairings
$(e_{p_{\pi(i)}},u_i)=1$ and
$(e_{p_{\pi(i)}},u_{i'})=0$ for $i'>i$ imply the nondegeneracy of the
restriction of $(\ ,\,)$ to
$\Span(e_{p_{\pi(1)}},\ldots,e_{p_{\pi(s_j)}},u_1,\ldots,u_{s_j})$ for
any $j$.
So $(A,B)$ lies in the $r=0$ locus of $Y_\la$, and from \eqref{Bpiineq1}
we have equality for all $j$ of \eqref{Bbd3} and \eqref{Bbd4b}, hence
$(A,B)\in U_\la$.

Under the further assumption that $\ell(\la)\ge d$,
we exhibit a point in the $r=2$ locus of $U_\la$.
We set $i_0=\pi^{-1}(d)$,
define $j_0$ to satisfy $s_{j_0-1}<i_0\le s_{j_0}$,
and set $t_0=p_{j_0}-1-s_{j_0}+i_0$
By \eqref{Bpiineq2}, we have
$t_0\ge 2n+2-p_d$. Now we distinguish two cases.

First is the case $t_0>2n+2-p_d$.
Then we define $A$ and $B$ as above, but with
$u_{i_0}=e_{t_0}$.
By \eqref{Bpiineq1}, $j_0>d$, and
we have $t_0>p_{j_0-1}\ge p_d$.
The computation of $\dim(B_j)$ goes through as before.
We have the pairing $(e_{p_{\pi(i)}},u_i)=1$ for $i\ne i_0$,
and $(e_{p_{\pi(i)}},u_{i'})=0$ for all $i$ and $i'$ with $i'>i$.
Since $s_d<i_0$, the argument for equality in \eqref{Bbd3}, \eqref{Bbd4b}
for $j\le d$ goes through as before.
Now we claim $(e_{p_j},u_{i_0})=0$ for all $j\le d$.
So we have $e_{p_d}$, $u_{i_0}\in B^\perp$, while
a direct argument shows that the bilinear form restricted to
$\Span(e_{p_1},\ldots,e_{p_{d-1}},u_1,\ldots,u_{i_0-1},u_{i_0+1},\ldots,u_d)$
is nondegenerate.
Hence $(A,B)$ lies in the $r=2$ locus of $U_\la$.

Suppose, to the contrary, there exists $c\le d$
with $(e_{p_c},u_{i_0})\ne 0$.
Then $p_c=2n+2-t_0$, and $c<d$.
We set $i_1=\pi^{-1}(c)$ and define $j_1$ so
$s_{j_1-1}<i_1\le s_{j_1}$.
Then
\begin{equation}
\label{Bi0j0i1j1}
p_{j_1}-1-s_{j_1}+i_1\ge 2n+2-p_c=p_{j_0}-1-s_{j_0}+i_0,
\end{equation}
hence the inequality is strict, and $i_1>i_0$.
Now, if we define $\tilde\pi\in {\mathfrak S}_d$ by
$\tilde\pi(i)=\pi(i)$ for $i\notin\{i_0,i_1\}$, with
$\tilde\pi(i_0)=c$ and $\tilde\pi(i_1)=d$, then
$\tilde\pi$ satisfies \eqref{Bpiineq1}, \eqref{Bpiineq2}.
The only inequality needed to be checked is
$p_c+p_{j_0}\ge 2n+3+s_{j_0}-i_0$, and this holds by \eqref{Bi0j0i1j1}.
Since $\tilde\pi$ satisfies
$\tilde\pi^{-1}(d)>\pi^{-1}(d)$, we have reached a contradiction.

In case $t_0=2n+2-p_d$ we define $A$ and $B$ as above,
but with
$u_{i_0}=e_c$,
where $c$ is defined to be the smallest
positive integer not equal to $p_j$ or to $2n+2-p_j$ for any $j$.
A general remark is that if equality in \eqref{Bpiineq2} holds for some
$i\le s_j$, then $\pi(i')>\pi(i)$ for all $i'<i$.
Indeed, if $i'<i$ then
\[
0\le p_{\pi(i')}+p_j-2n-3-s_j+i'<
p_{\pi(i')}+p_j-2n-3-s_j+i=p_{\pi(i')}-p_{\pi(i)}.
\]
In this case we have such equality for $i=i_0$.
Hence $i_0=1$.
The $n+d+1-k-j_0$ integers
$$p_1,\ldots,p_d, \,\, 2n+2-p_{j_0},\ldots,2n+2-p_m$$
are distinct and less than or equal to $2n+2-t_0$.
But $n+d+1-k-j_0=2n+2-p_{j_0}+s_{j_0}=2n+2-t_0$, so
$p_1$, $\ldots$, $p_d$ must be the first $d$ integers not equal to
$2n+2-p_j$ for any $j>d$.
An inductive argument shows that $\pi(i)=d+1-i$ for $i=1$, $\ldots$, $d$,
with equality in \eqref{Bpiineq2} for each $i$.
As before, we have $e_{p_d}$, $u_{i_0}\in B^\perp$,
and $(A,B)$ lies in the $r=2$ locus of $U_\la$.

Given $(A,B)\in U_\la$, we claim that for $j\le d$ are precisely
$2^{\#\{i\le j|\la_i=d-i\le k\}}$ isotropic spaces $\Sigma_j\subset
B_j$ such that $\dim(\Sigma_j\cap B_i)\ge i$ for $i\le j$.  We prove
this by induction on $j$, the base case $j=0$ being clear.  The
equalities \eqref{Bbd2} and \eqref{Bbd3} imply that a maximal
isotropic subspace of $B_j$ has dimension $j$.  So, using the
induction hypothesis, we have that $\dim(\Sigma_j\cap B_i)\ge i$ for
$i\le j$ implies $\Sigma_j\cap B_{j-1}=\Sigma_{j-1}$.  So, it is
enough to show that $\Sigma_{j-1}$ extends in precisely two ways to an
isotropic subspace of dimension $j$ in $B_j$ when $\la_j=d-j\le k$,
and extends uniquely to such a subspace otherwise.  If $\la_j=d-j\le
k$, then we have $\dim(\Sigma_{j-1}^\perp\cap
B_j)=s_j+1+\dim(\Sigma_{j-1}\cap B_j^\perp) =j+1$, and
$\Sigma_{j-1}^\perp\cap B_j\cap (\Sigma_{j-1}^\perp\cap B_j)^\perp=
\Sigma_{j-1}+(B_j\cap B_j^\perp)=\Sigma_{j-1}$ (since $B_j\cap
B_j^\perp=\Sigma_{j-1}\cap B_j^\perp$), so $(\Sigma_{j-1}^\perp\cap
B_j)/\Sigma_{j-1}$ is two-dimensional with induced nondegenerate
symmetric bilinear form.  Otherwise, we have
$\dim(\Sigma_{j-1}^\perp\cap B_j)\le j$, so the extension is unique.
By the equality \eqref{Bbd1}, $\Sigma_d\cap A=0$, so if we set
$\Sigma=\Sigma_d+A$ then $\Sigma\in X_\la$.  For any $\Sigma\in X_\la$
with $A\subset \Sigma\subset B$ we have $\dim(\Sigma\cap B_i)\ge i$
for all $i$, hence $\Sigma\cap B_d=\Sigma_d$ and $\Sigma=\Sigma_d+A$.

By the same argument, with families
of vector spaces over local Artinian $\C$-algebras, we see that
the morphism $T_\la\to Y_\la$ is unramified (hence \'etale) over $U_\la$.

Now suppose $\rho_{d-1}\not\subset\la$.
We want to show that $\dim(Y_\la)<\dim(T_\la)$.
Let $X_\la^\circ\subset X_\la$ be the Schubert cell.
It suffices to show that if $(A,B)\in Y_\la$ with
$A_d=0$, there exists
$\Sigma'\in X_\la\smallsetminus X_\la^\circ$ with
$A\subset \Sigma'\subset B$.
If this is false, then we can choose $\Sigma\in X^\circ_\la$ with
$A\subset \Sigma\subset B$. Set
$\Sigma_i=\Sigma\cap E_{p_i}$ for $1\le i\le m$.
We have $\la_j<d-j$ for some $j$ with $\la_j\le k$.
Then $s_j>1+\#_j$, hence
$\dim(\Sigma_{j-1}^\perp\cap E_{p_j})\ge p_j-\#_j>p_j-s_j+1$.
Using \eqref{Bbd2} we obtain
$\dim(\Sigma_{j-1}^\perp\cap B_j)>j+1$,
so there exists a $j$-dimensional isotropic extension $\Sigma'_j$
of $\Sigma_{j-1}$ contained in $B\cap E_{p_j-1}$.
For $i=j+1$, $\ldots$, $d$ we now choose an $i$-dimensional isotropic
extension $\Sigma'_i$ of $\Sigma'_{i-1}$ contained in $B_i$.
The subspace $\Sigma'=\Sigma'_d+A$ then satisfies
$\Sigma'\in X_\la\smallsetminus X^\circ_\la$ and
$A\subset \Sigma'\subset B$, and statement (a) is proved.

To complete the proof of (b) it suffices by semi-continuity of
fiber dimensions to treat the case that $\rho_{d-1}\subset \la$ and
$\ell(\la)=d-1$.
Generically on the $r=2$ locus of $Y_\la$, the intersection of
$B\cap B^\perp$ with $E_{p_{d-1}}$ is trivial, since
$\la_{d-1}\ge 1$.
So it suffices to show that if $(A,B)$ lies in the
$r=2$ locus of $Y_\la$ and
satisfies $A_d=0$ and $B^\perp\cap B_{d-1}=0$ then there exists
$\Sigma'\in X_\la\smallsetminus X_\la^\circ$ with
$A\subset \Sigma'\subset B$.
If this is false, then we can choose $\Sigma\in X^\circ_\la$ with
$A\subset \Sigma\subset B$.
Since $\la_d=0$,
the assumptions imply
$\dim(B^\perp\cap B_d)=2$.
We set $\Sigma_{d-1}=\Sigma\cap E_{p_{d-1}}$.
Since $B^\perp\cap B_d$ meets $\Sigma_{d-1}$ trivially, 
there exists an isotropic extension $\Sigma'_d$ of $\Sigma_{d-1}$
of dimension $d$ contained in $B\cap E_{p_d-1}$.
The space $\Sigma'=\Sigma'_d+A$ then satisfies
$\Sigma'\in X_\la\smallsetminus X_\la^\circ$.

Statements (c) and (d) follow from (a) and (b).
If $(A,B)\in U_\la$ then the generic $A'$
with $(A,A',B)\in Y'_\la$ is contained in a unique $\Sigma$
with $A\subset \Sigma\subset B$ and $\Sigma\in X_\la$.

In case $d=m+1$, then exactly as above we have \eqref{Bbd2}, and equality
in \eqref{Bbd2} implies 
\eqref{Bbd3} and \eqref{Bbd4b}.
We define $U_\la\subset Y_\la$ by equality in
\eqref{Bbd2}, \eqref{Bbd3}, \eqref{Bbd4b} for $1\le j\le m$.
When $\rho_{d-1}\subset\la$,
the construction of a point in the $r=0$ locus of $U_\la$
is as above, except that we take
$u_{i_0}=e_{d+1-i_0} + e_{2n+1-d+i_0}$ where $i_0=\pi^{-1}(d)$,
and this satisfies $(e_{p_j},u_{i_0})=0$ for $1\le j\le m$
(if for some $c$ we have $(e_{p_c}, u_{i_0})\ne 0$, then $p_c=d+1-i_0$,
and this contradicts the maximality of $\pi^{-1}(d)$, as argued above).
Then $B=\Span(e_{p_1},\ldots,e_{p_m},u_1,\ldots,u_d)$ lies in the
$r=0$ locus of $U_\la$, since $(u_{i_0},u_{i_0})\ne 0$.
The argument that exhibits precisely $2^{\#\{i\le j|\la_i=d-i\le k\}}$
isotropic $\Sigma_j\subset B_j$ such that
$\dim(\Sigma_j\cap B_i)\ge i$ for $i\le j$ goes through for $j\le m$,
leading to $\Sigma=\Sigma_m\in X_\la$.
The argument when $\rho_{d-1}\not\subset\la$ is unchanged except
that $i$ ranges from $j+1$ to $m$ in the construction of $\Sigma'_i$,
and then $\Sigma'=\Sigma'_m$ lies in
$X_\la\smallsetminus X^\circ_\la$.
\end{proof}

To state our main theorem relating Gromov--Witten invariants to
classical intersection numbers, we need some further notation.  We let
$Z^\circ_d$ be the $r=2$ locus of $Y_d$.  Define $Z_d$ to be the
variety parametrizing triples $(A,N,B)$ of subspaces such that $A
\subset N \subset B \subset N^\perp$, $\dim A = m-d$, $\dim N =
m-d+2$, and $\dim B = m+d$.  For $\lambda\in \cP(k,n)$, define a
subvariety $Z_\lambda \subset Z_d$ by setting
\[
 Z_\lambda(E_\bull) = \{ (A,N,B) \in Z_d \mid \exists\, \Sigma \in
   X_\lambda(E_\bull) : A \subset \Sigma \subset B \} \,. 
\]
Notice that there is a dense open subset of $Z_d$
(where $\dim(B\cap B^\perp)=\dim(A)+2$) that is
isomorphic to $Z^\circ_d$.

Let $Z'_d$ denote the variety parametrizing $4$-tuples $(A,A',N,B)$ of
subspaces of $V$ with $(A,N,B)\in Z_d$ and $A'$ isotropic with $A
\subset A'\subset B$ and $\dim A' = \dim A+1$. There is a
projection map $\psi:Z'_d\to Z_d$. For $\lambda\in \cP(k,n)$,
define a subvariety $Z'_\lambda \subset Z'_d$ as the locus of
$(A,A',N,B)$ such that there exists a $\Sigma\in X_\la$ with
$A'\subset \Sigma\subset B$.

In contrast to the symplectic case, the map $T_\la\to Y_\la$ can be
finite-to-one for orthogonal Grassmannians, so it is important to
distinguish between the image $Y_\la=\pi(T_\la)$ and the image class
$\pi_*[T_\la]$.  For this purpose we set $\upsilon_\la=\pi_*[T_\la]$
and $\upsilon'_\la=\pi'_*[T'_\la]$.  We analogously define $\zeta_\la$
and $\zeta'_\la$ to be the image classes, taking values in the
homology (or Chow groups) of $Z_d$ and $Z'_d$ respectively, of the
variety of $(A,N,\Sigma,B)$ and $(A,A',N,\Sigma,B)$ respectively, with
$\Sigma\in X_\la$.  The function sending a rational map $f:\bP^1\to
\OG$ to the pair $(\Ker(f), \Span(f))$ in general will be
$2^{M_d}$-to-$1$, where $M_d=M_d(\la,\mu,\nu)$ is a multiplicity
defined by
\begin{equation}
\label{eq:defnMd}
M_d(\la^1,\la^2,\la^3)=\sum_{i=1}^3 N_d(\la^i)-\left\{
\begin{array}{ll}
\min\bigl(\#\{\,i\,|\,N_d(\la^i)\ge 1\,\}, 2\bigr) & \text{if $d\le m$,}\\
0 & \text{if $d=m+1$.}
\end{array}\right.
\end{equation}

For each pair $(A,B)\in Y_d$, we identify the Grassmannian
$\ov{\G}(d,B/A)$ with the locus of isotropic $\Sigma$ of dimension $m$
such that $A\subset \Sigma \subset B$.
Let $r(d)$ be $0$ when $d$ is even and $2$ when $d$ is odd.  When
$d\leq m$, let $S_d=S_d(\la,\mu,\nu)$ be the set of pairs $(A,B)$ in the
$r=r(d)$ locus of $Y_\la\cap Y_\mu\cap Y_\nu$ such that if two or
three among the sets $X_\la\cap \overline{\G}(d,B/A)$, $X_\mu\cap
\overline{\G}(d,B/A)$, $X_\nu\cap \overline{\G}(d,B/A)$ have
cardinality one, the corresponding two or three subspaces of $B/A$ are
in pairwise general position.  When $d=m+1$ is even, let
$S_d=Y_\la\cap Y_\mu\cap Y_\nu$, and when $d=m+1$ is odd,
let $S_d=\emptyset$.

\begin{thm} \label{T:qclasB}
  Let $d\gequ 0$ and choose $\lambda,\mu,\nu \in \cP(k,n)$ such that
  $|\lambda|+|\mu|+|\nu| = \dim(\OG) + d(n+k)$.  Let $X_\lambda$,
  $X_\mu$, and $X_\nu$ be Schubert varieties of $\OG(m,2n+1)$ in
  general position, with associated subvarieties and classes
  in $Y_d$, $Y'_d$, $Z_d$, and $Z'_d$.

  \smallskip \noin {\rm (a)} The subvarieties $Y_\la$, $Y_\mu$, and
   $Y_\nu$ intersect transversally in $Y_d$, and their intersection is
   finite. For each point $(A,B)\in Y_\la \cap Y_\mu \cap Y_\nu$ we 
   have $A = B \cap B^\perp$ or $\dim(B\cap B^\perp) = \dim A + 2$.

  \smallskip \noin {\rm (b)} The assignment $f \mapsto
   (\Ker(f),\Span(f))$ gives a $2^{M_d}$-to-$1$ association between rational
   maps $f:\bP^1\to\OG$ of degree $d$ such that $f(0)\in X_\lambda$,
   $f(1)\in X_\mu$, $f(\infty)\in X_\nu$ and the subset $S_d$ of 
   $Y_\lambda \cap Y_\mu \cap Y_\nu$.

   \smallskip \noin {\rm (c)} When $d \lequ m$ is even, the 
   Gromov--Witten invariant
   $\gw{\ta_\lambda,\ta_\mu,\ta_\nu}{d}$ is equal to
\begin{align*}
   & \int_{Y_d}\upsilon_\la\cdot \upsilon_\mu\cdot \upsilon_\nu
    -\frac{1}{2} 
   \int_{Y'_d} (\varphi^*\upsilon_\la\cdot \upsilon'_\mu\cdot \upsilon'_\nu +
   \upsilon'_\la\cdot \varphi^*\upsilon_\mu\cdot \upsilon'_\nu +
   \upsilon'_\la\cdot \upsilon'_\mu\cdot \varphi^*\upsilon_\nu) \\
   & {} - \int_{Z_d}\zeta_\la\cdot \zeta_\mu\cdot\zeta_\nu
   + \frac{1}{2} \int_{Z'_d} (\psi^*\zeta_\la\cdot \zeta'_\mu\cdot \zeta'_\nu +
   \zeta'_\la\cdot \psi^*\zeta_\mu\cdot \zeta'_\nu +
   \zeta'_\la\cdot \zeta'_\mu\cdot \psi^*\zeta_\nu).
\end{align*}
When $d \lequ m$ is odd, the Gromov--Witten invariant
   $\gw{\ta_\lambda,\ta_\mu,\ta_\nu}{d}$ is equal to
\begin{align*}
   &  \int_{Z_d}\zeta_\la\cdot \zeta_\mu\cdot\zeta_\nu
   - \frac{1}{2} \int_{Z'_d} (\psi^*\zeta_\la\cdot \zeta'_\mu\cdot \zeta'_\nu +
   \zeta'_\la\cdot \psi^*\zeta_\mu\cdot \zeta'_\nu +
   \zeta'_\la\cdot \zeta'_\mu\cdot \psi^*\zeta_\nu).
\end{align*}
When $d=m+1$ is even, we have 
\[
\gw{\ta_\lambda,\ta_\mu,\ta_\nu}{d} = \int_{\G(2m+1,2n+1)}
  \upsilon_\lambda \cdot \upsilon_\mu \cdot \upsilon_\nu .
\]

  \smallskip \noin {\rm (d)} We have
   $\gw{\ta_\lambda,\ta_\mu,\ta_\nu}{d}=0$ if $\la$ does not contain
   $\rho_{d-1}$ when $d$ is even, or does not contain
   $(\rho_{d-1},1)$ when $d$ is odd.
\end{thm}
\begin{proof}
Given integers $0\lequ e_1\lequ m$, $0 \lequ e_2 \lequ d \lequ m+1$,
and $r\gequ 0$, we let $Y^r_{e_1,e_2}$ be the variety of pairs $(A,B)$
such that $A \subset B \subset A^\perp \subset V$, $\dim(A) = m-e_1$,
$\dim(B) = m+e_2$, and $\dim(B \cap B^\perp) = m-e_1+r$. In general,
$Y^r_{e_1,e_2}$ is empty unless $r \lequ \min(e_1+e_2,2k+e_1-e_2)$.
These varieties have a transitive action of $\SO_{2n+1}$, and, when
$d\lequ m$, $Y_d$ is the union of the varieties $Y^r_{d,d}$ for all
numbers $r$ with $0 \lequ r \lequ \min(2d,2k)$. Set
\[ Y^r_\lambda = \{ (A,B) \in Y^r_{e_1,e_2} \mid \exists\, \Sigma \in 
   X_\lambda : A \subset \Sigma \subset B \} \,.
\]
We claim that $Y^r_\lambda \cap Y^r_\mu \cap Y^r_\nu$ is empty unless
(i) $e_1=e_2=d$ and $r\in\{0,2\}$, or (ii) $e_1=m$, $e_2=d=m+1$, and
$r=0$. In both of these cases, $Y^r_\lambda \cap Y^r_\mu \cap Y^r_\nu$
will be a finite set of points.

Following the proof of Theorem \ref{T:qclasC} we compute that the dimension
of $Y^r_{e_1,e_2}$ is equal to  
\[ \frac{1}{2}
   (n^2 + 2nk -3k^2 + n-k + 2ne_1-2ke_1-e_1^2+4e_2k-2e_2^2+e_1+2e_2-r-r^2) \,.
\]
Define the variety $T^r = \{ (A,\Sigma,B) \mid (A,B) \in
Y^r_{e_1,e_2}, \Sigma \in \OG, A \subset \Sigma \subset B \}$, and
let $\pi_1 : T^r \to \OG$ and $\pi_2 : T^r \to Y^r_{e_1,e_2}$ be the
projections.  Then $Y^r_\lambda = \pi_2(\pi_1^{-1}(X_\lambda))$.
Given a point $(A,B) \in Y^r_{e_1,e_2}$, the fiber $\pi_2^{-1}((A,B))$
is the union of the varieties
\begin{equation}
\label{Qs}
 Q_s = \{ \Sigma \in \OG \mid A \subset \Sigma \subset B,\ 
   \dim(\Sigma\cap B \cap B^\perp) = m-e_1+s \}
\end{equation}
for all integers $s$. We have that $Q_s$ is empty unless $2s \gequ
r+e_1-e_2$.  The map $Q_s \to \OG(e_1-s, B/(B\cap B^\perp))$
given by $\Sigma \mapsto (\Sigma + (B\cap B^\perp))/(B\cap B^\perp)$
has fibers isomorphic to open subsets of $\G(e_1,e_1+r-s)$. It follows 
that 
\[ \dim Q_s = \frac{1}{2}
   (2e_1 e_2 -e_1^2 - e_1 + 2e_1 s -2 e_2 s +2rs-3s^2+ s) \,,
\]
whenever $Q_s$ is not empty.

If  $s$ is such that $Q_s$ has the same dimension as the fibers of
$\pi_2$, then the codimension of $Y^r_\lambda$ in $Y^r_{e_1,e_2}$ is
at least $|\lambda| - \dim Q_s$.  It follows that
$\codim(Y^r_\lambda)+\codim(Y^r_\mu)+\codim(Y^r_\nu) -
\dim(Y^r_{e_1,e_2})$ is greater than or equal to the number
$\Delta'(e_1,e_2,r,s) := \dim \OG + d(n+k) - \dim(Y^r_{e_1,e_2}) -
3\dim Q_s$.  A computation shows that $\Delta'(e_1,e_2,r,s)$ is equal
to
\[ (r-3s)(r-3s+1)/2 + (n+k)(d-e_2) + 
   (n-k-1+e_2-2e_1+3s)(e_2-e_1) \,. 
\]

Suppose $Y^r_\lambda \cap Y^r_\mu \cap Y^r_\nu \neq \emptyset$; then 
the Kleiman--Bertini theorem implies that $\Delta'(e_1,e_2,r,s) \lequ 0$. If
$e_2 > e_1$ then we must have $e_2=e_1+1=d=m+1$ and $r=s=0$. 
When $e_2\lequ e_1$
an argument as in the proof of Theorem \ref{T:qclasC} shows that 
$e_1=e_2=d$. But then $2\Delta'(d,d,r,s) = (r-3s)(r-3s+1) \lequ 0$, 
which implies that $r=3s$ or $r=3s-1$. Combining this with $2s\gequ r$, we 
get $(r,s)=(0,0)$ or $(r,s)=(2,1)$.
So $Y_\la\cap Y_\mu\cap Y_\nu$ is finite and
contained in the union of the $r=0$ and $r=2$ loci, and
$Y^0_\la\cap Y^0_\mu\cap Y^0_\nu$ and
$Y_\la\cap Y_\mu\cap Y_\nu\cap Z^\circ_d$ are smooth.
The transversality assertion (a) is, however, stronger.
Its proof uses the following result.

\begin{lemma}
\label{l:genunram}
Let $f\colon Y\to X$ be a finite flat morphism of schemes of finite
type over an algebraically closed field $k$.  Suppose that for any 
closed points $y\in Y$ and $x\in X$ with $f(y)=x$ there exists a
(smooth, connected, quasi-projective) pointed curve $(C,p)$ and a map
$C\to X$ sending $p$ to $x$ with the property that the projection
$\pi\colon C\times_XY\to C$ admits a section $s$ such that the
restriction of $\pi$ to $s(C\smallsetminus\{p\})$ is unramified.  Then
$f$ is unramified over a nonempty open subset of $X$.
\end{lemma}
\begin{proof}
Denote by $d$ the degree of $f$, and for $x\in X$ let
$n_f(x)$ be the number of geometric points of $f^{-1}(x)$.
By \cite[IV.15.5.1]{EGA} the function $n_f$ is lower semi-continuous.
The morphism $f$ is unramified over a nonempty open subset of $X$
if and only if $n_f(x)=d$ for some $x\in X$.
So, suppose that the largest value of $n_f$ is smaller than $d$,
and let $x\in X$ be a closed point at which $n_f(x)$ achieves this value.
Since $n_f(x)<d$ the fiber $f^{-1}(x)$ contains a non-reduced point $y$.
Let $C\to X$, with $\pi$ and $s$, be as in the statement of the lemma.
Since $C$ is smooth and $\pi$ is flat, the irreducible
component $s(C)$ of $C\times_XY$ must intersect some other irreducible
component.
Letting $D$ denote the scheme-theoretic closure of
$C\times_XY\smallsetminus s(C)$ in $C\times_XY$, we have a morphism
$C\amalg D\to C\times_XY$ of schemes flat over $C$, that is an
isomorphism over $C\smallsetminus\{p\}$.
Denoting by $\psi$ the morphism $D\to C$ we have $n_\psi(p)=n_\pi(p)$,
hence for general $p'\in C$ we have
$n_\pi(p')=1+n_\psi(p')\ge 1+n_\psi(p)=1+n_\pi(p)$.
So there exists $x'\in X$ with $n_f(x')>n_f(x)$, and we have reached
a contradiction.
\end{proof}

We observe that $Y_{(\rho_{d-1},1)}(E_\bull)=Y_{(1^d)}(E_\bull)$.
One containment is obvious.
The equality now follows from the fact that
$Y_{(\rho_{d-1},1)}(E_\bull)$ is a variety of
codimension $1$ in $Y_d$, while
$Y_{(1^d)}(E_\bull)$ is a variety but is not the whole of $Y_d$,
since for general $(A,B)$ we have $\dim(B\cap E_{n+k+d})=2d-1$ with
$(\ ,\,)$ restricted to $B\cap E_{n+k+d}$ nondegenerate,
so there can be no isotropic subspace of $B\cap E_{n+k+d}$ of dimension $d$.
We have containments
$$Z^\circ_d\subset Y_{(1^d)}(E_\bull)\subset Y_d$$
and hence an induced morphism of conormal bundles
$$\mathcal{N}_{Y_{(1^d)}/Y_d}|_{Z^\circ_d}\to
\mathcal{N}_{Z^\circ_d/Y_d}.$$
Let us define $W(E_\bull)$ to be the open subset of $Z^\circ_d$
consisting of $(A,B)$ where $A\cap E_{n+k+d+1}=0$ and
the morphism of conormal bundles is nondegenerate.

Fix $\la^0=(\rho_{d-1},1)$, and set $p^0_j=\bar p_j(\la^0)$ for $1\le j\le m$.
Fix some permutation in
${\mathfrak S}_d$ satisfying the conditions stated in
the proof of Lemma \ref{L:fiberB}.
Then the proof explicitly constructs a point
$(A,B)$ in the $r=2$ locus of $U_{\la^0}$, such that
$\Sigma:=\Span(e_{p_1},\ldots,e_{p_m})$ is contained in $B$, and
$A\cap E_{n+k+d+1}=0$.
Since $\Sigma\in X^\circ_{\la^0}$ and $T_{\la^0}\to Y_{\la^0}$ is \'etale
over $U_{\la^0}$, the point $(A,B)$ is a nonsingular point of
$Y_{\la^0}=Y_{(1^d)}$.
Hence we have $(A,B)\in W(E_\bull)$.
Given an arbitrary orthogonal basis $e'_1$, $\ldots$, $e'_{2n+1}$
of $V$, with corresponding isotropic flag $E'_\bull$,
the same construction yields a point
\begin{equation}
\label{ABinWE}
(A',B')\in W(E'_\bull).
\end{equation}

Now suppose $\la\supset\la^0$, and set
$p_j=\bar p_j(\la)$ for $1\le j\le m$.
Since $\ell(\la)\ge d$, there exists an orthogonal basis
$e'_1$, $\ldots$, $e'_{2n+1}$ consisting of basis vectors of $V$, such that
\begin{alignat*}{2}
e'_{p^0_j}&=e_{p_j},&&\quad\text{for $1\le j\le d$}, \\
e'_{n+k+d+1}&=e_c,&&\quad\text{for some $c>p_d$}, \\
e'_{n+k+d+1+j}&=e_{p_{d+j}},&&\quad\text{for $1\le j\le m-d$}.
\end{alignat*}
Given such a basis, with corresponding isotropic flag $E'_\bull=E'_\bull(\la)$,
we have
\begin{align}
(A',B') &\in Y_\la(E_\bull),
\label{etildebas1} \\
Y_\la(E_\bull) &\subset Y_{(1^d)}(E'_\bull),
\label{etildebas2}
\end{align}
the latter a consequence of the fact that $E_{p_d}\subset E'_{n+k+d}$.
Notice that $Y_\la(E_\bull)\cap W(E'_\bull)$ is an open subset
of the $r=2$ locus on $Y_\la(E_\bull)$, and is nonempty by
\eqref{ABinWE} and \eqref{etildebas1}.

In the situation where
$Y^2_\la\cap Y^2_\mu\cap Y^2_\nu\ne\emptyset$
we must have $\la\supset\la^0$ by Lemma \ref{L:fiberB} (b).
Letting $Z^\circ_\la(E_\bull,E'_\bull(\la))=
Y_\la(E_\bull)\cap W(E'_\bull(\la))$, we may assume by
Kleiman--Bertini that
$Y^2_\la\cap Y^2_\mu\cap Y^2_\nu$ is contained in the intersection
of the translates of
$Z^\circ_\la(E_\bull,E'_\bull(\la))$,
$Z^\circ_\mu(E_\bull,E'_\bull(\mu))$, and
$Z^\circ_\nu(E_\bull,E'_\bull(\nu))$;
the assertions thus far concerning
$Y_\la\cap Y_\mu\cap Y_\nu$ are valid on the translates coming from
a nonempty open subset $U\subset \SO_{2n+1}\times \SO_{2n+1}\times \SO_{2n+1}$.
We now verify the condition in Lemma \ref{l:genunram} for the map
$Y\to U$ where $Y$ is the family of intersections of translates of
$Y_\la$, $Y_\mu$, and $Y_\nu$.
Suppose $y\in Y$ maps to $x\in U$, where $x$ and $y$ are closed points.
If $y$ is a reduced point in the fiber over $x$ then we can take
$C\to U$ to be a constant map.
So we suppose $y$ to be a non-reduced point in the fiber over $x$,
which implies that $y$ lies in the $r=2$ locus,
and we proceed to construct $(C,p)$ and $C\to U$ as required.

Without loss of generality the first component of $x$ is the identity
element of $\SO_{2n+1}$.  So we have $(A,B)\in
Z^\circ_\la(E_\bull,E'_\bull(\la))$.  Set $N=B\cap B^\perp\cap
E'_{n+k+d+1}(\la)$.  Since $(A,B)\in W(E'_\bull(\la))$ we have
$\dim(N)=2$.  Let $M$ be a vector space of dimension $4$ containing
$N$ and contained in $A^\perp$, such that the restriction of $(\,,\,)$
to $M$ is nondegenerate.  Then $\C^{2n+1}$ is the direct sum of $M$
and $M^\perp$.

Every element of $\GL_N$ extends uniquely to an element of $\SO_M$,
which we regard as an element of $\SO_{2n+1}$ by letting it act trivially
on $M^\perp$.
This way we may regard $\SL_N$ as a subgroup of $\SO_{2n+1}$.
Now $\SL_N$ acts on $Y_d$ fixing $(A,B)$.
The subvariety $Z^\circ_d$ of $Y_d$ is $\SL_N$-invariant.
So, $\SL_N$ acts on the conormal space to $Z^\circ_d$ in $Y_d$ at
$(A,B)$.
The component of the fixed point locus for the $\SL_N$-action
containing $(A,B)$ is contained in $Z^\circ_d$,
hence the the conormal space to $Z^\circ_d$ at $(A,B)$, when decomposed into
irreducible $\SL_N$-representations,
has no trivial factors.
A three-dimensional representation of $\SL_N$ without
trivial factors is irreducible. From the orbit structure
we know that the conormal vector given by a defining equation
for $Y_{(1^d)}(E'_\bull)$ inside $Y_d$,
which is well-defined up to scale and nonzero since
$(A,B)\in W(E'_\bull(\la))$,
has $\SL_N$-orbit not contained in any proper linear subspace
of $\mathcal{N}_{Z^\circ_d/Y_d,(A,B)}$.

We make a similar analysis of copies of $\SL_2$ sitting inside the
second and third factors of $\SO_{2n+1}$, as above but with $E_\bull$
and $E'_\bull(\la)$ replaced by corresponding translates of $E_\bull$
and $E'_\bull(\mu)$ or $E'_\bull(\nu)$.  So we have $u$, $v$, $w\in
\mathcal{N}_{Z^\circ_d/Y_d,(A,B)}$ and a triple of actions of $\SL_2$
on $\mathcal{N}_{Z^\circ_d/Y_d,(A,B)}$ such that images of $u$, $v$,
and $w$ under a generic triple of group elements span
$\mathcal{N}_{Z^\circ_d/Y_d,(A,B)}$.  We have, then, a map
$$\SL_2\times \SL_2\times \SL_2\to \SO_{2n+1}\times \SO_{2n+1}\times
\SO_{2n+1}$$ sending the identity element to $x$ and a dense open
subset into $U$, such that $(A,B)$ is a point in every fiber of the
image and an unramified point in the fiber over the generic point of
the image (the latter assertion follows from \eqref{etildebas2} and
the fact that the images of $u$, $v$, and $w$ generically span
$\mathcal{N}_{Z^\circ_d/Y_d,(A,B)}$).  Taking $C$ to be a general
curve in $\SL_2\times \SL_2\times \SL_2$ containing the identity, with
$p$ equal to identity, the condition of Lemma \ref{l:genunram} is
satisfied.  So, part (a) is established.

For $d\le m$, we consider a pair
$(A,B)\in Y^r_\lambda \cap Y^r_\mu \cap Y^r_\nu$.
When $r=0$, $\ov{\G}(d,B/A)$ has two connected components. When $r=2$,
let $N=B \cap B^{\perp}$ and call an isotropic $\Sigma$ with
$A\subset\Sigma\subset B$ valid if $\dim(\Sigma\cap N)=1$. The locus
of all valid $\Sigma$ also splits into two connected components,
depending on the family of the isotropic $(d-1)$-dimensional space
$(\Sigma+N)/N$ in $B/N$.  We also call any point of $\ov{\G}(d,B/A)$
valid when $r=0$.  Let $P \in X_\lambda \cap \ov{\G}(d,B/A)$, $Q \in
X_\mu \cap \ov{\G}(d,B/A)$, and $R \in X_\nu \cap \ov{\G}(d,B/A)$. It
is easy to construct examples where $P\cap Q \neq A$; however a
dimension counting argument as in the proof of Theorem \ref{T:qclasC}
shows that $0\leq \dim(P\cap Q) - \dim(A) \leq 1$.

Consider a morphism $f : \bP^1 \to \OG$ of degree $d$ such that $f(0)
\in X_\lambda$, $f(1) \in X_\mu$, and $f(\infty) \in X_\nu$. Then
$(\Ker(f),\Span(f))$ must be a point of an intersection $Y^r_\lambda
\cap Y^r_\mu \cap Y^r_\nu$ in some space $Y^r_{e_1,e_2}$; thus $d\lequ
m+1$. Moreover, we must have $(e_1,e_2)=(d,d)$ and
$(r,s)\in\{(0,0),(2,1)\}$, if $d\lequ m$, or $(e_1,e_2)=(m,m+1)$ and
$r=s=0$, if $d=m+1$. The next result can be justified by an argument similar 
to the proof of \cite[Lemma 1]{Buch}. 
\begin{lemma}
\label{genpos}
Let $f : \bP^1 \to \G(c,c+d)$ be a morphism of degree $d$.  Assume
that $c \lequ d$, $\Ker(f) = 0$, and $\Span(f) = \C^{c+d}$.  
Then any two distinct $c$-dimensional subspaces in the
image of $f$ are in general position.
\end{lemma}

When $d=m+1$ is odd, then $\ov{\G}(d-1,B/A)$ has no rational curves of
degree $d$, and statements (b) and (d) in this case reduce to
trivialities.  When $d\le m$ and $r\ne r(d)$ (that is, $d$ is odd and
$r=0$ or $d$ is even and $r=2$) then there exists no triple of
distinct valid points of $\ov{\G}(d,B/A)$ in general position.  By
Lemma \ref{genpos}, now, the assignment in (b) maps any $f$ to a point
$(\Ker(f),\Span(f))$ in $S_d(\la,\mu,\nu)$.

Suppose $d\le m$.
We claim that the set $S_d$ in the statement of (b) is precisely the
set of points $(A,B)$ in $Y_\la\cap Y_\mu\cap Y_\nu$ for which
there exists a triple of points in $\ov{\G}(d,B/A)$ in
general position,
one in each Schubert variety $X_\la$, $X_\mu$, $X_\nu$, and that
the number of such triples is precisely $2^{M_d}$.
This claim follows easily, once we observe that when $N_d(\la)\ge 1$
the $2^{N_d}$ points of $X_\la\cap \ov{\G}(d,B/A)$ are distributed
equally amongst the two components of valid $\Sigma\in \ov{\G}(d,B/A)$.
Now, part (b) is established by invoking Proposition \ref{Gd2d}.

When $d=m+1$ is even, then $\ov{\G}(d-1,B/A)\cong \OG(d,2d)$
and all triples are in general position.
The number of triples is $2^{M_d}$ by Lemma \ref{L:fiberB} (a).
So (b) is true in this case.
Statement (c) in this case follows immediately.

In case $d\le m$,
statement (c) follows from (b) by keeping track of the possible distributions
of triples among the two connected components of valid
points of $\ov{\G}(d,B/A)$, and by noting that the degrees in
Lemma \ref{L:fiberB} parts (a) and (b) are equal, and the degrees in
Lemma \ref{L:fiberB} parts (c) and (d) are equal.
Statement (d) follows from Lemma \ref{L:fiberB}.
\end{proof}

\begin{example}
On $\OG(4,11)$ we have $\langle \ta_{643}, \ta_{6431},
\ta_{6432}\rangle_4 = 1$.  For general isotropic flags $F_\bull$,
$G_\bull$, the intersection $Y_{6431}(F.)\cap Y_{6432}(G.)$ has two
irreducible components, one with the (generically) unique lifts to
$X_{6431}(F_\bull)$ and to $X_{6432}(G_\bull)$ in general position,
and one with the lifts both containing $F_7\cap G_5$.  A general
translate of $Y_{643}$ intersects each of these components in a single
point, hence $\#(Y_\la\cap Y_\mu\cap Y_\nu)=2$ and (since
$N_4(643)=1$) we have 
\[
\int_{Y_4} \upsilon_{643}\cdot
\upsilon_{6431}\cdot \upsilon_{6432}=4.  
\]
The set $S_d(643,6431,6432)$ is just a single point, and the formula in
Theorem \ref{T:qclasB}(c) gives $\langle \ta_{643}, \ta_{6431},
\ta_{6432}\rangle_4=4-(1/2)\cdot 6$.  (Notice, since $\ell(643)<4$,
the intersection $Y_\la\cap Y_\mu\cap Y_\nu$ is contained in the $r=0$
locus, and there is no contribution from the integrals over $Z_4$ and
$Z'_4$.)
\end{example}

We have seen that the parameter space of lines on $\OG$ is 
the orthogonal two-step flag variety 
$Z_1=\OF(m-1,m+1;2n+1)$. It follows that 
\begin{equation}
\label{oglinenums}
\langle \ta_{\lambda},\ta_\mu, \ta_{\nu}\rangle_1 = 
\int_{Z_1} [Z_{\lambda}]\cdot [Z_\mu] \cdot [Z_\nu]
\end{equation}
where $Z_{\lambda}$, $Z_\mu$, and $Z_\nu$ are the 
associated Schubert varieties in $Z_1$. 

For any Young diagram $\lambda$, let $\ov{\la}$ denote the diagram
obtained by deleting the leftmost column of $\lambda$; hence
$\ov{\la}_i=\max\{\la_i-1,0\}$.  Given any Schubert variety
$X_{\la}$ in $\OG(m,V)$, we will consider an associated
Schubert variety $X_{\ov{\la}}$ in $\OG(m+1,V)$, with
cohomology class $\ta_{\ov{\la}}$. The following result characterizes
the degree $1$ quantum correction terms in the quantum Pieri rule, and
is an exact analogue of what happens for type A Grassmannians.  What
is different from type A is that there will also be degree $2$ quantum
correction terms in the Pieri rule for $\OG$.

\begin{prop}
\label{bprop0}
For any integer $p\in [1,n+k]$ and 
$\lambda$, $\mu\in \mathcal{P}(k,n)$
with $|\lambda|+|\mu|+p=\dim \OG + n + k$, we have
\begin{equation}
\label{beq1}
\langle \ta_\lambda, \ta_{\mu}, \ta_p\rangle_1 =
\int_{\OG(n-k+1,2n+1)} \ta_{\ov{\lambda}} \cdot \ta_{\ov{\mu}} \cdot
\ta_{p-1}.
\end{equation}
\end{prop}
\begin{proof}
Consider the three-step flag variety $U= \OF(m-1,m, m+1; 2n+1)$, with
its natural projections $\pi_1:U\to \OG$ and $\pi_2:U\to Z_1$.  Note
that for every $\lambda\in \cP(k,n)$ and isotropic flag $F_\bull$, we
have $Z_\la= \pi_2(\pi_1^{-1}(X_\la))$. We have the following
commutative diagram
\[
\xymatrix@M=6pt{
  U\ar[d]^{\pi_2} \ar[r]^(0.3){\varphi_2}  &  \OF(m,m+1;2n+1) 
       \ar[d]^{\pi}  \ar[r]^(0.55){\varphi_1} & \OG(m,2n+1) \\
  Z_1 \ar[r]^(0.3){\xi}  &  \OG(m+1,2n+1)  &  }
\]
where every arrow is a natural smooth projection and 
$\pi_1=\varphi_1\varphi_2$. Observe that for each $\la$ we have 
$\pi(\varphi_1^{-1}(X_{\la}(F_\bull))) = X_{\ov{\la}}(F_\bull)$.
We now apply (\ref{oglinenums}) 
and use the projection formula repeatedly to obtain
\begin{align*}
\langle \ta_{\la},\ta_{\mu}, \ta_{p}\rangle_1 & =
\int_{Z_1} [Z_{\lambda}]\cdot [Z_{\mu}]\cdot [Z_{(p)}] \\
& =  \int_{Z_1} \pi_{2*}\pi_1^*\ta_\la \cdot \pi_{2*}\pi_1^*\ta_\mu
\cdot \xi^*\ta_{p-1} \\
& =  \int_U\pi_2^*\pi_{2*}\pi_1^*\ta_\la \cdot \pi_1^*\ta_\mu \cdot
\varphi_2^*\pi^*\ta_{p-1} \\
& =  \int_{\OF(m,m+1;2n+1)} \varphi_{2*}\pi_2^*\pi_{2*}\pi_1^*\ta_\la \cdot
\varphi_1^*\ta_\mu \cdot \pi^*\ta_{p-1} \\
& =  \int_{\OF(m,m+1;2n+1)} \pi^*\pi_*\varphi_1^*\ta_\la\cdot
\varphi_1^*\ta_\mu \cdot \pi^*\ta_{p-1} \\
& =  \int_{\OG(m+1;2n+1)} \pi_*\varphi_1^*\ta_\la\cdot \pi_*
\varphi_1^*\ta_\mu \cdot \ta_{p-1}  \\
& =  \int_{\OG(m+1,2n+1)}
\ta_{\ov{\la}} \cdot \ta_{\ov{\mu}}\cdot \ta_{\ov{p}}.  \qedhere
\end{align*}
\end{proof}

\subsection{Quantum cohomology}
\label{qcOGodd}

The quantum cohomology ring $\QH^*(\OG)$ is defined similarly to that
for $\IG$, but the degree of $q$ here is $n+k$. The quantum Pieri
formula for non-maximal orthogonal Grassmannians involves both linear
and quadratic $q$ terms. This is surprising since the corresponding
Pieri rule for the maximal isotropic Grassmannian $\OG(n,2n+1)$ (which
is isomorphic to $\OG(n+1,2n+2)$) has only linear $q$ terms. Note that
the degree of $q$ on $\OG(n,2n+1)$ is $2n$, which is twice the
expected value, since its degree on $\OG(n-k,2n+1)$ equals $n+k$ when
$k>0$.  In fact, when $k>0$, the $q^2$ terms in the quantum Pieri
formula for $\OG(n-k,2n+1)$ behave like the $q$ terms in the analogous
formula for $\OG(n,2n+1)$, while the $q$ terms on $\OG(n-k,2n+1)$
behave (and may be computed) like the $q$ terms in the quantum Pieri
rule for the usual (type A) Grassmannian.

Geometrically, this jump in $q$-degree is explained by the degree
doubling phenomenon on the maximal orthogonal Grassmannian, whereby
lines are mapped to conics in projective space under the Pl\"ucker
embedding (since the first Chern class of the universal quotient
bundle $\cQ$ is twice a Schubert class).  Algebraically, it is
simplest to understand the transition between the maximal and
non-maximal isotropic cases by considering the variety
$\OG(n+1-k,2n+2)$ for $k\gequ 0$, as explained in the 
Introduction.
Formally, the single quantum parameter $q$ on $\OG(n+1,2n+2)$ is
replaced with two square roots $q_1$ and $q_2$ on $\OG(n,2n+2)$, which
in turn are identified on $\OG(n-1,2n+2)$. Compare Theorems
\ref{ogevenqupieri} and \ref{ogevenqupierin} for further details.

To formulate the quantum Pieri rule for $\OG$, we require some more
notation.  Let $\cP'(k,n+1)$ be the set of $\nu\in\cP(k,n+1)$ for
which $\ell(\nu)=n+1-k$, $2k \lequ \nu_1 \lequ n+k$, and the number of
boxes in the second column of $\nu$ is at most $\nu_1-2k+1$.  For any
$\nu\in\cP'(k,n+1)$, we let $\wt\nu \in \cP(k,n)$ be the partition
obtained by removing the first row of $\nu$ as well as $n+k-\nu_1$
boxes from the first column.  That is,
\[
\wt{\nu}=(\nu_2,\nu_3,\ldots,\nu_r), \
\text{where $r=\nu_1-2k+1$.}
\]
Recall also from Section \ref{qcIG} that for any partition $\la$, 
we set $\la^*=(\la_2,\la_3,\ldots)$.

\begin{thm}[Quantum Pieri rule for $\OG$] 
\label{ogoddqupieri}
For any $k$-strict partition 
$\lambda\in\cP(k,n)$ and integer $p\in [1,n+k]$, we have 
\begin{equation}
\label{ogqp}
\ta_p \cdot \ta_\la = 
\sum_{\la\to\mu} 2^{N'(\lambda,\mu)}\,\tau_\mu +\sum_{\la\to\nu} 
2^{N'(\lambda,\nu)} \,\tau_{\wt{\nu}}\, q \, +\, 
\sum_{\la^*\to\rho} 
2^{N'(\lambda^*,\rho)} \,\ta_{\rho^*}\, 
q^2
\end{equation}
in the quantum cohomology ring $\QH^*(\OG(n-k,2n+1))$. Here (i) the
first sum is classical, as in (\ref{ogclass}), (ii) the second is over
$\nu\in \cP'(k,n+1)$ with $\lambda\to\nu$ and $|\nu|=|\la|+p$, and
(iii) the third sum is empty unless $\la_1=n+k$, and over
$\rho\in\cP(k,n)$ such that $\rho_1=n+k$, $\lambda^*\to\rho$, and
$|\rho|=|\la|-n-k+p$.
\end{thm}

Our proof of Theorem \ref{ogoddqupieri} will require several auxiliary
results. Let $\mu^{\vee}$ be the dual partition to $\mu$ in
$\cP(k,n)$. From the the picture of $\mu^\vee$ given in Section
\ref{dualpar} it is straightforward to prove that
$\ell(\mu^{\vee})=n-k$ if and only if $\mu_1 < \ell(\mu)+2k$.
Proposition \ref{bprop0} shows that for $\la,\mu \in\cP(k,n)$, the
coefficient of $\ta_{\mu}\, q$ in the Pieri product $\ta_p\,
\ta_{\la}$ in $\QH^*(\OG(n-k,2n+1))$ is equal to the coefficient of
$\ta_{(\ov{\mu^{\vee}})^{\vee}}$ in the product $\ta_{p-1}\,
\ta_{\ov{\la}}$ in $\HH^*(\OG(n+1-k,2n+1))$, when
$\ell(\la)=\ell(\mu^{\vee})=n-k$, and equals $0$, otherwise.  By
computing the dual partition in $\cP(k-1,n)$ to $\ov{\mu^{\vee}}$, we
deduce the following result.

\begin{prop}
\label{bcor1}
Consider $\lambda$, $\mu\in \cP(k,n)$ with $|\mu|+n+k=|\lambda|+p$,
for $1\lequ p\lequ n+k$. If $\ell(\la)=n-k$ and $\mu_1<\ell(\mu)+2k$
then, in $\QH^*(\OG)$, the coefficient of $\ta_\mu\, q$ in
$\ta_p\,\ta_\lambda$ is equal to the coefficient of
$\ta_{(\ell(\mu)+2k-1,\, \ov{\mu})}$
in $\ta_{p-1}\,\ta_{\ov{\lambda}}\in \HH^*(\OG(n+1-k,2n+1))$.
Otherwise, the coefficient vanishes.
\end{prop}

To characterize the $q^2$ terms that occur in Pieri products we follow
the model of the maximal orthogonal Grassmannian $\OG(n,2n+1)$, and
idea in the proof of \cite[Thm.~6]{BKT}. We will need the next
proposition.

\begin{prop}
\label{bprop2}
Suppose $\lambda\in \cP(k,n)$ with $\lambda_1<n+k$.  Then, for any
$p$, the coefficient of $q^2$ in the quantum Pieri product
$\ta_p\,\ta_\lambda$ vanishes.
\end{prop}

\begin{proof}
For degree reasons we may assume $k\lequ n-2$.
Degree $2$ maps $f:\bP^1\to\OG$ are studied by considering
their image conics $C\subset \OG$. Let $X_p(E_\bull)$,
$X_\lambda(F_\bull)$, and $X_\mu(G_\bull)$ be given, for some
$\mu\in \cP(k,n)$, and consider the following conditions.
\begin{itemize}
\item[(i)] There are only finitely many conics $C$ in $\OG$ incident to
$X_p(E_\bull)$, $X_\lambda(F_\bull)$, and $X_\mu(G_\bull)$;
\item[(ii)] No such conic is incident
to the intersection of any two of
these varieties;
\item[(iii)] Each conic $C$ has
$\dim\Ker(C)=m-2$ and $\dim\Span(C)=m+2$;
\item[(iv)] The induced bilinear form on
$\Span(C)/\Ker(C)$ is always nondegenerate.
\end{itemize}
We claim that conditions (i)--(iv) imply that the set of conics incident to
$X_p(E_\bull)$, $X_\lambda(F_\bull)$, and $X_\mu(G_\bull)$ is empty.
Conditions (i)--(iv) hold in
particular when $|\lambda|+|\mu|+p=\dim \OG + 2(n+k)$ and $E_\bull$,
$F_\bull$, and $G_\bull$ are general (by dimension reasoning,
plus Theorem \ref{T:qclasB}).

The first step in the proof of the claim is a reduction to the case
\[
p=n+k,\qquad \lambda\in\{\lambda^0,\lambda^1\}, \qquad
\mu=(n+k,\ldots,2k+1),
\]
where $\lambda^0=(n+k-1,\ldots,2k)$ and
$\lambda^1=(n+k-1,\ldots,2k+1)$.  Assume (i)--(iv) and let $C$ denote
a conic in $\OG$ satisfying the incidence conditions.  Recall that
$X_p(E_\bull)=\{\,\Sigma\,|\,\Sigma\cap E_{\varepsilon(p)}\ne
0\,\}$.  If $\Pi\in X_p(E_\bull)\cap C$ and $0\ne x\in
\Pi\cap E_{\varepsilon(p)}$, then
$\{\,\Sigma\,|\,x\in\Sigma\,\}$ is a translate of $X_{n+k}$
contained in $X_p(E_\bull)$.  Let $\Sigma \in X_\lambda(F_\bull)\cap C$.
If $\ell(\la)=m$, we define $\widetilde{F}_{m+1}=(\Sigma^\perp\cap
F_i)+\Sigma$ with $i$ chosen so this space has dimension $m+1$, and
then $X_{\lambda^0}(\widetilde{F}_\bull)\subset X_\lambda(F_\bull)$.
If $\ell(\la)<m$, then put $R=\Sigma\cap F_j$ with $j$
chosen so that $\dim R=m-1$, and define
$\widetilde{F}_m=(R^\perp\cap F_i)+R$ for $i$ such that this has
dimension $m$.  Then $X_{\lambda^1}(\widetilde{F}_\bull)\subset
X_\lambda(F_\bull)$.
Set $A=\Ker(C)$ and $B=\Span(C)$.

\emph{Case 1.} Suppose $\lambda=\lambda^0$.
The three Schubert varieties are the point $G_m$,
the locus of $\Sigma$ containing $x$, and
the locus of $\Sigma\subset F_{m+1}$.
The space $B$ contains $x$, $G_m$, and an
$m$-dimensional subspace of $F_{m+1}$ containing $A$.
By conditions (ii) and (iii), the space $(G_m+\langle x\rangle )\cap F_{m+1}$
must contain some vector $t$ not in $A$, and
applying condition (iv) also, we see that
for general $y\in F_{m+1}$, the space
$B':=G_m+\langle x, y\rangle$
has dimension $n-k+2$ with the symmetric form nondegenerate on $B'/A$.
Let $C'$ be the component containing $G_m$
of the space of maximal isotropic spaces containing $A$ and contained in $B'$.
Set $\Sigma'=A+\langle t, y\rangle$; then
$\Sigma'\subset F_{m+1}$ and
$\Sigma'\in C'$, so $C'$ is a conic
incident to the three Schubert varieties.
There are infinitely many such conics $C'$, so we have a contradiction to (i).

\emph{Case 2.} Suppose $\lambda=\lambda^1$.
The three Schubert varieties are the point $G_m$,
the locus of $\Sigma$ containing $x$,
and the locus of $\Sigma$ with $\dim(\Sigma\cap F_m)\gequ m-1$.
Let $\Sigma_1$ be a point of $C$ in the last of these varieties.
We may suppose $\Sigma_1\ne F_m$, for otherwise we can
reduce to Case 1.
So $M:=\Sigma_1\cap F_m$ has dimension $m-1$.
If $A\subset F_m$ then
there exists isotropic $\Sigma'$ of dimension $m$
containing $A+\langle x\rangle$ and contained in
$(x^\perp\cap F_m)+\langle x\rangle$,
and we quickly get a contradiction to (ii).
So $A\not\subset F_m$, and hence $\dim(A\cap M)=m-3$.
If, now, $F_m\subset B$, i.e., $\dim(F_m\cap G_m)=m-2$, then
the component containing $G_m$ of the space of maximal isotropic spaces
containing $F_m\cap G_m$ and contained in $B$ is a conic, and as before
we reduce to Case 1.
So $A_0:=F_m\cap G_m$ has dimension $m-3$.
With $B_0:=F_m+G_m$ there are two possibilities for the bilinear form
on $B_0/A_0$: it could be nondegenerate or it could have 2-dimensional kernel.
In the former case, the infinite family of spaces $A'$ of dimension $m-2$
with $A_0\subset A'\subset G_m\cap x^\perp$ gives rise to an
infinite family of conics $C'$ with $\Ker(C')=A'$ and
$\Span(C')=(A')^\perp\cap B_0$, contradicting (i).
In the latter case there is a family of conics with kernel $A$
and varying span contradicting (i).
\end{proof}

\medskip
\noin
\begin{proof}[Proof of Theorem \ref{ogoddqupieri}]
Theorem \ref{T:qclasB} (d) implies that the quantum Pieri product
$\ta_p\,\ta_{\la}$ contains at most quadratic $q$ terms. We begin by
studying the $q$-linear terms in this product. For dimension reasons,
the right hand side of \eqref{beq1} vanishes when either $\lambda$ or
$\mu$ has length less than $n-k$. 

The classical Pieri rule for $\OG$ (Theorem \ref{thm.ogoddpieri})
implies that for $\lambda$, $\mu\in \cP(k,n)$ with $\ell(\la)=n-k$ and
$\mu_1<\ell(\mu)+2k$, we have $\ov{\la} \to (\ell(\mu)+2k-1,\,
\ov{\mu})$ in $\cP(k-1,n)$ if and only if $\la \to (\ell(\mu)+2k,\mu,
1^{n-k-\ell(\mu)})$ in $\cP(k,n+1)$, with the same coefficients
$N'$. It follows by Proposition \ref{bcor1} that for $\la$,
$\mu\in\cP(k,n)$, the coefficient of $\ta_\mu\, q$ in the quantum
product $\ta_p\,\ta_\lambda$ in $\QH^*(\OG)$ is equal to the
coefficient of $\ta_{(\ell(\mu)+2k,\mu, 1^{n-k-\ell(\mu)})}$ in the
cup product $\ta_p\,\ta_\lambda$ in $\HH^*(\OG(n+1-k,2n+3))$ when
$\ell(\la)=n-k$ and $\mu_1 < \ell(\mu)+2k$, and is 0
otherwise. Observe that the condition $\ell(\la)=n-k$ may be omitted,
since when $\ell(\la)<n-k$, the product $\ta_p\,\ta_\lambda$ in
$\HH^*(\OG(n+1-k,2n+3))$ involves no terms indexed by partitions of
length $n+1-k$. Notice that $\nu \mapsto \wt{\nu}$ induces a 1-1 map
$\cP'(k,n+1)\to\cP(k,n)$ with image $\{\mu\in \cP(k,n)\ :\
\mu_1<\ell(\mu)+2k\}$, and the inverse of this map is given by
$\mu\mapsto (\ell(\mu)+2k,\mu,1^{n-k-\ell(\mu)})$.  Combining these
facts, we see that the coefficient of $\ta_{\wt{\nu}}\, q$ on the
right hand side of \eqref{ogqp} is equal to the coefficient of
$\ta_{\nu}$ in the product $\ta_p\,\ta_{\la}$ in
$\HH^*(\OG(n+1-k,2n+3))$, for $\nu\in \cP'(k,n+1)$, and these are all
the linear $q$ terms.

We next show that the basic relation $\ta_{n+k}^2=q$ holds in
$\QH^*(\OG)$. Note that $\ta_{n+k}^2$ vanishes in cohomology, and the
coefficient of $q$ vanishes by Proposition \ref{bcor1}.  So
$\ta_{n+k}^2=c\,q^2$ for some $c\in \Z$.  That $c=1$ can be shown
geometrically by exhibiting a unique conic on $\OG$ passing through a
point and two general translates of $X_{n+k}$.  An alternative
argument uses associativity of the quantum product.  We have
$\ta_{n+k}\,\ta_{(1^{n-k})}= \ta_{(1^{n-k})}\, q$ by Proposition
\ref{bcor1}.  Hence $\ta_{n+k}^2\,
\ta_{(1^{n-k})}=\ta_{(1^{n-k})}\, q^2$, and $c=1$.

According to Proposition \ref{bcor1} the term $\ta_\mu\, q$ appears in a
Pieri product only when $\mu_1<\ell(\mu)+2k$, and in particular,
$\mu_1<n+k$. Proposition \ref{bprop2} asserts that whenever
$\la_1<n+k$, the product $\ta_p\,\ta_\la$ carries only such degree
one quantum correction terms $\ta_{\mu}\, q$.  One now completes the
proof of the Theorem as follows.  It suffices to consider products
$\ta_p \, \ta_{\la}$ when $\la_1=n+k$. In this case we have an
equation $\ta_\la=\ta_{n+k}\, \ta_{\la^*}$ in $\QH^*(\OG)$. If
$p=n+k$, then $\ta_{n+k}\, \ta_\la= \ta_{n+k}^2\, \ta_{\la^*} =
\ta_{\la^*}\, q^2$, and the quantum Pieri formula is verified.  If $p
< n+k$, then we write $\ta_p\,\ta_{\la} = \ta_{n+k}\, (\ta_p
\, \ta_{\la^*})$. By Proposition \ref{bprop2}, the product in
parentheses receives only linear quantum correction terms, and hence
is known by Proposition \ref{bcor1}. As the quantum Pieri rule for
multiplication by $\ta_{n+k}$ has already been proved, it remains only
to show that the result agrees with the formula in the theorem, and
this is easily checked.
\end{proof}

\begin{example} \label{E:qpieriB}
  In the quantum cohomology ring of $\OG(4,13)$ we have 
 \begin{gather*}
  \tau_4 \cdot
  \tau_{(5,3,2,2)} = 2 \tau_{(8,4,2,2)} + \tau_{(7,5,2,2)} + 2
  \tau_{(7,4,3,2)} + \tau_{(6,5,3,2)} + \tau_{(8,4,3,1)} \\ + 2
  \tau_{(4, 2, 2)}\, q + \tau_{(4, 3, 1)}\, q + 2 \tau_{(3, 2, 2, 1)}\, q 
  + 2 \tau_{(4, 2, 1, 1)}\, q.
  \end{gather*}
On the same Grassmannian we also have
\begin{gather*}
\tau_5 \cdot \tau_{(8,4,1,1)} = \tau_{(8,6,4,1)} + 2
  \tau_{(8,7,3,1)} + \tau_{(8,7,4)} + \tau_{(7,2,1,1)}\, q + 2 
  \tau_{(6,3,1,1)}\, q \\ + \tau_{(5,4,1,1)}\, q + \tau_{(1,1,1)}\, q^2 
  + 2 \tau_{(2,1)}\, q^2 + 4 \tau_{(3)}\, q^2.
\end{gather*}
\end{example}

\begin{thm}[Ring presentation] 
\label{ogoddqpres}
The quantum cohomology ring $\QH^*(\OG)$ is presented as a quotient of the
polynomial ring $\Z[\ta_1,\ldots,\ta_{n+k},q]$ modulo the relations
\begin{gather}
\label{ogoddQR1}
\det(\delta_{1+j-i}\ta_{1+j-i})_{1\lequ i,j \lequ r} = 0, \ \ \ \ 
n-k+1\lequ r \lequ n, \\
\label{ogoddQR1'}
\sum_{p=k+1}^r(-1)^p\tau_p
\det(\delta_{1+j-i}\ta_{1+j-i})_{1\lequ i,j \lequ r-p}=0, \ \ \ \
n+1\lequ r < n+k, \\
\label{ogoddQR1''}
\sum_{p=k+1}^{n+k}(-1)^p\tau_p
\det(\delta_{1+j-i}\ta_{1+j-i})_{1\lequ i,j \lequ n+k-p}=q
\end{gather}
and
\begin{equation}
\label{ogoddQR2}
\ta_r^2 + \sum_{i=1}^r(-1)^i \delta_{r-i}\ta_{r+i}\ta_{r-i}= 0,
 \ \  \ \ k+1\lequ r \lequ n.
\end{equation}
\end{thm}
\begin{proof}
Set $h_r= \det(\delta_{1+j-i}\ta_{1+j-i})_{1\lequ i,j \lequ r}$.  We
have that $h_r=\ta_{(1^r)}$ for $r \lequ n-k$, while $h_r=0$ for
$n-k+1\lequ r < n+k$, because these relations hold classically and the
degree of $q$ is $n+k$. This proves (\ref{ogoddQR1}) and
(\ref{ogoddQR1'}). The quantum Pieri rule implies that for $p>k$,
$\ta_p h_{n+k-p}=0$ unless $p=2k$, when $\ta_{2k}h_{n-k}=q$. We deduce
that (\ref{ogoddQR1''}) also holds in $\QH^*(\OG)$.

We are left with proving that there are no quantum correction terms in
relation (\ref{ogoddQR2}); the result then follows from \cite{ST}.
When $n-k>1$, this is immediate since
the quantum Pieri rule does not give rise to any quantum correction terms.
In the case of the quadric $\OG(1,2n+1)$ the quantum Pieri rule gives
\[
\ta_n^2-\ta_{n+1}\ta_{n-1}+\cdots + (-1)^{n-1} \ta_{2n-1}\ta_1 = c\,\ta_1\,q
\]
with the coefficient $c=1-2+2-\cdots\pm 2\mp 1=0$.
\end{proof}


\section{The Grassmannian $\OG(n+1-k,2n+2)$}
\label{QCOGeven}

\subsection{Schubert classes}
\label{ogevenscs}

In this section, we consider the even orthogonal
Grassmannian $\OG'=\OG(m,2n+2)$, which parametrizes the
$m$-dimensional isotropic subspaces in a 
vector space $V\cong \C^{2n+2}$ with a nondegenerate symmetric
bilinear form. The variety $\OG'$ has dimension $2m(n+1-m)+m(m-1)/2$.

Two subspaces $E$ and $F$ of $V$ are said to be {\em in the same
family} if $\dim(E\cap F)\equiv (n+1) (\text{mod}\, 2)$.  Fix once and
for all an isotropic subspace $L$ of $V$ with $\dim(L)=n+1$. An
isotropic flag is a complete flag $F_\bull$ of subspaces of $V$ such
that $F_{n+1+i} = F_{n+1-i}^\perp$ for all $0 \lequ i \lequ n$, and
$F_{n+1}$ and $L$ are in the same family. As the orthogonal space
$F_n^{\perp}/F_n$ contains only two isotropic lines, to each such flag
$F_\bull$ there corresponds an alternate isotropic flag $\wt{F}_\bull$
such that $\wt{F}_i=F_i$ for all $i\lequ n$ but with $\wt{F}_{n+1}$ in
the opposite family from $F_{n+1}$.

Set $k=n+1-m>0$. The Schubert varieties in $\OG'$ are indexed by a 
set $\wt{\cP}(k,n)$ which differs from that used in previous sections.
To any $k$-strict partition $\la$ we associate a number in $\{0,1,2\}$
called the {\em type} of $\lambda$, denoted $\type(\la)$. If $\lambda$
has no part equal to $k$, then we set $\type(\la)=0$; otherwise we 
have $\type(\la)=1$ or $\type(\la)=2$ (thus `type' is a multi-valued
function). The elements of $\wt{\cP}(k,n)$ are the $k$-strict partitions
contained in an $m\times (n+k)$ rectangle of all three possible types. 
For every $\la\in\wt{\cP}(k,n)$, we define index set 
$P'=\{p'_1<\cdots < p'_m\}\subset [1,2n+2]$ with
\begin{multline*} p'_j(\lambda) = n+k-\lambda_j + 
   \#\{\,i<j\,|\,\lambda_i+\lambda_j \lequ 2k-1+j-i\,\} \\
   {} + \begin{cases} 
      1 & \text{if $\lambda_j > k$, or $\lambda_j=k < \lambda_{j-1}$ and
        $n+j+\type(\lambda)$ is even}, \\
      2 & \text{otherwise}.
   \end{cases}
\end{multline*}

Given any isotropic flag $F_\bull$, each $\la\in\wt{\cP}(k,n)$ 
indexes a codimension $|\la|$ Schubert variety $X_{\la}(F_\bull)$ in 
$\OG'$, defined as the locus of $\Sigma \in \OG'$ such that 
\[
\dim(\Sigma \cap F_{p'_j}) \gequ j, \  \text{if $p'_j\neq n+2$, \ and } 
\dim(\Sigma \cap \wt{F}_{n+1}) \gequ j, \  \text{if $p'_j = n+2$,}
\]
for all $j$ with $1 \lequ j \lequ \ell(\la)$.
For each $\la\in\wt{\cP}(k,n)$, we let $\ta_\la$
denote the cohomology class in $\HH^{2|\la|}(\OG',\Z)$ dual to the cycle
determined by the Schubert variety indexed by $\la$.  Each such
Schubert class has a type which is the same as the type of $\la$; this
serves to distinguish two separate classes for each partition $\la$
with some part $\la_i=k$. 

\subsection{Classical Pieri rule}
\label{ogevenclpieri}

The {\em special} Schubert varieties for $\OG(n+1-k,2n+2)$ are the 
varieties $X_1, \ldots, X_{k-1}, X_k, X'_k, X_{k+1},\ldots, X_{n+k}$.
These are defined by a single Schubert condition as follows. For $p\neq k$,
we have
$$X_p(F_\bull)=
\{\,\Sigma\in \OG'\,|\,\Sigma\cap F_{\varepsilon(p)}\ne 0\,\}$$
where $\varepsilon(p)$ is given by (\ref{vare}). If $n$ is even, then 
\[
X_k(F_\bull)= \{\,\Sigma\in \OG'\,|\,\Sigma\cap F_{n+1}\ne
0\,\}
\]
and
\[
X'_k(F_\bull)= \{\,\Sigma\in \OG'\,|\,\Sigma\cap \wt{F}_{n+1}\ne
0\,\},
\]
while the roles of $F_{n+1}$ and $\wt{F}_{n+1}$ are switched if $n$ is
odd.  We let $\ta_p$ denote the cohomology class of $X_p(F_\bull)$ for
$1\lequ p \lequ n+k$ and $\ta'_k$ denote the cohomology class of
$X'_k(F_\bull)$; note that $\type(\tau_k)=1$ and $\type(\tau'_k)=2$.

The Pieri rule for $\OG'$ requires a slightly different shifting
convention than that used for $\IG$ and $\OG$. Given a $k$-strict
partition $\la$, we say that the box in row $r$ and column $c$ of
$\la$ is {\em $k'$-related} to the box in row $r'$ and column $c'$ if
$|c-(2k+1)/2|+r=|c'-(2k+1)/2|+r'$. Using this convention, the relation
$\lambda\to \mu$ is defined as in Definition~\ref{D:pieriarrow}, with
the added condition that $\type(\la)+ \type(\mu)\neq 3$. Moreover, the
multiplicity $N'(\la,\mu)$ is equal to the number (respectively, one
less than the number) of components of $\A$, if $p\lequ k$
(respectively, if $p>k$).

Let $g(\la,\mu)$ be the number of columns of $\mu$ among 
the first $k$ which do not have more boxes than the corresponding column
of $\la$, and 
\[
h(\la,\mu) = g(\la,\mu)+\max(\type(\la),\type(\mu)).
\]
If $p\neq k$, then set $\delta_{\la\mu}=1$. If $p=k$ and 
$N'(\la, \mu)>0$, then set
\[
\delta_{\la\mu}=\delta'_{\la \mu}=1/2,
\]
while if $N'(\la, \mu)=0$, define
\[
\delta_{\la \mu} = \begin{cases} 
1 & \text{if $h(\la,\mu)$ is odd}, \\
0 & \text{otherwise}
\end{cases}
\qquad \mathrm{and}
\qquad 
\delta'_{\la\mu} = \begin{cases}
1 & \text{if $h(\la,\mu)$ is even}, \\
0 & \text{otherwise.}
\end{cases}
\]

\begin{thm}[Pieri rule for $\OG(m,2n+2)$]
\label{thm.ogevenpieri}
For any element $\la\in\wt{\cP}(k,n)$ and integer $p\in [1,n+k]$, we
have
\begin{equation}
\label{ogevenclass}
\ta_p \cdot\ta_{\lambda} = \sum_\mu
\delta_{\la\mu}\,2^{N'(\la,\mu)}\,\ta_\mu,
\end{equation}
where the sum is over all $\mu\in\wt{\cP}(k,n)$ with $\lambda\to\mu$
and $|\mu|=|\la|+p$. Furthermore, the product
$\ta'_k\cdot\ta_{\lambda}$ is obtained by replacing $\delta_{\lambda
\mu}$ with $\delta'_{\lambda \mu}$ throughout.
\end{thm}

This theorem will be proved in Section \ref{Pierirule}.

\begin{example}
  For the Grassmannian $\OG(5,14)$ we have $k=2$ and $n=6$.  For the
  partition $\lambda = (8,7,2,1,1)$ with $\type(\lambda)=1$, we get
  the following partitions $\mu \in \wt{\cP}(2,6)$, all of type 0 or 1,
  such that $\lambda \to \mu$ and $|\mu| = |\lambda|+2$:
\[ \pic{.40}{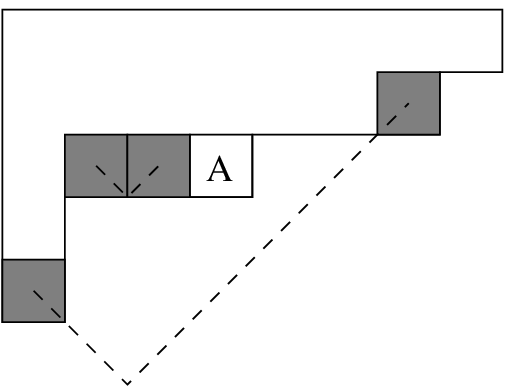}\hmm{5}\pic{.40}{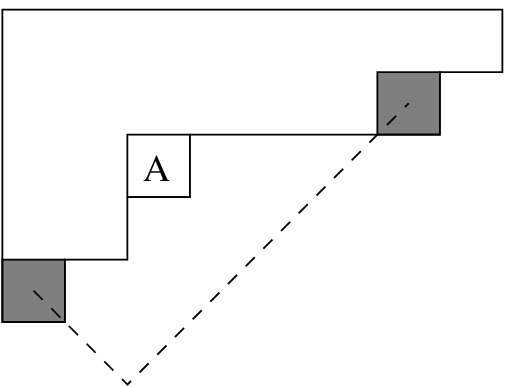}\hmm{5}\pic{.40}{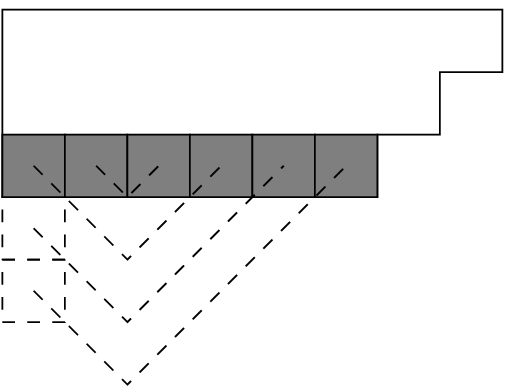} \]
We obtain $\tau_2 \cdot \tau_\lambda = \tau_{(8,7,4,1,1)} +
\tau_{(8,7,3,2,1)} + \tau_{(8,7,6)}$ and $\tau'_2 \cdot \tau_\lambda =
\tau_{(8,7,4,1,1)} + \tau_{(8,7,3,2,1)}$.  Notice that the product
$(\tau_2+\tau'_2) \cdot \tau_\lambda$ is obtained from
(\ref{ogevenclass}) by ignoring $\delta_{\lambda \mu}$.
\end{example}

\subsection{Presentation of $\HH^*(\OG',\Z)$}
\label{ogevenpresentation}

If $\cS$ (respectively $\cQ$) denotes the tautological 
subbundle (respectively, quotient bundle) over $\OG'$,
then one has that
\begin{equation}
\label{ctotau}
c_p(\cQ)=
\begin{cases}
\ta_p &\text{if $p< k$},\\
\ta_k+\ta_k' &\text{if $p=k$},\\
2\ta_p &\text{if $p> k$}.
\end{cases}
\end{equation}
For each $r>0$, let $\Delta_r$ denote the $r\times r$ Schur determinant
\[
\Delta_r=\det(c_{1+j-i})_{1\lequ i,j \lequ r}, 
\]
where each variable $c_p$ represents the Chern class $c_p(\cQ)$. 
For each $\la\in\wt{\cP}(k,n)$ we define a monomial 
$\ta^{\la}$ in the special Schubert classes as follows. If $\la$
is not of type 2, then set $\ta^{\la}=\prod_i\ta_{\la_i}$. If 
$\type(\la)=2$ then $\ta^{\la}$ is defined by the same product formula,
but replacing each occurrence of $\ta_k$ with $\ta'_k$.

\begin{thm}
\label{presOGeven}
{\em a)}
Define polynomials $c_p$ using the equations (\ref{ctotau}). Then
the cohomology ring
$\HH^*(\OG(n+1-k,2n+2),\Z)$ is presented as a quotient of the polynomial
ring $\Z[\ta_1,\ldots,\ta_k,\ta_k',\ta_{k+1},\ldots,\ta_{n+k}]$ 
modulo the relations 
\begin{gather}
\label{ogevenR1}
\Delta_r = 0, \ \ \ \ 
n-k+1 <  r \lequ n, \\
\label{ogevenR1'}
\ta_k\Delta_{n+1-k}=
\ta'_k \Delta_{n+1-k}=
\sum_{p=k+1}^{n+1}(-1)^{p+k+1}\tau_p
\Delta_{n+1-p}, \\
\label{ogevenR1''}
\sum_{p=k+1}^r(-1)^p\tau_p
\Delta_{r-p}=0, \ \ \ \
n+1 < r \lequ n+k,
\end{gather}
and
\begin{gather}
\label{ogevenR2}
\ta_r^2 + \sum_{i=1}^r(-1)^i \ta_{r+i}c_{r-i}= 0,
 \ \  \ \ k+1 \lequ r \lequ n, \\
\label{ogevenR2'}
\ta_k\ta'_k+\sum_{i=1}^k(-1)^i \ta_{k+i}\ta_{k-i}=0.
\end{gather}

\medskip
\noin
{\em b)} The monomials $\ta^{\la}$ with $\la\in\wt{\cP}(k,n)$ 
form a $\Z$-basis for $\HH^*(\OG',\Z)$.
\end{thm}
\begin{proof}
We prove only part (a), as the proof of (b) is similar to that of
Theorem \ref{presIG}.  We first note that $\HH^*(\OG')$ is a free
abelian group of rank $2^{n+1-k}\binom{n+1}{k}$, and argue using the
Pieri rule for $\OG'$ that the special Schubert classes generate the
cohomology ring of $\OG'$ over $\Z$. The Whitney sum formula and the
Pieri rule may be used to show that the displayed relations hold in
$\HH^*(\OG')$, as in the proof of Theorem \ref{presOGodd}.
The first equality in (\ref{ogevenR1'}) also follows from the relation
$(\ta_k-\ta_k')c_{n+1-k}(\cS)=0$; a proof of this is given in
\cite[Sec.\ 6.1]{T}.

We proceed to apply Lemma \ref{preslemma} once again.  To prove that
the displayed relations form a regular sequence in
$A=K[\ta_1,\ldots,\ta_k,\ta_k',\ldots, \ta_{n+k}]$ for any field $K$,
we may assume that $K$ is algebraically closed. We claim that the
affine variety determined by all the relations is a single point (the
origin). This suffices, since Hilbert's Nullstellensatz then asserts
that the ideal $I$ of relations is contained in some power of the
maximal ideal at the origin, and thus $A/I$ is a finite dimensional
$K$-vector space.  To prove the claim, we again separate cases
according the the characteristic of $K$. In addition, one of the
relations is $(\ta_k-\ta_k')\Delta_{n+1-k}=0$, which implies
$\ta_k=\ta_k'$ or $\Delta_{n+1-k}=0$.

If $\text{char}(K)\neq 2$ and $\ta_k=\ta_k'$, then the relations 
(\ref{ogevenR1})--(\ref{ogevenR2'})
and the power
series argument in the proofs of Theorems \ref{presIG} and \ref{presOGodd}
show that
$\Delta_{n-k+2}=\cdots = \Delta_{2n+1}=0$. Lemma \ref{typeApres} then
implies that all the $\ta_i$ must vanish. If 
$\text{char}(K)\neq 2$ and $\Delta_{n+1-k}=0$, then the relations
(\ref{ogevenR1})--(\ref{ogevenR1''}), (\ref{ogevenR2}) and the same 
argument may be used to show that 
\[
\ta_1=\cdots = \ta_{k-1} = \ta_k + \ta_k' = \ta_{k+1} = \cdots = \ta_{n+k}=0.
\]
It follows that $\ta_k=-\ta'_k$, and now (\ref{ogevenR2'}) gives 
$\ta_k=\ta_k'=0$.

If $\text{char}(K)=2$ and $\ta_k=\ta_k'$.,
then the Schur determinants involving
$\ta_1,\ldots,\ta_{k-1}$ in degrees $n-k+2,\ldots,n$ all vanish, so by 
Lemma \ref{typeApres} we get $\ta_1=\cdots= \ta_{k-1}=0$. Now
(\ref{ogevenR1'}) implies that $\ta_{n+1}=0$, and then relations
(\ref{ogevenR1''}) give $\ta_{n+2}=\cdots = \ta_{n+k}=0$. The
remaining relations now show that all the $\ta_i$ vanish. Finally,
if $\text{char}(K)= 2$ and $\Delta_{n+1-k}=0$, then we have 
determinantal relations involving $\ta_1,\ldots,\ta_{k-1},
\ta_k+\ta_k'$ in degrees $n-k+1,\ldots,n$. Therefore all these 
elements are zero, in particular $\ta_k=\ta_k'$, and we are reduced
to the previous case. 
\end{proof}

\subsection{Gromov-Witten invariants}
\label{ogevengwis}

The theory here is quite similar to the case of the odd orthogonal
Grassmannian of Section \ref{ogoddgwis}. However since the Picard
group of the Grassmannian $\OG(n,2n+2)$ has rank two, we will assume
that $m < n$ in this and the following section, and discuss the
quantum cohomology of $\OG(n,2n+2)$ in the Appendix. When $m < n$ (or
$k > 1$), the Gromov-Witten invariants $\langle \ta_\lambda, \ta_\mu,
\ta_\nu \rangle_d$ are defined as in Section \ref{ogoddgwis}, for
elements $\la$, $\mu$, and $\nu$ of $\wt{\cP}(k,n)$ 
such that $|\lambda| + |\mu| +
|\nu| = \dim(\OG')+d(n+k)$.

For any integer $d\lequ m$ and Schubert variety $X_\la$ in $\OG'$, the
varieties $Y_d$, $Y'_d$, $Z_d$, $Z'_d$ and the associated subvarieties
$Y_\la$, $Y'_\la$, $Z_\la$, $Z'_\la$,
classes $\upsilon_\la$, $\upsilon'_\la$,
$\zeta_\la$, $\zeta'_\la$, and quantity $r$
are defined exactly as
in Section \ref{ogoddgwis}, together with the projection maps
$\varphi : Y'_d \to Y_d$ and $\psi: Z'_d \to Z_d$. 
When $d=m+1$, we have $Y_{m+1}=\G(2m+1,V)$
with subvarieties $Y_\la\subset Y_{m+1}$ as before.
For $\la\in \wt{\cP}(k,n)$ we define
$N'_d=N'_d(\la)=\#\{j\le m\,|\,\la_j=d-j<k\}$,
and we define the quantity $M'_d$ as in \eqref{eq:defnMd} but with
$N'_d$ in place of $N_d$.
The subset $S_d(\la,\mu,\nu)$ of the $r=r(d)$ locus of $Y_\la\cap
Y_\mu\cap Y_\nu$ is also defined as before.

\begin{lemma} \label{L:fiberD}
{\rm (a)} The restricted projection
$\pi : T_\la(E_\bull) \to Y_\la(E_\bull)$ is generically
$2^{N'_d}$-to-$1$ when $\rho_{d-1}\subset\la$ and has fibers of positive
dimension when $\rho_{d-1}\not\subset\la$.
\smallskip \\
\noin When $d\le m$, we furthermore have:\\
\noin {\rm (b)} The restriction of $\pi$ over the
$r=2$ locus of $Y_\la(E_\bull)$ is
generically unramified $2^{N'_d}$-to-$1$ when $(\rho_{d-1},1)\subset\la$
and has fibers of positive dimension when
$(\rho_{d-1},1)\not\subset\la$.
\smallskip \\
\noin {\rm (c)} The map $\pi'\colon T'_\la(E_\bull)\to
Y'_\la(E_\bull)$ is generically $1$-to-$1$ when $\rho_{d-1}\subset\la$.
\smallskip \\
\noin {\rm (d)} The restriction of $\pi'$ to the
$r=2$ locus is generically $1$-to-$1$ when
$(\rho_{d-1},1)\subset\la$.
\end{lemma}
\begin{proof}
The proof of Lemma \ref{L:fiberB} can be copied, with the following
modifications.
We take $p_j=p'_j(\la)$,
$p_{m+1}=2n+3$, $\#_j=\{i<j\,|\,p_i+p_j>2n+3\}$, and
$$s_j=\left\{
\begin{array}{ll}
\max(d-j-k,0),&\text{when $\la_j=k<\la_{j-1}$},\\
\max(m+d+p_j-2n-2-j,0),&\text{otherwise}.
\end{array}
\right.
$$
Then $\rho_{d-1}\subset\la$
implies $s_j=\#_j+1$ when $\la_j=d-j<k$ and $s_j\le \#_j$ otherwise.
Throughout, $\la_j=d-j\le k$ should be replaced by
$\la_j=d-j<k$, $N_d$ by $N'_d$, $2n$ by $2n+1$, and
$A_j$, $B_j$, $\Sigma_j$ defined by intersecting with
$\wt{E}_{n+1}$ when $p_j=n+2$.

In the argument with $\rho_{d-1}\not\subset\la$, an additional case
$\la_j=k<\min(\la_{j-1},d-j)$ must be considered.
In this case we have $s_j>0$, so $\dim(B_j)\ge j+1$.
Hence $\dim(B\cap E_n)\ge j$, so there exists a $j$-dimensional
isotropic extension $\Sigma'_j$ of $\Sigma_{j-1}$ contained in
$B\cap E_n$.
\end{proof}

\begin{thm} \label{T:qclasD}
  Let $d\gequ 0$ and choose $\lambda,\mu,\nu \in \wt{\cP}(k,n)$ such that
  $|\lambda|+|\mu|+|\nu| = \dim(\OG') + d(n+k)$.  Let $X_\lambda$,
  $X_\mu$, and $X_\nu$ be Schubert varieties of $\OG(m,2n+2)$ in
  general position, with associated subvarieties and classes in $Y_d$, $Y'_d$,
  $Z_d$, and $Z'_d$.

  \smallskip \noin {\rm (a)} The subvarieties $Y_\la$, $Y_\mu$, and $Y_\nu$
  intersect transversally in $Y_d$, and their intersection is finite.
  For each point in $(A,B) \in Y_\lambda
  \cap Y_\mu \cap Y_\nu$ we have $A=B\cap B^\perp$ or 
  $\dim(B \cap B^\perp)=\dim A+2$.

  \smallskip \noin {\rm (b)} The assignment $f \mapsto
   (\Ker(f),\Span(f))$ gives a $2^{M'_d}$-to-$1$ association between
   rational maps
   $f:\bP^1\to\OG'$ of degree $d$ such that $f(0)\in X_\lambda$,
   $f(1)\in X_\mu$, $f(\infty)\in X_\nu$ and the subset $S_d$ of 
   $Y_\lambda \cap Y_\mu \cap Y_\nu$.

   \smallskip \noin {\rm (c)} When $d\lequ m$ is even, the
   Gromov-Witten invariant $\gw{\ta_\lambda,\ta_\mu,\ta_\nu}{d}$ is
   equal to
\begin{align*}
   & \int_{Y_d}\upsilon_\la\cdot \upsilon_\mu\cdot \upsilon_\nu
    - \frac{1}{2}
   \int_{Y'_d} (\varphi^*\upsilon_\la\cdot \upsilon'_\mu\cdot \upsilon'_\nu +
   \upsilon'_\la\cdot \varphi^*\upsilon_\mu\cdot \upsilon'_\nu +
   \upsilon'_\la\cdot \upsilon'_\mu\cdot \varphi^*\upsilon_\nu) \\
   & {} - \int_{Z_d}\zeta_\la\cdot \zeta_\mu\cdot\zeta_\nu
   + \frac{1}{2} \int_{Z'_d} (\psi^*\zeta_\la\cdot \zeta'_\mu\cdot \zeta'_\nu +
   \zeta'_\la\cdot \psi^*\zeta_\mu\cdot \zeta'_\nu +
   \zeta'_\la\cdot \zeta'_\mu\cdot \psi^*\zeta_\nu).
\end{align*}
When $d\lequ m$ is odd, the
   Gromov-Witten invariant $\gw{\ta_\lambda,\ta_\mu,\ta_\nu}{d}$ is
   equal to
\begin{align*}
   & \int_{Z_d}\zeta_\la\cdot \zeta_\mu\cdot\zeta_\nu
   - \frac{1}{2} \int_{Z'_d} (\psi^*\zeta_\la\cdot \zeta'_\mu\cdot \zeta'_\nu +
   \zeta'_\la\cdot \psi^*\zeta_\mu\cdot \zeta'_\nu +
   \zeta'_\la\cdot \zeta'_\mu\cdot \psi^*\zeta_\nu).
\end{align*}
When $d=m+1$ is even, we have
\[
\gw{\ta_\lambda,\ta_\mu,\ta_\nu}{d} = \int_{\G(2m+1,2n+2)}
  \upsilon_\lambda \cdot \upsilon_\mu \cdot \upsilon_\nu .
\]

  \smallskip \noin {\rm (d)} We have
   $\gw{\ta_\lambda,\ta_\mu,\ta_\nu}{d}=0$ if $\la$ does not contain
   $\rho_{d-1}$ when $d$ is even, or does not contain
   $(\rho_{d-1},1)$ when $d$ is odd.
\end{thm}
\begin{proof} 
The argument is very similar to that in Theorem \ref{T:qclasB}.
Given integers $0\lequ e_1\lequ m$, $0 \lequ e_2 \lequ d \lequ m+1$, 
and $r\gequ 0$ we let $Y^r_{e_1,e_2}$ be the variety (or in some cases,
disjoint union of two varieties) of pairs $(A,B)$ such
that $A \subset B \subset A^\perp \subset V$, $\dim(A) = m-e_1$,
$\dim(B) = m+e_2$, and $\dim(B \cap B^\perp) = m-e_1+r$. These
varieties have a transitive action of $\SO_{2n+2}$. Let
\[ Y^r_\lambda = \{ (A,B) \in Y^r_{e_1,e_2} \mid \exists\, \Sigma \in 
   X_\lambda : A \subset \Sigma \subset B \} \,.
\]
We will show as before that $Y^r_\lambda \cap Y^r_\mu \cap Y^r_\nu$ is
empty unless (i) $e_1=e_2=d$ and $r\in\{0,2\}$, or (ii) $e_1=m$,
$e_2=d=m+1$, and $r=0$.

Following the proof of Theorem \ref{T:qclasC} we compute that the dimension
of (any component of) $Y^r_{e_1,e_2}$ is equal to  
\[ \frac{1}{2}
   (n^2 + 2nk -3k^2 + n + 3k + 2ne_1-2ke_1-e_1^2+4e_2k-2e_2^2+ 3e_1-r-r^2) \,.
\]
The varieties $Q_s$ defined by equation (\ref{Qs}) have the same dimension
as in the proof of Theorem \ref{T:qclasB}. We deduce that 
$\codim(Y^r_\lambda)+\codim(Y^r_\mu)+\codim(Y^r_\nu) -
\dim(Y^r_{e_1,e_2})$ is greater than or equal to the number
$\Delta''(e_1,e_2,r,s) := \dim \OG' + d(n+k) - \dim(Y^r_{e_1,e_2}) -
3\dim Q_s$, which is equal to
\[ (r-3s)(r-3s+1)/2 + (n+k)(d-e_2) + 
   (n-k+e_2-2e_1+3s)(e_2-e_1) \,. 
\]
The remainder of the argument is exactly the same as in the case of $\OG$.
The analysis of $Y_{\lambda^0}$ can be carried out with
$\lambda^0=(\rho_{d-1},1)$ of arbitrary type (the analysis yields
$Y_{\lambda^0}=Y_{(1^d)}$, so the type is irrelevant).
\end{proof}

For $\lambda\in\wt{\cP}(k,n)$, let $\ov{\la}$ denote the diagram
obtained by deleting the leftmost column of $\lambda$, with the
convention that $\type(\ov{\la})=\type(\la)$. 

\begin{prop}
\label{dprop0}
For  $p\in [1,n+k]$ and $\lambda$, $\mu\in \wt{\cP}(k,n)$
with $|\lambda|+|\mu|+p=\dim \OG' + n + k$, we have
\begin{equation}
\label{deq1}
\langle \ta_\lambda, \ta_{\mu}, \ta_p\rangle_1 =
\int_{\OG(n+2-k,2n+2)} \ta_{\ov{\lambda}} \cdot \ta_{\ov{\mu}} \cdot
\ta_{p-1}.
\end{equation}
Moreover, when $p=k$, we have
\begin{equation}
\label{deq1'}
\langle \ta_\lambda, \ta_{\mu}, \ta'_k\rangle_1 =
\int_{\OG(n+2-k,2n+2)} \ta_{\ov{\lambda}} \cdot \ta_{\ov{\mu}} \cdot
\ta'_{k-1}.
\end{equation}
\end{prop}
\begin{proof}
The argument is the same as the one in the proof of Proposition 
\ref{bprop0}.
\end{proof}

\subsection{Quantum cohomology}
\label{qcOGeven}
As in Section \ref{qcOGodd}, the degree of the variable $q$ in 
$\QH^*(\OG')$ is $n+k$, and the quantum Pieri rule will involve
both $q$ and $q^2$ terms. 
Let $\wt{\cP}'(k,n+1)$ be the set of $\nu\in\wt{\cP}(k,n+1)$ such that
$\ell(\nu)=n+2-k$, $2k-1\lequ \nu_1\lequ n+k$, and the number of boxes
in the second column of $\nu$ is at most $\nu_1-2k+2$.  For any
$\nu\in\wt{\cP}'(k,n+1)$, we let $\wt{\nu}\in \wt{\cP}(k,n)$ be the
element of $\wt{\cP}'(k,n)$ obtained by removing the first row of
$\nu$ as well as $n+k+1-\nu_1$ boxes from the first column.  That is,
\[
\wt{\nu}=(\nu_2,\nu_3,\ldots,\nu_r), \
\text{where $r=\nu_1-2k+1$.}
\]
Moreover, we have $\type(\wt{\nu})=\type(\nu)$, if $\type(\nu)=0$, 
and otherwise $\type(\wt{\nu})= 3 - \type(\nu)$. Finally, for any
$\la\in \wt{\cP}(k,n)$, we define $\la^*$ with $\type(\la^*)=\type(\la)$
by $\la^*=(\la_2,\la_3,\ldots)$.

\begin{thm}[Quantum Pieri rule for $\OG'$] 
\label{ogevenqupieri}
For any $k$-strict partition $\la\in\wt{\cP}(k,n)$ and integer 
$p\in [1,n+k]$, we have 
\[
\ta_p \cdot \ta_\la =
\sum_{\la\to\mu} \delta_{\la\mu} \, 2^{N'(\la,\mu)}\,\tau_\mu +
\sum_{\la\to\nu} 
\delta_{\la\nu} \, 2^{N'(\la,\nu)}\,\tau_{\wt{\nu}}\, q \, +\, 
\sum_{\la^*\to\rho} \delta_{\la^*\rho} \,
2^{N'(\la^*,\rho)} \,\ta_{\rho^*}\, q^2
\]
in the quantum cohomology ring $\QH^*(\OG(n+1-k,2n+2))$. Here (i) the
first sum is classical, as in (\ref{ogevenclass}), (ii) the second sum
is over $\nu\in \wt{\cP}'(k,n+1)$ with $\la\to\nu$ and
$|\nu|=|\la|+p$, and (iii) the third sum is empty unless $\la_1=n+k$,
and over $\rho\in\wt{\cP}(k,n)$ such that $\rho_1=n+k$,
$\la^*\to\rho$, and $|\rho|=|\la|-n-k+p$.  Furthermore, the product
$\tau'_k\cdot\ta_{\la}$ is obtained by replacing 
$\delta$ with $\delta'$ throughout.
\end{thm}
\begin{proof} 
The argument is similar to the proof of Theorem
\ref{ogoddqupieri}.
Theorem \ref{T:qclasD}(d) implies that the quantum Pieri products
$\ta_p\,\ta_{\la}$ and $\ta'_k\,\ta_{\la}$ contain at most quadratic
$q$ terms. A dimension count shows that the right hand side of
equations \eqref{deq1} and \eqref{deq1'} vanishes when either
$\lambda$ or $\mu$ has length less than $n+1-k$. 

If $\mu^{\vee}$ is the dual partition to $\mu$ in $\wt{\cP}(k,n)$
(see Sections \ref{svD} and \ref{dualpar}), then $\ell(\mu^{\vee})=n+1-k$
if and only if $\mu_1 < \ell(\mu)+2k -1$. Furthermore, if $k > 1$ and
$\mu\in\wt{\cP}(k,n)$ satisfies $\mu_1< \ell(\mu)+2k -1$, then
$(\ov{\mu^{\vee}})^{\vee} = (\ell(\mu)+2k-2, \ov{\mu})^{\dagger}$,
where the dagger $\dagger$ means that $\type((\ell(\mu)+2k-2,
\ov{\mu})^{\dagger})= \type(\mu)$, if $\type(\mu)=0$, and otherwise
$\type((\ell(\mu)+2k-2, \ov{\mu})^{\dagger})= 3 - \type(\mu)$.  We
deduce the following result.

\begin{prop}
\label{dcor1}
Consider $\lambda$, $\mu\in \wt{\cP}(k,n)$ with 
$|\mu|+n+k=|\lambda|+p$, for $1\lequ p\lequ n+k$.
If $\ell(\la)=n+1-k$ and $\mu_1<\ell(\mu)+2k-1$ then, in $\QH^*(\OG')$,
the coefficient of $\ta_\mu\, q$ in $\ta_p\,\ta_\lambda$ is
equal to the coefficient of
$\ta_{(\ell(\mu)+2k-2,\, \ov{\mu})^{\dagger}}$
in $\ta_{p-1}\,\ta_{\ov{\lambda}}\in \HH^*(\OG(n+2-k,2n+2))$. 
Otherwise, the coefficient vanishes.
\end{prop}

Theorem \ref{thm.ogevenpieri} implies that for $\lambda$, $\mu\in
\wt{\cP}(k,n)$ with $\ell(\la)=n+1-k$ and $\mu_1<\ell(\mu)+2k-1$, we
have $\ov{\la} \to (\ell(\mu)+2k-2,\, \ov{\mu})^{\dagger}$ in
$\wt{\cP}(k-1,n)$ if and only if $\la \to (\ell(\mu)+2k-1,\mu,
1^{n+1-k-\ell(\mu)})^{\dagger}$ in $\wt{\cP}(k,n+1)$, with the same
coefficients ($\delta$ and $N'$ numbers). It follows that for $\la$,
$\mu\in\wt{\cP}(k,n)$, the coefficient of $\ta_\mu\, q$ in the quantum
product $\ta_p\,\ta_\lambda$ in $\QH^*(\OG')$ is equal to the
coefficient of $\ta_{(\ell(\mu)+2k-1,\mu,
1^{n+1-k-\ell(\mu)})^{\dagger}}$ in the cup product
$\ta_p\,\ta_\lambda$ in $\HH^*(\OG(n+2-k,2n+4))$ when $\mu_1 <
\ell(\mu)+2k-1$ (and $\ell(\la)=n+1-k$), and is 0 otherwise. In
addition, $\nu \mapsto \wt{\nu}$ induces a 1-1 map
$\wt{\cP}'(k,n+1)\to\wt{\cP}(k,n)$ with image $\{\mu\in \wt{\cP}(k,n)\
:\ \mu_1<\ell(\mu)+2k-1\}$, and the inverse of this map is given by
$\mu\mapsto (\ell(\mu)+2k-1,\mu,1^{n+1-k-\ell(\mu)})^{\dagger}$.  We
deduce that the coefficient of $\ta_{\wt{\nu}}\, q$ in Theorem
\ref{ogevenqupieri} is equal to the coefficient of $\ta_{\nu}$ in the
product $\ta_p\,\ta_{\la}$ in $\HH^*(\OG(n+2-k,2n+4))$, for $\nu\in
\wt{\cP}'(k,n+1)$, and these are all the linear $q$ terms. The linear 
$q$ terms in $\ta'_k\,\ta_{\la}$ are handled similarly.

The rest of the argument is the same as in the proof of 
Theorem \ref{ogoddqupieri}, using the basic relation 
$\ta_{n+k}^2=q^2$ in $\QH^*(\OG')$.
\end{proof}

\begin{example}
  In the quantum cohomology ring of $\OG(5,14)$ we have $\tau_2 \cdot
  \tau_{(8,7,2,1,1)} = \tau_{(8,7,4,1,1)} + \tau_{(8,7,3,2,1)} +
  \tau_{(8,7,6)} + \tau'_{(7,2,2,1,1)}\, q +  \tau_{(2,1,1,1)}\, q^2 + 
  \tau_{(2,2,1)}\, q^2 +  \tau_{(3,1,1)}\, q^2$ and $\tau'_2 \cdot
  \tau_{(8,7,2,1,1)} = \tau_{(8,7,4,1,1)} + \tau_{(8,7,3,2,1)} + 
  \tau_{(7,3,1,1,1)}\, q + \tau_{(2,1,1,1)}\, q^2 + \tau_{(2,2,1)}\, q^2 + 
  \tau_{(3,1,1)}\, q^2$.  The Schubert class $\tau'_{(7,2,2,1,1)}$ has type
  2 while all remaining classes in these expansions have types 0 or 1.
\end{example}

\begin{thm}[Ring presentation] 
\label{ogevenqpres}
The quantum cohomology ring
$\QH^*(\OG')$ is presented as a quotient of the polynomial
ring $\Z[\ta_1,\ldots,\ta_k,\ta_k',\ta_{k+1},\ldots,\ta_{n+k},
q]$ modulo the relations 
\begin{gather}
\label{ogevenQR1}
\Delta_r = 0, \ \ \ \ 
n-k+1 < r \lequ n, \\
\label{ogevenQR1'}
\ta_k\Delta_{n+1-k}=
\ta'_k \Delta_{n+1-k}=
\sum_{p=k+1}^{n+1}(-1)^{p+k+1}\tau_p
\Delta_{n+1-p}, \\
\label{ogevenQR1''}
\sum_{p=k+1}^r(-1)^p\tau_p
\Delta_{r-p}=0, \ \ \ \
n+1 < r < n+k, \\
\label{ogevenQR1q}
\sum_{p=k+1}^{n+k}(-1)^p\tau_p
\Delta_{n+k-p}=-q,
\end{gather}
and
\begin{gather}
\label{ogevenQR2}
\ta_r^2 + \sum_{i=1}^r(-1)^i \ta_{r+i}c_{r-i}= 0,
 \ \  \ \ k+1 \lequ r \lequ n, \\
\label{ogevenQR2'}
\ta_k\ta'_k+\sum_{i=1}^k(-1)^i \ta_{k+i}\ta_{k-i}=0,
\end{gather}
where the variables $c_p$ are defined by (\ref{ctotau}).
\end{thm}
\begin{proof}
We proceed as in the proof of Theorem \ref{ogoddqpres}.
The relations (\ref{ogevenQR1})--(\ref{ogevenQR1''}) are true
because they hold classically and the degree of $q$ is $n+k$. 
The quantum Pieri rule implies that the only the $p=2k-1$
summand in (\ref{ogevenQR1q}) gives a non-zero quantum correction,
equal to $-q$. The remaining relations are easily checked to 
hold in $\QH^*(\OG')$ when $k<n$, but the case $k=n$ yields the
quadric $\OG(1,2n+2)$, which must be checked separately. 
In this case the degree of $q$ is $2n$,
and we must verify (\ref{ogevenQR2'}) using the quantum Pieri rule.
The coefficient of $q$ in $\tau_n\tau'_n$ is $1$ when $n$ is even,
and $0$ when $n$ is odd, while
the $i$th term in the sum, for $1\le i\le n-1$, contributes $(-1)^iq$.
So, (\ref{ogevenQR2'}) holds in $\QH^*(\OG(1,2n+2))$.
\end{proof}

\begin{remark}
  The method of computing Gromov-Witten invariants
  explained in Section \ref{S:computing} carries over to types B and D,
  using polynomial expressions in the special Schubert classes of the
  orthogonal Grassmannians.
\end{remark}

\section{Schubert varieties in isotropic Grassmannians}
\label{schvars}

We begin this section by giving a uniform description of the 
Schubert varieties in isotropic Grassmannians $X$, parametrizing 
them using {\em index sets}. These sets record the attitude of 
a subspace in $X$ with respect to a fixed isotropic flag, and 
are important ingredients in our proof of the classical 
Pieri formula for $X$.

Let $V \cong \C^N$ be a complex vector space equipped with a
non-degenerate skew-symmetric or symmetric bilinear form $(\ ,\,)$.
Given a non-negative integer $m \lequ N/2$, we let $X$ denote the 
Grassmannian of $m$-dimensional isotropic subspaces of $V$,
\[ X = \{ \Sigma \subset V \mid \dim(\Sigma)=m \text{ and }
   (\ ,\, )\vert_{\Sigma}\equiv 0 \} \,. 
\]
This variety has a transitive action of the group $G =\Sp(V)$
or $G=\SO(V)$ of linear automorphisms preserving the form on $V$.
There is a single exception: when $m=N/2$ and the form is symmetric,
then the space of isotropic subspaces has two isomorphic connected
components, each a single $\SO(V)$ orbit.

Let $F_\bull$ be an isotropic flag of $V$. The Schubert varieties in
$X$ relative to the flag $F_\bull$ are the orbit closures for the
action of the stabilizer $B \subset G$ of $F_\bull$. We proceed to 
give an elementary description of these varieties.

For any subset $H \subset V$ we let $\langle H \rangle \subset V$
denote the linear span of $H$.  We will say that a basis
$\{e_1,\dots,e_N\}$ of $V$ is a {\em standard basis\/} with respect to
the isotropic flag $F_\bull$, if $(e_i,e_j) = 0$ for $i+j \neq N+1$,
$(e_i,e_{N+1-i}) = 1$, for $1\lequ i \lequ \lfloor(N+1)/2\rfloor$, 
and $F_p = \langle e_1,\dots,e_p \rangle$ for each $p$.

\begin{lemma} \label{L:chbas}
  Given any isotropic subspace $\Sigma \subset V$ and an isotropic
  flag $F_\bull \subset V$, there exists a standard basis
  $\{e_1,\dots,e_N\}$ of $V$ with respect to $F_\bull$ such that
  $\Sigma = \langle \{e_1,\dots,e_N \} \cap \Sigma \rangle$.
\end{lemma}
\begin{proof}
  Choose a vector $0 \neq e_1 \in F_1$.  We choose a second vector
  $e_N \in V \smallsetminus F_{N-1}$ such that
  $(e_1,e_N)=1$ as follows.  If $e_1 \in
  \Sigma$ then choose $e_N \in V \smallsetminus F_{N-1}$.  If $e_1
  \not \in \Sigma$ and $\Sigma \subset F_{N-1}$ then choose $e_N
  \in \Sigma^\perp \smallsetminus F_{N-1}$.  Finally, if $e_1 \not \in
  \Sigma$ and $\Sigma \not \subset F_{N-1}$ then choose $e_N \in
  \Sigma \smallsetminus F_{N-1}$.
  
  Set $V' = \langle e_1,e_N \rangle^\perp$, $\Sigma' = \Sigma \cap
  V'$, and $F'_i = F_{i+1} \cap V'$.  By induction we can find a basis
  $\{e'_1,\dots,e'_{N-2}\}$ for $V'$ satisfying the requirement of
  the lemma with respect to $\Sigma'$ and $F'_\bull$.  We obtain the
  required basis for $V$ by setting $e_i = e'_{i-1}$ for $1< i < N$.
\end{proof}

For any point $\Sigma \in X$ we define a subset $P(\Sigma) \subset
[1,N]$ of cardinality $m$ by
\[ P(\Sigma) = \{ p \in [1,N] \mid \Sigma \cap F_p \supsetneq \Sigma
   \cap F_{p-1} \} \,.
\]
Notice that $P(\Sigma') = P(\Sigma)$ for any point $\Sigma'$ in the
orbit $B.\Sigma \subset X$.  On the other hand, it follows from
Lemma~\ref{L:chbas} that any point $\Sigma' \in X$ such that
$P(\Sigma') = P(\Sigma)$ must be in this orbit.  We call a subset $P
\subset [1,N]$ of cardinality $m$ an {\em index set\/} if for all $i,j
\in P$ we have $i+j \neq N+1$.  Any set $P(\Sigma)$ is an index set,
since no vector in $F_i \smallsetminus F_{i-1}$ is orthogonal to a
vector in $F_{N+1-i} \smallsetminus F_{N-i}$.  On the other hand,
given any index set $P$ we can construct a point $\Sigma \in X$ with
$P(\Sigma) = P$.  In fact, if $\{e_1,\dots,e_N\}$ is a standard basis
for $V$ with respect to \ $F_\bull$, then $\Sigma = \langle e_i : i\in
P \rangle$ has this property.  In other words, the $B$-orbits (or
Schubert cells) in $X$ correspond 1-1 to the index sets $P$. We let
$X^\circ_P(F_\bull)$ denote the Schubert cell given by $P$, that is
\[ X^\circ_P(F_\bull) = \{ \Sigma \in X \mid P(\Sigma) = P \} \,. \]
The Schubert variety $X_P(F_\bull)$ is defined as the closure of the
Schubert cell $X^\circ_P(F_\bull)$; we let $|P|$ be its codimension in
$X$. For each index set $P=\{p_j\}$, let $[X_P]\in H^{2|P|}(X,\Z)$
denote the cohomology class Poincar\'e dual to the cycle defined by
$X_P(F_\bull)$. 

\subsection{Type C}
\label{svC}

Let $V\cong \C^{2n}$ be a symplectic vector space and $X=\IG(m,2n)$.
Given another index set $Q = \{q_1 < \dots < q_m\}$ we write $Q \lequ
P$ if $q_j \lequ p_j$ for each $j$.

\begin{prop}
\label{closureC}
For any index set $P = \{ p_1 < p_2 < \dots < p_m \} \subset [1,2n]$ we
have
\begin{equation}
\label{schvarineq}
   X_P(F_\bull)=\{ \Sigma \in \IG \mid \dim(\Sigma \cap
   F_{p_j}) \gequ j,  \ \ \forall \,  1 \lequ j \lequ m \} \,. 
\end{equation}
\end{prop}
\begin{proof}
  The set on the right hand side of (\ref{schvarineq}) is closed and
  equals the union of the orbits $X^\circ_Q$ for all index sets $Q$
  such that $Q \lequ P$.  We must show that each of these orbits is
  contained in the closure of $X^\circ_P$.  Assuming $Q < P$, it is
  enough to construct an index set $P'$ such that $Q \lequ P' < P$ and
  $X^\circ_{P'} \subset \overline{X^\circ_P}$.
  
  Choose $j$ minimal such that $q_j < p_j$ and fix a standard basis
  $\{e_1,\dots,e_{2n}\}$ of $V$ with respect to $F_\bull$.  If
  $2n+1-q_j \not \in P$ or $q_j+p_j=2n+1$, then set $P' =
  \{q_1,\dots,q_j,p_{j+1},\dots,p_m\}$, and define a morphism $\bP^1
  \to \IG$ by
\[
[s:t] \mapsto \Sigma_{[s:t]} =
  \langle e_{p_1},\dots,e_{p_{j-1}}, s e_{q_j} + t e_{p_j},
  e_{p_{j+1}}, \dots, e_{p_m} \rangle.
\]
Since $\Sigma_{[1:0]} \in
  X^\circ_{P'}$ and $\Sigma_{[s:t]} \in X^\circ_{P}$ for $t \neq 0$,
  it follows that $X^\circ_{P'} \subset \overline{X^\circ_P}$.
  
  If $2n+1-q_j \in P$ and $q_j+p_j \neq 2n+1$,  
  set $P' = (P \smallsetminus \{p_j,2n+1-q_j\}) \cup
  \{q_j,2n+1-p_j\}$.  We then use the morphism $\bP^1 \to \IG$ given by
\[
[s:t] \mapsto \Sigma_{[s:t]} = \langle e_p : p \in P\cap P' \rangle
  \oplus \langle s e_{q_j} + t e_{p_j}, s e_{2n+1-p_j} \pm t
  e_{2n+1-q_j} \rangle,
\]
where the sign is chosen so that $\Sigma_{[s:t]}$ is isotropic, to
show that $X^\circ_{P'} \subset \overline{X^\circ_P}$, as required.
\end{proof}

Define the dual index set $P^\vee = \{p^\vee_j\}$
by setting $p^\vee_j = 2n+1-p_{m+1-j}$.

\begin{prop}
\label{BCduals} 
For any index sets $P$ and $Q$, we have 
\[ 
\int_{\IG} [X_P] \cdot [X_Q] = \delta_{Q,P^\vee} \,.
\]
\end{prop}
\begin{proof}
 Let $F_\bull$ and $G_\bull$ be general isotropic flags in $V$, and
  assume $P=\{p_1 < \dots < p_m\}$ and $Q = \{q_1 < \dots < q_m\}$ are
  index sets such that $X^\circ_P(F_\bull) \cap X^\circ_Q(G_\bull)
  \neq \emptyset$.  For any point $\Sigma$ in this intersection we
  have $\dim(\Sigma\cap F_{p_{m+1-j}}) \gequ m+1-j$ and
  $\dim(\Sigma\cap G_{q_j}) \gequ j$, which implies that
  $F_{p_{m+1-j}} \cap G_{q_j} \neq 0$.  It follows from this that $q_j
  \gequ 2n+1-p_{m+1-j}$ for each $j$.  Notice also that
  $X^\circ_P(F_\bull) \cap X^\circ_{P^\vee}(G_\bull) = \{ \Sigma_0
  \}$, where $\Sigma_0 = \bigoplus_j (F_{p_j} \cap G_{2n+1-p_{m+1-j}})
  \subset V$.  It follows that $\int [X_P] \cdot [X_{P^\vee}] = 1$,
  and that $X^\circ_P(F_\bull)$ and $X^\circ_{P^\vee}(G_\bull)$ have
  complementary dimensions in $\IG$.  If $Q \neq P^\vee$ then $Q >
  P^\vee$, so $X^\circ_{P^\vee}(G_\bull)$ must be a proper closed
  subset of $X^\circ_Q(G_\bull)$ and the intersection
  $X^\circ_P(F_\bull) \cap X^\circ_Q(G_\bull)$ has positive dimension.
\end{proof}

Set $k=n-m$. We establish a 
bijection between the index sets $P$ and the set $\cP(k,n)$ of $k$-strict
partitions $\lambda$ contained in an $m \times (n+k)$ rectangle, i.e.\ 
$\ell(\lambda) \lequ m$ and $\lambda_1 \lequ n+k$. 
Let $P = \{ p_1 < p_2 < \dots < p_m \}$ be any index set.  Choose $0
\lequ s \lequ m$ such that $p_s \lequ n < p_{s+1}$, where $p_0=0$ and
$p_{m+1} = 2n+1$.  Write 
\[
[n+1,2n] \smallsetminus \{ 2n+1-p_1, \dots,
2n+1-p_s \} = \{ r_1 < r_2 < \dots < r_{n-s} \}
\] 
and choose indices $1\lequ t_{s+1} < \dots < t_m \lequ n-s$ 
so that $p_j = r_{t_j}$ for $s+1
\lequ j \lequ m$.  The bijection maps $P$ to the $k$-strict partition
$\lambda = (\lambda_1,\dots,\lambda_m)$ given by
\[
\la_j = \begin{cases}
           n+k+1-p_j & \text{if $1 \lequ j \lequ s$}, \\
           k+j-s-t_j & \text{if $s+1 \lequ j \lequ m$}.
\end{cases}
\]
Since the subsets of $\{r_1,\dots,r_{n-s}\}$ of
cardinality $m-s$ correspond 1-1 to the partitions
$(\lambda_{s+1},\dots,\lambda_m)$ contained in an $(m-s) \times k$
rectangle, it follows that the assignment $P \mapsto \lambda$ gives a
bijection between the index sets $P$ and the $k$-strict partitions
$\lambda$ in $\cP(k,n)$.  We use the
notation $X_\lambda(F_\bull) \equiv X_P(F_\bull)$ for the Schubert
varieties in $\IG$ relative to the isotropic flag $F_\bull$.

\begin{prop} \label{P:trnslC}
Let $P = \{p_1 < \dots < p_m\} \subset [1,2n]$ be an index set and
$\lambda = (\lambda_1, \dots, \lambda_m)$ the corresponding $k$-strict
partition.  For $1 \lequ i \lequ j \lequ m$ we have

\emph{(i)} $\lambda_j \lequ k$ if and only if $p_j > n$;

\emph{(ii)} $\lambda_i+\lambda_j \lequ 2k+j-i$ if and only if $p_i+p_j >
    2n+1$;

\emph{(iii)} $\lambda_j = n+k+1-p_j + \#\{i<j : p_i+p_j > 2n+1\}$; and

\emph{(iv)} $p_j = n+k+1-\lambda_j + \#\{i<j : \lambda_i+\lambda_j \lequ
    2k+j-i\}$.
\end{prop}
\begin{proof}
  The first point (i) is clear from the definitions.  Let $s$, $\{r_1
  < \dots < r_{n-s}\}$, and $\{t_{s+1}<\dots<t_m\}$ be chosen as
  above.  For $j > s$ we have $t_j = \# \{ i\lequ n-s : r_i \lequ p_j \} =
  p_j-n - \#\{i \lequ s : 2n+1-p_i < p_j \}$, so we obtain $\lambda_j =
  n+k-p_j+j-s+ \#\{i \lequ s : p_i+p_j > 2n+1 \}$, which implies (iii).
  Point (ii) is clearly true if $p_j \lequ
n$ or $p_i > n$, so assume that $p_i \lequ n < p_j$.  Then $\lambda_i +
\lambda_j = 2n+2k+2-(p_i+p_j) + \#\{l < j : p_l+p_j > 2n+1\}$, and
  (ii) follows because $\#\{l < j : p_l+p_j > 2n+1\}$ is smaller than
or equal to $j-i+(p_i+p_j-2n-2)$ when $p_i+p_j > 2n+1$, while it is
greater than or equal to $j-i-1-(2n-p_i-p_j)$ when $p_i+p_j < 2n+1$.
Finally (iv) follows from (ii) and (iii).
\end{proof}

We note that if $\lambda_j = 0$ then the Schubert condition
$\dim(\Sigma \cap F_{p_j}) \gequ j$ is redundant.  In fact, we have
$p_j = n+k+1 + \#\{i<j : p_i+p_j > 2n+1\}$.  If $p_1+p_j > 2n+1$ then
$p_j = n+k+j$, so the Schubert condition holds automatically.
Otherwise choose $i \lequ j$ maximal so that $p_i+p_j < 2n+1$ and set
$C = \Sigma \cap F_{p_i}$.  Then $p_j = n+k+j-i$.  We claim that the
above Schubert condition is a consequence of the condition $\dim(C)
\gequ i$.  In fact, since $\Sigma \subset C^\perp$ and $F_{p_j} \subset
F_{2n-p_i} = F_{p_i}^\perp \subset C^\perp$, we obtain $\dim(\Sigma
\cap F_{p_j}) \gequ m + p_j - (2n-i) = j$ as required.

\begin{cor}
For each partition $\la\in\cP(k,n)$, the Schubert variety indexed 
by $\lambda$ is given by
\[ X_\lambda(F_\bull) = \{ \Sigma \in X \mid
   \dim(\Sigma \cap F_{p_j(\lambda)}) \gequ j, \ \  \forall \, 1 \lequ j \lequ
   \ell(\lambda) \}
\]
where $p_j(\lambda) := n+k+1-\lambda_j + \#\{i<j : \lambda_i+\lambda_j
\lequ 2k+j-i\}$.
\end{cor}

One can also check that the Schubert condition $\dim(\Sigma \cap
F_{p_j}) \gequ j$ is redundant if $\lambda_j = \lambda_{j+1}+1 \gequ 
k+2$, or if $\lambda_j = \lambda_{j+1} < 2k+j-\lambda_1$.  However,
the following example shows that a Schubert condition
$\dim(\Sigma \cap F_{p_j}) \gequ j$ may be necessary, even if
$\lambda_j = \lambda_{j+1}$.

\begin{example} \label{E:schubcond}
  Let $X = \IG(3,8)$ and $\lambda = (3,1,1)$.  Then
  $X_\lambda(F_\bull)$ is the variety of points $\Sigma \in X$ such
  that $\dim(\Sigma \cap F_3) \gequ 1$, $\dim(\Sigma \cap F_5) \gequ 2$,
  and $\dim(\Sigma \cap F_7) \gequ 3$.  One may check that $\{ \Sigma
  \in X \mid \Sigma\cap F_3\ne 0\text{ and } \Sigma \subset F_7 \} =
  X_\lambda(F_\bull) \cup X_5(F_\bull)$.
\end{example}

Choose a standard basis $\{e_1,\dots,e_{2n}\}$ for $V$ with respect to
$F_\bull$.  The $B$-orbit given by $P$ is an affine space with
coordinates given using $m \times 2n$ matrices $A = \{a_{jr}\}$ in
which some entries $a_{jr}$ are free parameters and others are
determined by the free entries.  The matrix corresponds to the
subspace $\Sigma \subset V$ spanned by its rows.  Each entry
$a_{j,p_j}$ of $A$ is equal to 1, while $a_{jr} = 0$ for $r > p_j$.
We also set $a_{j,p_i} = 0$ for $i<j$.  If $r < p_j$ and $r \not \in
\{p_i, 2n+1-p_i\}$ for all $i<j$ then $a_{jr}$ is a free variable.
Finally, for each $i<j$ such that $p_i+p_j > 2n+1$, the entry
$a_{j,2n+1-p_i}$ is uniquely determined from the free variables by the
requirement that rows $i$ and $j$ in $A$ are orthogonal.  For
example, if $m=3$, $n=5$, and $P = \{3,5,9\}$ then the matrix $A$ has
the shape
\[ A = \left[ \begin{array}{cccccccccc}
   *&*&1&0&0&0&0&0&0&0 \\
   *&*&0&*&1&0&0&0&0&0 \\
   *&*&0&*&0&=&*&=&1&0
\end{array}\right] \,. \]
Since the number of free variables $*$ in row $j$ of $A$ is equal to
$p_j-j-\#\{i<j : p_i+p_j > 2n+1\} = n+k+1-j-\lambda_j$, it follows
that the dimension of $X_\lambda(F_\bull)$ is $m(n+k+1) - m(m+1)/2 -
|\lambda|$.

\begin{prop}
  The dimension of $\IG(m,2n)$ is $2m(n-m)+m(m+1)/2$, and the
  codimension of $X_\lambda(F_\bull)$ is $|\lambda|$.
\end{prop}

For a $k$-strict partition $\lambda$ we define the dual partition
$\lambda^\vee$ to be the unique $k$-strict partition such that
$p_j(\lambda^\vee) = 2n+1-p_{m+1-j}(\lambda)$.
Proposition~\ref{P:trnslC} implies that
\[ \begin{split} \lambda^\vee_{m+1-j} = 
  2k + 1 - \lambda_{j} 
  + \#\{ i<j : \lambda_{i}+\lambda_{j} \lequ 2k+j-i \} \\
  +\, \#\{ i>j : \lambda_{i}+\lambda_{j} > 2k+i-j \}
  \,.
\end{split} \]
The relationship between the diagrams of $\la$ and $\la^\vee$ is
described in Section \ref{dualpar}.
If $[X_\lambda] \in H^{2|\lambda|}(\IG)$ denotes the cohomology
class which corresponds to $X_\lambda(F_\bull)$, 
then Proposition \ref{BCduals} gives
\[
\int_{\IG} [X_\lambda] \cdot [X_\mu] = \delta_{\mu,\lambda^\vee} \,.
\]

\subsection{Type B}
\label{svB}

The situation here is very similar to that in the previous type C
section, so we will point out the main differences and leave the
details to the reader. We equip $V \cong \C^{2n+1}$ with a
non-degenerate symmetric bilinear form $(\ ,\,)$, and let
$X=\OG(m,2n+1)$. This algebraic variety has the same
dimension as $\IG(m,2n)$.

For every $\Sigma\in \OG$, the index set 
\[
P = P(\Sigma) = \{ p_1 < p_2 < \dots < p_m \} \subset [1,2n+1] 
\]
satisfies $n+1 \notin P$, and the closure of the
Schubert cell $X^{\circ}_P(F_\bull)$ in $\OG$ is the Schubert variety
\[
X_P(F_\bull) = \{ \Sigma \in \OG \mid \dim(\Sigma \cap
   F_{p_j}) \gequ j,  \ \ \forall \, 1 \lequ j \lequ m \} \,. 
\]
Note that Proposition \ref{BCduals} holds for $\OG$ as well as $\IG$.

The index sets $\ov{P}\subset [1,2n+1]$ for $\OG$ are in bijection
with the index sets $P \subset [1,2n]$ 
for $\IG$ and the $k$-strict partitions in
$\cP(k,n)$.  For any $\la\in\cP(k,n)$, the corresponding index set
$\ov{P}=\{ \ov{p}_1 < \cdots < \ov{p}_m\}$ satisfies
\[
\ov{p}_j(\la)= \begin{cases}
           p_j(\la)+1 & \text{if $\la_j \lequ k$}, \\
           p_j(\la) & \text{if $\la_j > k$},
\end{cases}
\]
where $p_j(\lambda) = n+k+1-\lambda_j + \#\{i<j : \lambda_i+\lambda_j
\lequ 2k+j-i\}$ are the indices used in the previous subsection. 
The dual index set $\ov{P}^{\vee}$ in type B satisfies
$\ov{p}^\vee_j = 2n+2-p_{m+1-j}$.

Let $\{e_1,\dots,e_{2n+1}\}$ be a standard basis of $V$ with respect
to the isotropic flag $F_\bull$. We may then represent the 
 Schubert cell indexed by $\ov{P}$ by an $m \times
 (2n+1)$ matrix $A=\{a_{jr}\}$ whose rows span the subspaces in the
 cell, as in type C.  Each entry $a_{j,\ov{p}_j}$ of $A$ is equal to
 1, $a_{jr} = 0$ for $r > \ov{p}_j$, while $a_{j,\ov{p}_i} = 0$ for
 $i<j$.  If $r < \ov{p}_j$ and $r \not \in \{\ov{p}_i,
 2n+2-\ov{p}_i\}$ for all $i\lequ j$ then $a_{jr}$ is a free variable.
 Finally, for each $i \lequ j$ such that $\ov{p}_i+\ov{p}_j > 2n+2$,
 the entry $a_{j,2n+2-\ov{p}_i}$ is determined from the free variables
 and the isotropicity condition on the rows of $A$.  For example, if
 $m=3$, $n=5$, and $\ov{P} = \{3,5,10\}$ then the matrix $A$ has the
 shape
\[ A = \left[
   \begin{array}{ccccccccccc}
   *&*&1&0&0&0&0&0&0&0&0 \\
   *&*&0&*&1&0&0&0&0&0&0 \\
   *&=&0&*&0&*&=&*&=&1&0
\end{array}\right] \,. \]
By counting the number of free entries $*$ in $A$, we see that the codimension
of $X_{\la}(F_\bull)$ in $\OG$ is equal to $|\la|$.

\subsection{Type D}
\label{svD}

Let $V \cong \C^{2n+2}$ be a complex vector space equipped with a
non-degenerate symmetric bilinear form $(\ ,\, )$. 
For any nonnegative integer $m < n+1$, let $\OG' = \OG(m,2n+2)$. The
index sets $P=\{p_j\}$ parametrize the $B$-orbits
$X^{\circ}_P(F_\bull)$ in $\OG'$, as before. However the closures of
these orbits behave differently. Given another index set $Q=\{q_j\}$,
we will write $Q \preceq P$ when $q_j\lequ p_j$ for each $j$ and if $p_i
= n+2$ for some $i$ then $q_i \neq n+1$. The next result shows that
the Schubert variety $X_P(F_\bull)$ is equal to the union of the
orbits $X^\circ_Q(F_\bull)$ for all index sets $Q$ such that $Q \preceq
P$.

\begin{prop}
\label{closureD}
Let $P = \{ p_1 < p_2 < \dots < p_m \} \subset [1,2n+2]$ be an index set.
If $n+2\notin P$, then
\[
   X_P(F_\bull)=\{ \Sigma \in \OG' \mid \dim(\Sigma \cap
   F_{p_j}) \gequ j,  \ \ \forall \,  1 \lequ j \lequ m \} \,,
\]
while if $n+2\in P$, then
\begin{gather*}
 X_P(F_\bull)=\{ \Sigma \in \OG' \mid 
   \Sigma\cap F_n = \Sigma\cap F_{n+1}, \ 
   \dim(\Sigma \cap
   F_{p_j}) \gequ j,  \ \forall \,  1 \lequ j \lequ m \}  \\
= \{\Sigma \in \OG' \mid 
\dim(\Sigma \cap F_{p_j}) \gequ j, \ \text{{\em if} $p_j\neq n+2$}, \ 
\dim(\Sigma \cap \wt{F}_{n+1}) \gequ j, \  \text{{\em if} $p_j = n+2$}\} \,.
\end{gather*}
\end{prop}
\begin{proof}
  The sets displayed on the right hand sides of the above equations are
  closed. Following the proof of Proposition \ref{closureC}, it will
  suffice to construct, for every index set $Q \prec P$
  an index set $P'$ such that $Q \preceq P'\prec P$ and
  $X^\circ_{P'} \subset \overline{X^\circ_P}$.
  
  Choose $j$ minimal such that $q_j < p_j$ and fix a standard basis
  $\{e_1,\dots,e_{2n+2}\}$ of $V$ with respect to $F_\bull$.  If
  $2n+3-q_j \not \in P$ we  
  set $P' = \{q_1,\dots,q_j,p_{j+1},\dots,p_m\}$, and use the  
  morphism $\bP^1\to \OG'$ given by
\[
[s:t] \mapsto \Sigma_{[s:t]} =
  \langle e_{p_1},\dots,e_{p_{j-1}}, s e_{q_j} + t e_{p_j},
  e_{p_{j+1}}, \dots, e_{p_m} \rangle,
\]
as in the symplectic case. If $q_j+p_j = 2n+3$, use the same set $P'$
and the morphism
\[
[s:t] \mapsto \Sigma_{[s:t]} = \langle e_{p_1},\dots,e_{p_{j-1}}, s^2
  e_{q_j} + st(e_{n+1}+e_{n+2})-t^2 e_{p_j}, e_{p_{j+1}}, \dots,
  e_{p_m} \rangle.
\]
Finally, if $2n+3-q_j \in P$ and $q_j+p_j \neq 2n+3$,  
  set $P' = (P \smallsetminus \{p_j,2n+3-q_j\}) \cup
  \{q_j,2n+3-p_j\}$, and use the morphism
\[
[s:t] \mapsto \Sigma_{[s:t]} = \langle e_p : p \in P\cap P' \rangle
  \oplus \langle s e_{q_j} + t e_{p_j}, s e_{2n+3-p_j} - t
  e_{2n+3-q_j} \rangle. \qedhere
\]
\end{proof}

In type D, for any index set $P$, the dual index set $P^{\vee}$ is defined by
\[
p^{\vee}_j = \begin{cases}
2n+3-p_{m+1-j} & \text{if $n$ is odd or $p_j\notin\{n+1,n+2\}$}, \\
p_j & \text{if $n$ is even and $p_j\in\{n+1,n+2\}$}.
\end{cases}
\]
The next proposition is then proved exactly as in type C.

\begin{prop}
\label{Dduals}
For any index sets $P$ and $Q$, we have 
\[ 
\int_{\OG'} [X_P] \cdot [X_Q] = \delta_{Q,P^\vee} \,.
\]
\end{prop}

We say that each index set $P$ has a {\em type}, which is a number
$\type(P)\in\{0,1,2\}$. If $P\cap\{n+1,n+2\}=\emptyset$, then
$\type(P)=0$. Otherwise, $\type(P)$ is equal to $1$ plus the parity 
mod 2 of the codimension of $\Sigma\cap F_{n+1}$ in $F_{n+1}$, 
for all $\Sigma$ in the 
Schubert cell $X^{\circ}_P(F_\bull)$. Equivalently, $\type(P)$, when 
non-zero, is equal to $1$ plus the parity of the number of integers in
$[1,n+1]\smallsetminus P$.

Set $k=n+1-m$. We will define a different set $\wt{\cP}(k,n)$ of
indices for the Schubert classes in $\OG'$ which makes their
codimension apparent. As with index sets, we agree that
any $k$-strict partition $\la$ has a type in $\{0,1,2\}$. If $\la_j=k$
for some $j$, then $\type(\la)\in\{1,2\}$; otherwise, $\type(\la)=0$.  
The elements of $\wt{\cP}(k,n)$ are the $k$-strict
partitions of all possible types which are contained in an
$m\times (n+k)$ rectangle.  

Given an index set $P' = \{ p'_1 < p'_2 < \dots < p'_m \} \subset
[1,2n+2]$, the corresponding element $\la\in\wt{\cP}(k,n)$ satisfies
$\type(\la)=\type(P')$, and its underlying partition is obtained by a
prescription similar to that in Section \ref{svC}. Choose $s$ such
that $p'_s \lequ n+1 < p'_{s+1}$, and write
\[
[n+2,2n+2] \smallsetminus \{ 2n+3-p'_1, \dots,
2n+3-p'_s \} = \{ r_1 < r_2 < \dots < r_{n+1-s} \}.
\] 
Next, choose indices $1\lequ t_{s+1} < \dots < t_m \lequ n+1-s$ 
so that $p'_j = r_{t_j}$ for $s+1
\lequ j \lequ m$.  Then $P'$ maps to the $k$-strict partition
$\lambda = (\lambda_1,\dots,\lambda_m)$ given by
\[
\la_j = \begin{cases}
           n+k+1-p'_j & \text{if $1 \lequ j \lequ s$}, \\
           k+j-s-t_j & \text{if $s+1 \lequ j \lequ m$}.
\end{cases}
\]

Arguing in the same way as in Proposition \ref{P:trnslC}, we prove

\begin{prop} \label{P:trnslD}
Let $P' = \{p'_1 < \dots < p'_m\} \subset [1,2n+2]$ be an index set and
$\lambda = (\lambda_1, \dots, \lambda_m)$ the corresponding element of
$\wt{\cP}(k,n)$.  For $1 \lequ i \lequ j \lequ m$ we have

\emph{(i)} $\lambda_j \lequ k$ if and only if $p'_j > n$;

\emph{(ii)} $\la_j = k < \la_{j-1}$ if and only if $p'_j \in \{n+1,n+2\}$;

\emph{(iii)} For $i<j$, $\lambda_i+\lambda_j \lequ 2k-1+j-i$ 
if and only if $p'_i+p'_j > 2n+3$; and 

\emph{(iv)} $\lambda_j = \begin{cases}
         n+k+1-p'_j & \text{if $p'_j\lequ n+1$}, \\
         n+k+2-p'_j + \#\{i<j : p'_i+p'_j > 2n+3\} & \text{if $p'_j>n+1$}.
\end{cases}$
\end{prop}

Conversely, for every $\la\in\wt{\cP}(k,n)$, the 
associated index set $P'\subset [1,2n+2]$ satisfies
\begin{multline*} p'_j = n+k-\lambda_j + 
   \#\{\,i<j\,|\,\lambda_i+\lambda_j \lequ 2k-1+j-i\,\} \\
   {} + \begin{cases} 
      1 & \text{if $\lambda_j > k$, or $\lambda_j=k < \lambda_{j-1}$ and
        $n+j+\type(\lambda)$ is even}, \\
      2 & \text{otherwise}.
   \end{cases}
\end{multline*}
The conditions defining the Schubert variety
$X_{\la}(F_\bull)$ indexed by an element $\la$ of
$\wt{\cP}(k,n)$ are given in Section \ref{ogevenscs}.

Let $\{e_1,\dots,e_{2n+2}\}$ be a standard basis of $V$ with respect to 
the isotropic flag $F_\bull$. 
The Schubert cells indexed by a set $P'$ may then be
represented by an $m \times (2n+2)$ matrix $A=\{a_{jr}\}$ as before.
Each entry $a_{j,p'_j}$ of
$A$ is equal to 1, $a_{jr} = 0$ for $r > p'_j$, while $a_{j,p'_i} = 0$
for $i<j$.  If $r < p'_j$ and $r \not \in \{p'_i, 2n+3-p'_i\}$ for all
$i\lequ j$ then $a_{jr}$ is a free variable.  Finally, for each $i
\lequ j$ such that $p'_i+p'_j > 2n+3$, the entry $a_{j,2n+3-p'_i}$ is
determined from the free variables and the isotropicity condition on
the rows of $A$.  For example, if $m=3$, $n=5$, and $P' = \{4,8,11\}$
then the matrix $A$ has the shape
\[ A = \left[
   \begin{array}{cccccccccccc}
   *&*&*&1&0&0&0&0&0&0&0&0 \\
   *&*&*&0&=&*&*&1&0&0&0&0 \\
   *&=&*&0&=&*&*&0&=&*&1&0
\end{array}\right] \,. \]
It follows that the codimension
of $X_{\la}(F_\bull)$ in $\OG'$ is equal to $|\la|$.

For each $\la\in\wt{\cP}(k,n)$, we define a dual element $\la^\vee\in
\wt{\cP}(k,n)$ by requiring that 
\[
p'_j(\la^\vee)= 
\begin{cases}
2n+3 - p'_{m+1-j}(\la) & \text{if $n$ is odd or $p'_j(\la)\notin 
\{n+1,n+2\}$}, \\
p'_j(\la) & \text{if $n$ is even and $p'_j(\la)\in
\{n+1,n+2\}$}.
\end{cases}
\]
As in Proposition \ref{Dduals}, 
for any two elements $\la,\mu\in\wt{\cP}(k,n)$, we have 
\[ 
\int_{\OG'} [X_\lambda] \cdot [X_\mu] = \delta_{\mu,\lambda^\vee} \,.
\]

\subsection{Dual partitions}
\label{dualpar}

We now give a pictorial description of the relationship between a
$k$-strict partition $\lambda$, the corresponding index set, and the
Poincar\'e dual partition $\lambda^\vee$.  We begin our discussion in
types B and C.

Let $\Pi$ be the diagram obtained by attaching an $m \times k$
rectangle to the left side of a staircase partition with $n$ rows.
When $n=7$ and $k=3$, this looks as follows.
\[ \Pi \ \ = \ \ \ \raisebox{-36pt}{\pic{.6}{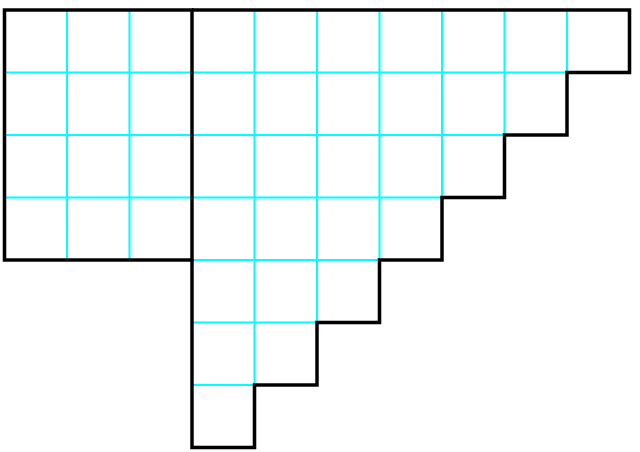}} \]

The elements of $\cP(k,n)$ are exactly the $k$-strict partitions whose
diagrams fit inside $\Pi$.  For $\lambda \in \cP(k,n)$ we let
$\floor{\lambda}_k$ be the set of boxes of $\lambda$ in columns $k+1$
through $k+n$.  If $P = \{p_1 < \dots < p_m\}$ is the index set
corresponding to $\lambda$, then the values $p_i$ which are less than
or equal to $n$ are obtained by subtracting the number of boxes in the
$i$-th row of $\floor{\lambda}_k$ from $n+1$, for $1 \lequ i \lequ
\ell_k(\lambda)$.

We will organize the boxes of the staircase partition which are
outside $\lambda$ into south-west to north-east {\em diagonals}.
Notice that exactly $m-\ell_k(\lambda)$ of these diagonals are {\em
not\/} $k$-related to one of the bottom boxes in the first $k$
columns of $\lambda$.  We will call these diagonals {\em non-related}.
In type C we obtain the integers $p_i$ which are greater than $n$ by
adding $n$ to the length of each of the non-related diagonals; in type
B we add $n+1$ to these lengths.  For example, the partition
$\lambda = (7,4,2)$ results in the (type C) index set $P = \{ 8-4,
8-1, 7+3, 7+7\} = \{4, 7, 10, 14\}$.

\[ \lambda \ = \ \ \raisebox{-53pt}{\pic{.6}{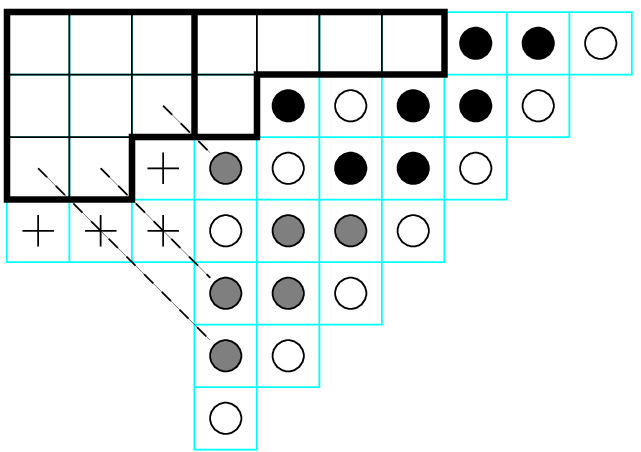}} 
   \ \ \ \ ; \ \ \ \ 
   \lambda^\vee \ = \ \ \raisebox{-22pt}{\pic{0.6}{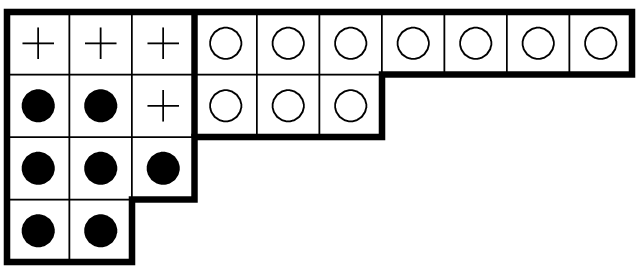}}
\]

\medskip
The parts of $\lambda^\vee$ which are larger than $k$ are obtained by
adding $k$ to the lengths of the non-related diagonals.  In other
words, these lengths are the parts of $\floor{\lambda^\vee}_k$.
Finally, for $1 \lequ j \lequ k$, the number of boxes in the $j$-th
column of $\lambda^\vee$ is equal to the number of boxes of the $j$-th
column of $\Pi$ which are outside $\lambda$, plus the length of the
diagonal that is $k$-related to the bottom box of column $j$ of
$\lambda$, minus $(k+1-j)$.  The dual of the partition $\lambda =
(7,4,2)$ is $\lambda^\vee = (10,6,3,2)$.

Consider the analogous picture in type D.  The shapes of the elements
of $\wt{\cP}(k,n)$ are the $k$-strict partitions that fit inside the
diagram $\Pi$ with $m=n+1-k$.  Let $\lambda \in \wt{\cP}(k,n)$ and let
$P' = \{p'_1 < \dots < p'_m\}$ be the corresponding index set.  The
values $p'_i$ which are less than or equal to $n$ are obtained by
subtracting the parts of $\floor{\lambda}_k$ from $n+1$.  The values
$p'_i$ which are greater than $n+1$ depend on $\type(\lambda)$ as well
as the lengths of the non-related diagonals.

The dual partition $\lambda^\vee$ can be constructed as follows.  The
parts of $\floor{\lambda^\vee}_k$ consist of the lengths of the
non-related diagonals, as in types B and C.  For $1 \lequ j \lequ
k$, the number of boxes in column $j$ of $\lambda^\vee$ is equal to
the number of boxes in the $j$-th column of $\Pi$ which are outside
$\lambda$, plus the length of the diagonal that is $k'$-related to the
bottom box of column $j$ of $\lambda$, minus $(k-j)$.  For example,
when $k=3$ and $n=7$, the dual of the partition $\lambda =
(10,8,3,2,1)$ is $\lambda^\vee = (7,5,3,1)$.

\[ \lambda \ = \ \ \raisebox{-50pt}{\pic{0.6}{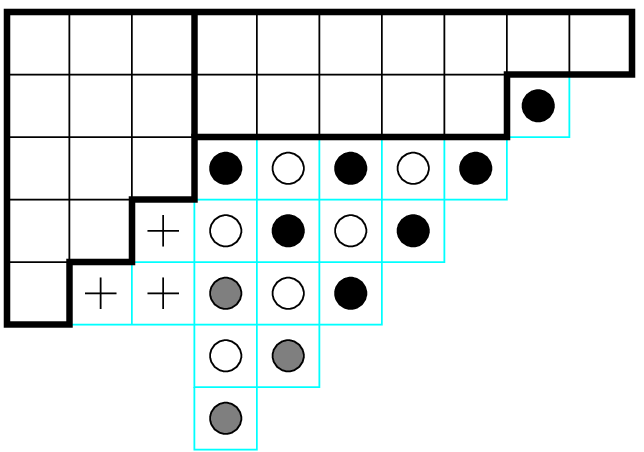}}
   \ \ \ \ ; \ \ \ \ 
   \lambda^\vee \ = \ \ \raisebox{-18pt}{\pic{0.6}{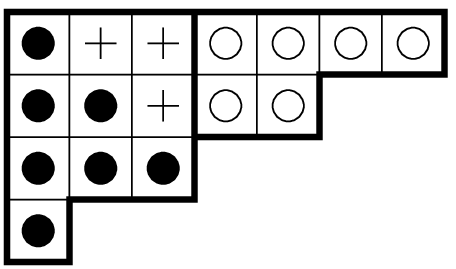}} 
\]

\medskip
We have $\type(\la^{\vee})=0$ if and only if $\type(\la)=0$.  Notice
that when $\type(\lambda)>0$, we have $\ell_k(\la) + \ell_k(\la^\vee)
= n-k$, since $p'_i(\la)\in\{n+1,n+2\}$ for some $i$. As $\type(P')$ 
is equal to $1$ plus the parity of the number of integers in 
$[1,n+1]\smallsetminus P'$, we deduce that 
$\ell_k(\la)+\type(\la) \equiv n + \ell_k(\la^\vee) + \type(\la^\vee)
\ (\text{mod}\ 2)$.  It follows that when $\type(\la)>0$, we have
$\type(\la^\vee) \equiv k + \type(\la) \ (\text{mod}\ 2)$.

\subsection{Other parametrizations}
In this short section we indicate how our notation for Schubert 
varieties compares to that used in previous related works, in particular
\cite{PRpieri}, \cite{PRpieri2}, and \cite{T}. 

Let $W_n=S_n\ltimes\Z_2^n$ be the Weyl group for the root system
$C_n$, thought of as a group of permutations with a sign attached to
each entry. The group $W_n$ is generated by simple reflections
$s_0,\ldots,s_{n-1}$, and the symmetric group $S_n$ is the subgroup of
$W_n$ generated by the transpositions $s_i=(i,i+1)$ for $i>0$.  If
$W_k$ denotes the parabolic subgroup of $W_n$ generated by $\{s_i\ |\
i\neq k\}$, then it is well known that the set $W^{(k)}\subset W_n$ of
minimal length coset representatives of $W_k$ parametrizes the
Schubert varieties in $\IG(n-k,2n)$ and $\OG(n-k,2n+1)$.

Pragacz and Ratajski \cite{PRpieri} defined a set of partition pairs
in bijection with the elements of $W^{(k)}$. The Schubert varieties
are thus parametrized by the set of all partition pairs
$\pp{\alpha}{\beta}$ with $\alpha$ contained in a $k\times (n-k)$
rectangle and $\beta$ strict such that $\beta_1\lequ n$ and
$\alpha_k\gequ \ell(\beta)$. Each such partition pair corresponds to a
unique $k$-strict partition $\lambda\in\cP(k,n)$, defined by
$\lambda=\alpha'+\beta$, where $\alpha'$ is the conjugate (or
transpose) of $\alpha$.

Let $\wt{W}_{n+1}$ be the Weyl group for the root system of type
$D_{n+1}$ and $\wt{W}^{(k)}$ the corresponding parameter set for the
Schubert varieties in $\OG(n+1-k,2n+2)$.  The translation from
elements of $\wt{W}^{(k)}$ to partitions $\la\in \wt{\cP}(k,n)$ is
similar to the above, following \cite[6.1]{T}. However in the present
work our definition of $\type(\la)$ when $\type(\la)\in \{1,2\}$
differs from the convention in \cite{T}, and is essentially that
introduced in \cite{PRpieri2}. When $\type(\la)>0$, we have
$\type(\la)=1$ if and only if the first entry of the corresponding
Weyl group element is unbarred. This relates the notation used in this
paper to the cited earlier works; for further details, we refer the
reader to \cite[4.1, 5.1, 6.1]{T}.

\section{The classical Pieri rules}
\label{Pierirule}

We present here our proofs of the classical Pieri rules for the
Grassmannians $X$ parametrizing isotropic subspaces in a vector space
$V$. It is worth noting that one can achieve a uniform proof, in all
three Lie types, of a rule for products with the Chern classes
$c_r(\cQ)$ of the universal quotient bundle over $X$. To do this, we
introduce the projectivization $\bP(\cS)$ of the universal subbundle
with the natural projections $\pi: \bP(\cS) \to X$ and $\psi:\bP(\cS) \to
\bP(V)$. The projection formula then gives
\begin{equation}
\label{projform}
\int_X c_r(\cQ) \cdot [X_P]\cdot [X_{Q^\vee}] = 
\int_{\bP(V)} c_1(\cO(1))^{m-1+r}\cdot \psi_*\pi^*
[X_P(F_\bull) \cap X_{Q^\vee}(G_\bull)], 
\end{equation}
where $P$, $Q$ are index sets and 
$F_\bull$, $G_\bull$ are isotropic flags in $V$ in general
position. 

We claim that whenever $Y_{P,Q}:= X_P(F_\bull) \cap
X_{Q^\vee}(G_\bull)$ is non-empty, then the restriction of $\psi$ to
$\pi^{-1}(Y_{P,Q})$ either (i) has positive dimensional fibers, in
which case the integral (\ref{projform}) vanishes, or (ii) is a
birational isomorphism onto a target variety which is a complete
intersection of $N$ quadrics in $\bP(V)$, when the integral is equal
to the degree $2^N$.  The rule for multiplication by $c_r(\cQ)$ in the
ring $\HH^*(X)$ follows immediately, and for the Grassmannian $\IG(m,2n)$,
this is enough to finish the proof of the Pieri rule. In the
orthogonal cases, a supplementary analysis is required to establish the
rule for multiplication with the special Schubert classes; this 
uses the quadric in $\bP(V)$ defined by the symmetric form.

\subsection{Types C and B}

Let $V \cong \C^{2n}$ be a symplectic vector space and $\IG =
\IG(m,2n)$. For any index set 
$P = \{p_1 < \cdots < p_m\}\subset [1,2n]$ and isotropic 
flag $F_\bull \subset V$, let $X_P(F_\bull)$ be the Schubert 
variety in $\IG$ defined by (\ref{schvarineq}).
We set $p_0=0$ and $p_{m+1}=2n+1$ for convenience. 

If $Q$ is another index set, then we have $Q\lequ P$ if and only if
$X_P(F_\bull) \cap X_{Q^\vee}(G_\bull) \neq \emptyset$ for all
isotropic flags $F_\bull$ and $G_\bull$, where $Q^\vee$ is the index
set dual to $Q$.  When $Q \lequ P$ we set $D(P,Q) = \{(j,c) \mid q_j
\lequ c \lequ p_j \}$, and consider this as a northwest to southeast
skew diagram of boxes.  Define a {\em cut} through the diagram
$D(P,Q)$ to be any integer $c \in [0,2n]$ such that no row of
$D(P,Q)$ contains boxes in both column $c$ and column $c+1$.
Equivalently, we have $p_j \lequ c < q_{j+1}$ for some $j$.  Let
\[
I(P,Q) = \{ c \in [0,n] : c \text{ or } 2n-c \text{ is
  a cut through } D(P,Q)\},
\] 
and let $N(P,Q)$ be the number integers
$c \in I(P,Q)$ such that $c \gequ 2$ and $c-1 \not\in I(P,Q)$.

Fix a symplectic basis $e_1,\dots,e_{2n}$ of $V$, such that $(e_i,e_j)
= \pm \delta_{i+j,2n+1}$.  We let $x_1,\dots,x_{2n} \in V^*$ be the
dual basis.  Given index sets $P$, $Q$ with $Q \lequ P$,
 we define $Z_{P,Q} \subset
\bP(V)$ to be the subvariety cut out by the following equations.

\medskip
\noindent
(a) $x_{c+1}x_{2n-c} + x_{c+2}x_{2n-c-1} + \dots + x_d x_{2n+1-d} =
0$, whenever $c < d$ are consecutive elements of $I(P,Q)$ such that
$d-c \gequ 2$.

\medskip
\noin
(b) $x_c = 0$, whenever $D(P,Q)$ has no boxes in column $c$, or a row
of $D(P,Q)$ contains exactly one box, which is located in column
$2n+1-c$.

\medskip
\noin
It follows that $Z_{P,Q}$ is an irreducible complete intersection of
degree $2^{N(P,Q)}$.

Let $F_\bull$ and $G_\bull$ be the isotropic flags defined by $F_i =
\langle e_1,\dots,e_i\rangle$ and $G_i = \langle
e_{2n+1-i},\dots,e_{2n}\rangle$, and set $Y_{P,Q} = X_P(F_\bull) \cap
X_{Q^\vee}(G_\bull) \subset \IG$. Recall that we have natural smooth
maps $\pi : \bP(\cS) \to \IG$ and $\psi : \bP(\cS) \to \bP(V)$.

\begin{lemma}
  We have that $\psi(\pi^{-1}(Y_{P,Q})) \subset Z_{P,Q}$.
\end{lemma}
\begin{proof}
  Let $\Sigma \in Y_{P,Q}$ and let $v$ be a non-zero vector in
  $\Sigma$.  We must show that $v$ satisfies the equations (a) and
  (b).  Let $c, d$ be consecutive elements of $I(P,Q)$ such that
  $|d-c| \gequ 2$, and choose cuts $c' \in \{c,2n-c\}$ and $d' \in
  \{d,2n-d\}$ for the diagram $D(P,Q)$, as well as integers $i,j$ such
  that $p_i \lequ c' < q_{i+1}$ and $p_j \lequ d' < q_{j+1}$.  By the
  choice of $c$ and $d$, we must have $i\neq j$.  After possibly
  interchanging $c$ and $d$, we may assume that $i < j$.  Since
  $\dim(\Sigma \cap F_{p_i}) \gequ i$, $\dim(\Sigma \cap
  G_{2n+1-q_{i+1}} \cap F_{p_j}) \gequ j-i$, and $\dim(\Sigma \cap
  G_{2n+1-q_{j+1}}) \gequ m-j$, we can write $v = v_1 + v_2 + v_3$
  where $v_1 \in \Sigma \cap F_{p_i}$, $v_2 \in \Sigma \cap
  G_{2n+1-q_{i+1}} \cap F_{p_j}$, and $v_3 \in \Sigma \cap
  G_{2n+1-q_{j+1}}$.  The equation (a) amounts to $(v_2,v_3)=0$ if
  $c<d$ and to $(v_1,v_2)=0$ if $c>d$.
  
  Now consider an equation $x_c = 0$ of type (b).  If $D(P,Q)$ has no
  boxes in column $c$, then $p_j < c < q_{j+1}$ for some $j$, and the
  equation is true because $\Sigma = (\Sigma \cap F_{p_j}) + (\Sigma
  \cap G_{2n+1-q_{j+1}})$.  Otherwise we have $q_j = p_j = 2n+1-c$ for
  some $j$, and the equation follows because $e_{2n+1-c} \in \Sigma$.
\end{proof}

Given two index sets $Q \lequ P$, we will write $P \to Q$ if the diagram
$D(P,Q)$ contains no $2\times 2$ squares, and whenever this diagram
contains two boxes in column $c$, it contains one box in column
$2n+1-c$.  In other words we have $P \to Q$ if and only if (i) $Q \lequ
P$, (ii) $p_j \lequ q_{j+1}$ for $1\lequ j \lequ m-1$, and (iii) whenever
$p_j = q_{j+1}$ we have $q_i < 2n+1-p_j < p_i$ for some $i$.

\begin{prop}
\label{YtoZinC}
  For index sets $Q \lequ P$ we have $\dim(\pi^{-1}(Y_{P,Q})) \gequ
  \dim(Z_{P,Q})$, with equality if and only if $P \to Q$.  When the
  latter occurs, the map $\psi : \pi^{-1}(Y_{P,Q}) \to Z_{P,Q}$ is a
  birational isomorphism.
\end{prop}
\begin{proof}
  Assume at first that $D(P,Q)$ contains no $2\times 2$ squares,
  i.e.\ $p_j \lequ q_{j+1}$ for each $j$.  In this case we will show
  that $\psi$ maps $\pi^{-1}(Y_{P,Q})$ onto $Z_{P,Q}$.  We may assume
  that $m \gequ 2$, and, after possibly replacing $(P,Q)$ with
  $(Q^\vee,P^\vee)$, that $p_1+q_m \lequ 2n+1$.  It is enough to show
  that the image of $\psi$ contains any vector $x =
  (x_1,\dots,x_{2n})$ of $Z_{P,Q}$ such that each coordinate $x_c$ is
  non-zero, unless an equation (b) says otherwise.  Notice that $x \in
  \psi(\pi^{-1}(Y_{P,Q}))$ if and only if there exists $\Sigma \in
  Y_{P,Q}$ with $x \in \Sigma$.
  
  Set $P' = \{p_2,\dots,p_m\}$ and $Q' = \{q_2,\dots,q_m\}$.  We claim
  that for some $a \in \C$, the vectors $v = (x_1,\dots,x_{p_1}-a,
  0,\dots,0) \in V$ and $v' = (0,\dots,0, a,x_{p_1+1},\dots,x_{2n})
  \in V$ satisfy $(v,v') = 0$ and $v' \in Z_{P',Q'}$.  If $p_1 = q_2$
  then $x_{2n+1-p_1} \neq 0$, provided that $p_1+p_m>2n+1$, so it is possible to
  choose $a$ such that $(v,v') = 0$, while $(v,v')=0$ for any choice of $a$
  when $p_1+p_m<2n+1$.  The vector $v'$ then satisfies
  the quadratic equations defining $Z_{P',Q'}$ because $x=v+v'$
  satisfies the equations defining $Z_{P,Q}$.  Otherwise we have $p_1
  < q_2$, and we can take $a=0$.  Since $p_1$ is a cut through
  $D(P,Q)$, it follows from the equations defining $Z_{P,Q}$ that
  $(v,v')=0$.
  
  By induction there exists $\Sigma' \in Y_{P',Q'} \subset \IG(m-1,V)$
  such that $v' \in \Sigma'$.  Since $\dim(\Sigma' \cap F_{p_{m-1}})
  \gequ m-2$ and $v' \not\in F_{p_{m-1}}$ because $x_{p_m} \neq 0$, it
  follows that $\Sigma' \subset F_{p_{m-1}} \oplus \C v'$.  Now since
  $p_1+p_{m-1} \lequ p_1+q_m$ we obtain $p_1+p_{m-1} < 2n+1$, so
  $(v,\Sigma') = 0$.  Since $\Sigma' \subset G_{2n+1-q_2}$ and
  $x_{q_1} \neq 0$, we deduce that $\Sigma = \C v \oplus \Sigma'$ is a
  point of $\IG(m,V)$.  Since $\dim(\Sigma' \cap F_{p_j}) \gequ j-1$
  for each $j$ and $\dim(\Sigma' \cap G_{2n+1-q_j}) \gequ m+1-j$
  for $j\gequ 2$, it follows that $\Sigma \in Y_{P,Q}$, as required.
  
  Let $U_{P,Q} \subset Y_{P,Q}$ be the open subset of isotropic spaces
  $\Sigma = \langle u_1,\dots,u_m\rangle$, for which each vector $u_j$ can
  be written as $u_j = a_{j,q_j} e_{q_j} + \dots + a_{j,p_j} e_{p_j}$,
  with $a_{j,c} \neq 0$ unless $Z_{P,Q}$ satisfies the equation $x_c =
  0$.  We claim that $U_{P,Q} \neq \emptyset$.  In fact, by induction
  on $m$ there exists $\Sigma' = \langle u_2,\dots,u_m\rangle \in
  U_{P',Q'}$.  If $q_1=p_1$ then we set $u_1 = e_{p_1}$; if $p_1+p_m >
  2n+1$, so that $q_m < 2n+1-p_1 < p_m$, we furthermore replace the
  coordinate $a_{m,2n+1-p_1}$ of $u_m$ with zero.  Then $\Sigma =
  \langle u_1,\dots,u_m\rangle$ is a point of $U_{P,Q}$.  Otherwise we have
  $q_1<p_1$.  The condition that $(u_1,u_m) = 0$ amounts to the linear
  equation $a_{m,2n+1-q_1} a_{1,q_1} + \dots + a_{m,2n+1-p_1}
  a_{1,p_1}= 0$ in the coordinates of $u_1$.  If $p_1+p_m<2n+1$ then
  this equation is trivial, while if $p_1+p_m>2n+1$, then since $\Sigma' \in
  U_{P',Q'}$ we have $a_{m,2n-p_1} \neq 0$ and $a_{m,2n+1-p_1} \neq
  0$, so this equation has a solution for which $a_{1,c} \neq 0$ for
  all $q_1 \lequ c \lequ p_1$, and $\Sigma = \langle u_1,\dots,u_m\rangle 
  \in U_{P,Q}$, as required.
  
  The intersection of two Schubert cells in general position is
  irreducible by results of Deodhar \cite[Cor.~1.2, Prop.~5.3(iv)]{De},
  and this immediately implies the irreducibility of the intersection of
  two Schubert varieties in general position.  It
  follows that $U_{P,Q}$ is a dense open subset of $Y_{P,Q}$.
  We note that in our special case, this can also be
  checked directly, but leave this as an exercise to the reader.
  We deduce that $\pi^{-1}(Y_{P,Q})$ has a dense open subset of pairs
  $(\C x, \Sigma)$, for which $\Sigma = \langle u_1,\dots,u_n\rangle \in
  U_{P,Q}$ with $u_j \in F_{p_j}\cap G_{2n+1-q_j}$ and $x =
  u_1+\dots+u_n$.
  
  Suppose the diagram $D(P,Q)$ contains two boxes in some column $c$
  and no boxes in column $2n+1-c$, and let $\Sigma =
  \langle u_1,\dots,u_m\rangle \in U_{P,Q}$ and $x = u_1+\dots+u_m$ be as
  above.  Then $c = p_j = q_{j+1}$ for some $j$, and $\Sigma_t =
  \langle u_1,\dots,u_{j-1}, u_j+t\,e_c, u_{j+1}-t\,e_c,
  u_{j+2},\dots,u_m\rangle$ is a point of $Y_{P,Q}$ with $x \in \Sigma_t$
  for each $t \in \C$.  This shows that $\psi^{-1}(x)$ has positive
  dimension, so $\dim(\pi^{-1}(Y_{P,Q})) > \dim(Z_{P,Q})$.
  
  On the other hand, if the diagram $D(P,Q)$ contains a box in column
  $2n+1-c$ whenever there are two boxes in column $c$, then we have $P
  \to Q$, and we must show that $\psi : \pi^{-1}(Y_{P,Q}) \to Z_{P,Q}$
  is birational.  It is enough to show that if $\Sigma =
  \langle u_1,\dots,u_m\rangle \in U_{P,Q}$ with $u_j \in F_{p_j} \cap
  G_{2n+1-q_j}$, then $\psi^{-1}(u_1+\dots+u_m) \cap
  \pi^{-1}(U_{P,Q})$ contains a single point, or equivalently,
  $\Sigma$ is the only point of $U_{P,Q}$ containing the vector $x =
  u_1+\dots+u_m$.  In fact, we claim this is true with $U_{P,Q}$ replaced
  by the larger open subset $W_{P,Q}\subset Y_{P,Q}$ of isotropic spaces
  $\Sigma=\langle u_1,\ldots,u_m \rangle$, for which each vector $u_j$ can
  be written as $u_j=a_{j,q_j}e_{q_j}+\cdots+a_{j,p_j}e_{p_j}$ with
  $a_{j,c}\ne 0$ unless $2n+1-q_1<c<p_j$ or $Z_{P,Q}$ satisfies the
  equation $x_c=0$.  Suppose $\Sigma=\langle u_1,\ldots,u_m \rangle\in W_{P,Q}$
  and $x=u_1+\cdots+u_m$ is also contained in $\Sigma' =
  \langle u'_1,\dots,u'_m\rangle \in W_{P,Q}$, where $u'_j \in F_{p_j}\cap
  G_{2n+1-q_j}$.  Then $x = a_1 u'_1 + \dots + a_m u'_m$ where $a_j
  \in \C$.  We claim that $a_1 u'_1 = u_1$.  This is clear if $p_1 <
  q_2$.  Otherwise $p_1=q_2$, and we can write $a_1 u'_1 = u_1 + b
  e_{p_1}$ for some $b \in \C$.  Since $0 = (a_1 u'_1, x) =
  b\,(e_{p_1}, x) = b\, x_{2n+1-p_1}$ and $x_{2n+1-p_1} \neq 0$, it
  follows that $b=0$.  Now since $\langle u_2,\dots,u_m\rangle$ and
  $\langle u'_2,\dots,u'_m\rangle$ are points of $W_{P',Q'} \subset
  \IG(m-1,V)$, both containing $u_2+\dots+u_m$, it follows by
  induction on $m$ that $\langle u_2,\dots,u_m\rangle =
  \langle u'_2,\dots,u'_m\rangle$.  Therefore $\Sigma' = \Sigma$ as
  required.
  
  It remains to prove that if the diagram $D(P,Q)$ contains a $2\times
  2$ square, then $\dim(\pi^{-1}(Y_{P,Q})) > \dim(Z_{P,Q})$.  We do
  this by induction on $\dim(Y_{P,Q}) = |Q| - |P|$.
  Choose $j$ minimal such that $p_j > q_{j+1}$, i.e.\ $j$ is the top
  row of $D(P,Q)$ containing (the upper half of) some $2\times 2$
  square.  This implies that $p_{j-1} \lequ q_j < q_{j+1}$.
  
  Assume first that $2n+1-r \in P$ for all $r \in [q_{j+1}, p_j-1]$.
  Then choose $i$ such that $p_i = 2n+2-p_j$, and notice that $p_l =
  p_i-i+l$ for all $l \in [i, l_0]$, where $l_0=i+p_j-q_{j+1}-1$.
  This implies that $q_i < p_i$, since otherwise we would get $q_{l_0}
  = p_{l_0} = 2n+1-q_{j+1}$.  Set $P' = (P \smallsetminus \{p_i,p_j\})
  \cup \{p_i-1,p_j-1\}$.  Then we have $\dim(Y_{P',Q}) = \dim(Y_{P,Q})
  - 1$ and $\dim(Z_{P',Q}) \gequ \dim(Z_{P,Q}) - 1$.  We distinguish
  three cases.  If $p_i < q_{i+1}$ then the diagram $D(P',Q)$ has no
  boxes in column $p_i$, while it has two boxes in column $p_j-1$.  If
  $p_i = q_{i+1}$ then $q_{j+1} \lequ p_j-2$, which implies that
  $D(P',Q)$ contains a $2\times 2$ square.  In both of these cases
  we obtain by induction that $\dim(\pi^{-1}(Y_{P',Q})) >
  \dim(Z_{P',Q})$.  Finally, if $q_{i+1} < p_i$ then $Z_{P',Q} =
  Z_{P,Q}$, and by induction we have $\dim(\pi^{-1}(Y_{P',Q})) \gequ
  \dim(Z_{P',Q})$.  In all three cases we deduce that
  $\dim(\pi^{-1}(Y_{P,Q})) < \dim(Z_{P,Q})$ as required.
  
  If the above assumption fails, then choose $r < p_j$ maximal such
  that $2n+1-r \not\in P$.  Then $r \gequ q_{j+1}$.  Set $P' = (P
  \smallsetminus \{p_j\}) \cup \{r\}$.  In this case we obtain
  $\dim(Y_{P',Q}) = \dim(Y_{P,Q}) - 1$ and $Z_{P',Q} = Z_{P,Q}$, and
  the induction hypothesis shows that $\dim(\pi^{-1}(Y_{P',Q})) \gequ
  \dim(Z_{P',Q})$.  This completes the proof.
\end{proof}

\begin{thm}[Pieri rule for $\IG(m,2n)$] 
\label{indexpieriC}
For any index set $P$ and integer
 $r \in [1,n+k]$, we have
 \[
\sigma_r \cdot [X_P] = \sum 2^{N(P,Q)} \,[X_Q],
\]
where the sum is over all 
index sets $Q$ such that $P \to Q$ and $|Q| = |P|+r$.
\end{thm}
\begin{proof}
For any index set $Q$, the coefficient of $[X_Q]$ in the expansion of
the product $\sigma_r \cdot [X_P]$ is equal to $\int_{\IG} c_r(\cQ)
\cdot [X_P]\cdot [X_{Q^\vee}]$. The result now follows from
Proposition \ref{YtoZinC} and the projection formula, as explained 
earlier.
\end{proof}

Now suppose $V \cong \C^{2n+1}$ is an orthogonal space and consider
$\OG = \OG(m,2n+1)$. As noted in Section \ref{ogoddclpieri}, the Pieri
rule for $\OG$ is equivalent to the corresponding rule for $\IG$.
Given any two index sets $P$, $Q$ we define the diagram $D(P,Q)$ and
the relation $P \to Q$ as before, and set
\[
I'(P,Q) = \{ c \in [0,n] : c \text{ or } 2n+1-c \text{ is
  a cut through } D(P,Q)\} \cup \{n+1\}.
\]
Moreover, let $N'(P,Q)$ be the number (respectively, one less than the
number) of integers $c \in I'(P,Q)$ such that 
$c\gequ 2$ and $c-1 \not\in I'(P,Q)$, if $r\lequ k$ (respectively,
if $r>k$).

\begin{thm}[Pieri rule for $\OG(m,2n+1)$] 
\label{indexpieriB}
For any index set $P$ and integer $r \in [1,n+k]$, we have
 \[
\tau_r \cdot [X_P] = \sum 2^{N'(P,Q)} \,[X_Q],
\]
where the sum is over all index sets $Q$ such that $P \to Q$ and 
$|Q| = |P|+r$.
\end{thm}

\subsection{Type D}

Let $V \cong \C^{2n+2}$ be an orthogonal vector space and $\OG' =
\OG(m,2n+2)$. For any  isotropic flag $F_\bull \subset V$ and index set 
$P$, we have an open Schubert cell
\[ X_P^\circ(F_\bull) = \{ \Sigma \in \OG \mid
   \Sigma \cap F_{p_j} \supsetneq \Sigma\cap F_{p_j-1} 
   ~\forall 1 \lequ j \lequ m \} \,.
\]
The Schubert variety $X_P(F_\bull)$ is the closure of this set.

We will use the order $Q \preceq P$ on the index sets defined in 
Section \ref{svD}; this is equivalent to the condition
$X_P(F_\bull) \cap X_{Q^\vee}(G_\bull) \neq \emptyset$ for all
isotropic flags $F_\bull$ and $G_\bull$.
When $Q \preceq P$ we set $D(P,Q) = \{(j,c) \mid q_j \lequ c \lequ p_j\}$. 
Define a {\em cut\/} through the diagram $D(P,Q)$ to be any integer $c
\in [0,2n+2]$ such that $p_j \lequ c < q_{j+1}$ for
some $j$.  Set
\[I'(P,Q) = \{ c \in [0,n] : c \text{ or } 2n+2-c
\text{ is a cut through } D(P,Q) \} \cup \{n+1\},\]
and let 
$N(P,Q)$ be the number of integers $c \in I'(P,Q)$
such that $c \gequ 2$ and $c-1 \not\in I'(P,Q)$.

Given index sets $Q \preceq P$ for $\OG'$ we will write $P \to Q$ if 
(i) the diagram $D(P,Q)$ contains
no $2\times 2$ squares, except that a single $2\times 2$ square is
allowed across the middle (in columns $n+1$ and $n+2$), (ii)
whenever column $c$ of $D(P,Q)$ contains two boxes, column
$2n+3-c$ contains at least one box, and (iii) $D(P,Q)$ cannot have exactly 
3 boxes in columns $n+1$ and $n+2$.

Fix an orthogonal basis $e_1,\dots,e_{2n+2}$ of $V$ such that $(e_i,e_j)
= \delta_{i+j,2n+3}$ and $\langle e_1,\ldots,e_{n+1} \rangle$ is in the
same family as our chosen maximal isotropic subspace $L$. We let
$x_1,\dots,x_{2n+2} \in V^*$ be the dual basis.  For index sets $Q \preceq P$
we define $Z_{P,Q} \subset \bP(V)$ to be the subvariety cut out by the
following equations.

\medskip
\noin
(a) $x_{c+1}x_{2n+2-c} + x_{c+2}x_{2n+1-c} + \dots + x_d x_{2n+3-d} = 0$,
whenever $c < d$ are consecutive elements of $I'(P,Q)$ such that $d-c
\gequ 2$.

\medskip
\noin (b) $x_c = 0$, whenever $D(P,Q)$ has no boxes in column $c$; or
some row of $D(P,Q)$ contains exactly one box, which is located in
column $2n+3-c$; or $c\in\{n+1, n+2\}$ and only one row of $D(P,Q)$
contains a box in column $c$, and this row starts or terminates at
column $2n+3-c$.

\medskip
\noin
We see that $Z'_{P,Q}$ is an irreducible complete intersection of
degree $2^{N(P,Q)}$.

Let $F_\bull$ and $G_\bull$ be the isotropic flags defined by $F_i =
\langle e_1,\ldots,e_i\rangle$, for each $i$, and $G_i = \langle
e_{2n+3-i},\ldots,e_{2n+2}\rangle$ for $1\lequ i \lequ n$, while
\[
G_{n+1}=
\begin{cases} 
 \langle e_{n+2},\ldots, e_{2n+2} \rangle & 
 \text{if $n$ is odd}, \\ 
 \langle e_{n+1}, e_{n+3}, \ldots, e_{2n+2} \rangle & 
 \text{if $n$ is even}.
\end{cases}
\] 
Set $Y_{P,Q} = X_P(F_\bull)
\cap X_{Q^\vee}(G_\bull) \subset \OG'$.  Let $\cS$ be the tautological
subbundle over $\OG'$, and let $\pi : \bP(\cS) \to \OG'$ and $\psi :
\bP(\cS) \to \bP(V)$ be the natural projections. Arguing as in the
previous section, we prove the following.

\begin{prop}
  For index sets $Q \preceq P$ we have $\dim(\pi^{-1}(Y_{P,Q})) \gequ
  \dim(Z_{P,Q})$, with equality if and only if $P \to Q$.  When the
  latter occurs, the map $\psi : \pi^{-1}(Y_{P,Q}) \to Z_{P,Q}$ is a
  birational isomorphism.
\end{prop}

\begin{cor}
\label{indexcorD}
For any index set $P$ and integer $r \in [1,n+k]$, we have
 \[
c_r(\cQ) \cdot [X_P] = \sum 2^{N(P,Q)} \,[X_Q],
\]
where the sum is over all 
index sets $Q$ such that $P \to Q$ and $|Q| = |P|+r$.
\end{cor}

We now refine Corollary \ref{indexcorD} to obtain the Pieri rule for
products by the special Schubert classes in $\OG'$.  Define 
\begin{gather*}
S= \{i \in [1,n+1]\ :\  q_j\lequ i \lequ p_j \ \text{for
some $j$}\}, \\
S'= \{p\in P\ :\ p \gequ n+2 \ \text{and} \  2n+3-p \in S \},
\end{gather*}
and set $h(P,Q) = |S|+|S'|+n$.
Let $N'(P,Q)=N(P,Q)$ (respectively, $N'(P,Q)=N(P,Q)-1$) if $r \lequ k$
(respectively, if $r>k$). If $r \neq k$, then set $\delta_{PQ}=1$. If
$r=k$ and $N'(P,Q)>0$, then set
\[
\delta_{PQ} =\delta'_{PQ} = 1/2,
\]
while if $N'(P,Q)=0$, define
\[
\delta_{PQ} = 
\begin{cases}
1 & \text{if $h(P,Q)$ is odd}, \\
0 & \text{otherwise}, 
\end{cases}
\qquad \mathrm{and}
\qquad 
\delta'_{PQ} = 
\begin{cases}
1 & \text{if $h(P,Q)$ is even}, \\
0 & \text{otherwise}. 
\end{cases}
\]

\begin{thm}[Pieri rule for $\OG(m,2n+2)$] 
\label{indexpieriD}
For any index set $P$ and integer $r \in [1,n+k]$, 
we have
\begin{equation}
\label{indexDpieri}
\tau_r \cdot [X_P] = \sum \delta_{PQ}\,2^{N'(P,Q)}\,[X_Q],
\end{equation}
where the sum is over all index sets $Q$ such that $P \to Q$ and 
$|Q| = |P|+r$. Furthermore, the product
$\ta'_k\ [X_P]$ is obtained by replacing $\delta_{PQ}$
with $\delta'_{PQ}$ throughout.
\end{thm}
\begin{proof}
When $r\neq k$, equation (\ref{indexDpieri}) follows immediately from 
(\ref{ctotau}) and Corollary \ref{indexcorD}. For the remaining terms
we must compute the integrals
\[
I_1=\int_{\OG'} \ta_k \cdot [X_P]\cdot [X_{Q^\vee}]
\qquad 
\mathrm{and}
\qquad
I_2=\int_{\OG'} \ta'_k \cdot [X_P]\cdot [X_{Q^\vee}].
\]
The argument uses the projection formula and the 
quadric hypersurface $W\subset\bP(V)$ defined by the orthogonal form on
$V$. Let $h=c_1(\cO_{\bP(V)}(1)\vert_W)$ denote the hyperplane class in
$\HH^*(W)$, so that $h^n= e + f$, where $e$ is the class of the ruling
$\bP(L)\subset W$, and $f$ is the class of the opposite ruling. 
If $\theta:\bP(\cS)\to W$
denotes the natural morphism, then
\begin{equation}
\label{ef}
\pi_*\theta^*e = \begin{cases}
  \ta_k & \text{if $n$ is even}, \\
  \ta'_k & \text{if $n$ is odd},
\end{cases}
\end{equation}
while the opposite relations hold for $\pi_*\theta^*f$.

Let $\iota : W \hookrightarrow \bP(V)$ be the inclusion map, so that 
$\psi = \iota\theta$. The
image of the map $\iota^*: \HH^{2n}(\bP(V)) \to \HH^{2n}(W)$ consists of
classes with equal coefficients in the rulings $e$ and $f$.  If the
degree $2^{N(P,Q)}$ of $Z_{P,Q}$ is greater than one, then the class
$\theta_*\pi^*[Y_{P,Q}]$ must lie in the image of $\iota^*$. Applying
the projection formula, we deduce from this and the relations
(\ref{ef}) that $I_1=I_2=2^{N(P,Q)-1}$, as required.

On the other hand, if $Z_{P,Q}$ is a linear subspace of $\bP(V)$ lying
inside the quadric $W$, it suffices to determine its ruling class. Note
that in this situation we have $I'(P,Q)=[0,n+1]$ and $N(P,Q)=0$.  
Let $T= \{2n+3- p\ : \ p\in S'\}$; it follows 
from the linear equations (b) defining $Z_{P,Q}$ that $Z_{P,Q} \cap
\langle e_1,\ldots, e_{n+1} \rangle =  
\langle e_i\ |\ i \in S\ssm T\rangle$.
To see this, note that if $p_j\in S'$ is such that $p_j>n+2$, then 
$q_j=p_j$, i.e., there is only one box in row $j$ of $D(P,Q)$. Indeed,
if $q_j<p_j$, then $p_j-1$ is not a cut, and since there is a box in
column $2n+3-p_j$ of $D(P,Q)$ and $2n+3-p_j\neq p_i$ for all $i$, we 
see that $2n+2-(p_j-1)$ is not a cut either. This contradicts the fact
that $I'(P,Q)=[0,n+1]$. 

We deduce that if $h(P,Q)$ is odd, then $Z_{P,Q}$ and $\bP(L)$ are in 
the same family, and in opposite families if $h(P,Q)$ is even. 
In $\HH^*(W)$ we have the relations 
$e^2=f^2=0$ and $ef=eh^n$, if $n$ is odd, and
$e^2=f^2=eh^n$ and $ef=0$, if $n$ is even. Using the projection formula
again, we conclude that $I_1=1$ and $I_2=0$, if $h(P,Q)$ is odd, with 
the roles reversed if $h(P,Q)$ is even. This completes the proof.
\end{proof}

\subsection{Proofs of Theorems \ref{T:pieriC} and \ref{thm.ogevenpieri}}
\label{indtopar}

We show how to derive Theorem~\ref{T:pieriC} from Theorem
\ref{indexpieriC}, by going from index sets to partitions. 
It suffices to establish the equivalence of the
Pieri relations ``$\lambda \to \mu$'' and ``$P \to Q$'', together with
the equality of the corresponding intersection multiplicities.  This
is the content of the following two propositions.

\begin{prop}
\label{pqlm}
  Let $P$ and $Q$ be index sets for $\IG(n-k,2n)$ and let $\lambda$
  and $\mu$ be the corresponding $k$-strict partitions.  Then $P \to
  Q$ if and only if $\lambda \to \mu$. 
\end{prop}
\begin{proof}
  The assumptions of the proposition tell us that $\lambda_j =
  n+k+1-p_j+|A_j|$ and $\mu_j = n+k+1-q_j+|B_j|$ for each $j$, where
  $A_j = \{i<j : p_i+p_j > 2n+1\}$ and $B_j = \{i<j : q_i+q_j >
  2n+1\}$.
  
  Assume that $P \to Q$.  We first show that $\mu$ can be obtained by
  removing a horizontal strip from the first $k$ columns of $\lambda$
  and adding a horizontal strip to the result.  We need to check that
  $\lambda_j-1 \lequ \mu_j \lequ \lambda_{j-1}$ for each $j$, and
  $\lambda_j \lequ \mu_j$ when $\lambda_j > k$.  To see that $\mu_j
  \lequ \lambda_{j-1}$, notice that if $i \in B_j \smallsetminus
  A_{j-1}$ and $i < j-1$ then $2n+1-q_j < q_i \lequ p_i <
  2n+1-p_{j-1}$, which can hold for at most $q_j-p_{j-1}-1$ integers
  $i$.  If $p_{j-1}<q_j$ we therefore get $|B_j| \lequ
  |A_{j-1}|+q_j-p_{j-1}$ and $\lambda_{j-1}-\mu_j =
  q_j-p_{j-1}+|A_{j-1}|-|B_j| \gequ 0$.  If $p_{j-1} = q_j$, then the
  condition $P\to Q$ implies that $q_i < 2n+1-q_j < p_i$ for some $i$.
  If $i < j-1$ then $i \in A_{j-1} \smallsetminus B_j$, while if $i
  \gequ j-1$ then $q_{j-1}+q_j \lequ q_i+q_j < 2n+1$, so $j-1 \not\in
  B_j$.  It follows that $|B_j| \lequ |A_{j-1}|$ and $\mu_j \lequ
  \lambda_{j-1}$, as required.  If $\lambda_j > k$ then $q_j \lequ p_j
  \lequ n$, which implies that $\lambda_j = n+k+1-p_j \lequ n+k+1-q_j =
  \mu_j$.  In general we note that if $i<j-1$ satisfies that $i \in
  A_j$ and $i+1 \not\in B_j$, then $2n+1-p_j < p_i \lequ q_{i+1} <
  2n+1-q_j$, which is true for at most $p_j-q_j-1$ integers $i$.  If
  $q_j < p_j$, this implies that $|B_j| \gequ |A_j| - p_j+q_j$ and
  $\mu_j \gequ \lambda_j$.  If $q_j = p_j$ then $|B_j| \gequ |A_j| - 1$,
  which implies that $\mu_j \gequ \lambda_j-1$, as required.
  
  We next verify that condition (2) of Definition~\ref{D:pieriarrow}
  holds.  Assume that the box $(j,\lambda_j)$ is not in $\mu$, i.e.\ 
  $\mu_j = \lambda_j-1$.  Then the above analysis shows that $q_j=p_j$
  and $|B_j| = |A_j| - 1$.  If we set $i = \min(A_j)$, then $A_j =
  [i,j-1]$, $B_j = [i+1,j-1]$, and $\lambda_j = n+k+1-p_j+j-i$.  Since
  $A_i \subset A_j$ we see that $A_i = \emptyset$, so $\lambda_i =
  n+k+1-p_i$ and $\mu_i = n+k+1-q_i$.  Set $c = p_j-n+k+1$.  Then
  $c-k-1+i = k+1-\lambda_j+j$, so the box $(j,\lambda_j)$ is
  $k$-related to $(i,c)$.  Since $i \in A_j \smallsetminus B_j$ we
  furthermore obtain $q_i < 2n+1-p_j < p_i$, which is equivalent to
  $\lambda_i < c-1 < \mu_i$.  We deduce that the boxes $(i,c-1)$ and
  $(i,c)$ belong to $\mu \smallsetminus \lambda$, and these boxes are
  $k$-related to $(j-1,\lambda_j)$ and $(j,\lambda_j)$.
  
  For condition (2) it remains to prove that $\mu_{i+1} \lequ c-3$ and
  $\lambda_{i-1} \gequ c+1$.  Notice that $q_{i+1} + q_j \neq 2n+2$;
  otherwise we obtain from $2n+1-p_j < p_i \lequ q_{i+1}$ that $p_i =
  q_{i+1}$, so we must have $q_l < 2n+1-p_i < p_l$ for some $l$, and
  this would imply that $q_l < q_j = p_j \lequ p_l$.
  We therefore obtain $c-3-\mu_{i+1} = q_{i+1}+q_j-2n-3 \gequ 0$.
  Similarly, if $p_{i-1}+p_j = 2n$ then $p_{i-1} \lequ q_i < 2n+1-q_j$
  shows that $p_{i-1}=q_i$, so $q_l < 2n+1-q_i < p_l$ for some $l$,
  but this implies that $q_l \lequ q_j=p_j < p_l$.
  It follows that $\lambda_{i-1}-c-1 = 2n-1-p_{i-1}-p_j \gequ 0$,
  as required.
  
  We finally verify that condition (1) holds.  Assume that $\lambda$
  and $\mu$ have equally many boxes in column $c$, where $c \lequ k$,
  and let $j$ be the number of boxes in this column.  Then we have
  $\mu_{j+1} < c \lequ \lambda_j$.  Now suppose the box $(j,c)$ is
  $k$-related to two boxes $(i,c')$ and $(i-1,c'+1)$ of $\mu
  \smallsetminus \lambda$, where $c' > k$.  Then $c+c' = 2k+2+j-i$ and
  $c' = \mu_i = \lambda_{i-1}$.  Since $c' > k$, the latter equality
  implies that $p_{i-1} = q_i$.  Since $\lambda_{i-1}+\lambda_j \gequ
  c'+c > 2k+j-i$, it follows that $q_i+p_j = p_{i-1}+p_j < 2n+1$.
  Similarly, since $\mu_i+\mu_{j+1} \lequ c'+c-1 = 2k+(j+1)-i$, we
  obtain $q_i+q_{j+1} > 2n+1$.  But now we have $p_j < 2n+1-q_i <
  q_{j+1}$, so the diagram $D(P,Q)$ has no boxes in column $2n+1-q_i$,
  a contradiction.
  
  Assume now that $\lambda \to \mu$. We claim that if $i\leq j$ and
  $\la_i+\la_j >2k+j-i$, then $\mu_i+\mu_j >2k+j-i$. If not, then
  since $\la_i>k$, we obtain $\la_i+\la_j-1\leq \mu_i+\mu_j\leq
  2k+j-i<\la_i+\la_j$. Therefore $\mu_i=\la_i$, $\mu_j=\la_j-1$, and
  $\la_i+\la_j = 2k+1+j-i$. It follows that box $(j,\la_j)$ is
  $k$-related to box $(i+1,\la_i)$, contradicting (2) of Definition
  \ref{D:pieriarrow}. The claim implies that $B_j\subset A_j$ for
  every $j$.

  We must show that $P \to Q$, and start by checking that $Q \lequ P$.
  If this is false, then choose $j$ such that $q_j > p_j$, and note
  that $\lambda_j - \mu_j = |A_j| - |B_j| + q_j - p_j \gequ 1$.  The
  assumption $\lambda \to \mu$ then implies that $\mu_j =
  \lambda_j-1$, so $A_j = B_j$ and $\lambda_j \lequ k$.  Furthermore,
  the boxes $(j,\lambda_j)$ and $(j-1,\lambda_j)$ must be $k$-related
  to boxes $(i,c)$ and $(i,c-1)$ of $\mu \smallsetminus \lambda$, with
  $c = 2k+2-\lambda_j+j-i$.  We obtain $\lambda_i < c-1 < \mu_i$,
  which implies that $\lambda_i+\lambda_j \lequ 2k+j-i$ and
  $\mu_i+\mu_j > 2k+j-i$.  But this means that $i \in A_j
  \smallsetminus B_j$, a contradiction.
  
  Now suppose that $D(P,Q)$ contains a $2\times 2$ square, and
  choose $j$ minimal such that $p_j > q_{j+1}$.  If $i \in A_j$
  satisfies that $i+1 \not\in B_{j+1}$, then $2n+1-p_j < p_i \leq
  q_{i+1} < 2n+1-q_{j+1}$, which is true for at most $p_j-q_{j+1}-1$
  integers $i$.  It follows that $\mu_{j+1}-\lambda_j =
  |B_{j+1}|-|A_j|+p_j-q_{j+1} > 0$, contradicting that $\mu
  \smallsetminus \lambda$ is a horizontal strip.  This shows that
  $D(P,Q)$ contains no $2\times 2$ squares.
  
  Finally, suppose that $p_j=q_{j+1}$ for some $j<n-k$; we will show that
  there exists an $i$ such that $q_i < 2n+1-p_j < p_i$. Equivalently, 
  assuming that $\la_j-\mu_{j+1} =|A_j|-|B_{j+1}|\gequ 0$, we must prove 
  that
\[
\la_j+\la_i \lequ 2k+|j-i| \ \ \ (\text{I1})  \ \ \ 
\ \text{and} \ \ \ \ 
\mu_{j+1} + \mu_i > 2k+ |j+1-i|  \ \ \ (\text{I2})
\]
  for some $i$. 

  If $p_j > n$, then $i=j$ satisfies (I1). Any $i\in A_j\ssm B_{j+1}$
  solves our problem, so assume that $A_j \subset B_{j+1}$. Since
  $|A_j| \gequ |B_{j+1}|$, we must have $A_j=B_{j+1}$. Hence $j\notin
  B_{j+1}$, and thus $i=j$ also satisfies (I2).

  If $p_j \lequ n$, then
  $\la_j$ and $\mu_{j+1}$ are both greater than $k$, and hence $A_j = 
  B_{j+1} = \emptyset$. Therefore $\mu_j > \la_j=\mu_{j+1}> k$, and 
  the boxes $x=(j,\la_j+1)$ and $y=(j+1, \la_j)$ are both in $\mu\ssm\la$. 
  Set $c=n+k+1-\la_j-j$.  If $\la_{n-k}\gequ c$ then the boxes $x$ and $y$ are
  $k$-related to the box $(n-k,c)$, and (1) of Definition \ref{D:pieriarrow}
  is violated. So $\la_{n-k}\lequ c-1=2k+(n-k)-j-\la_j$, and so $i=n-k$
  satisfies (I1). 

  Let $i>j$ be minimal such that (I1) holds; we claim 
  that (I2) is also true for this $i$. Note that $\la_i \lequ k$. If 
  $\la_j < k+i-j$, then by minimality of $i$ we have $\la_{i-1}>k$, and
  this is impossible because $\la$ is $k$-strict.
  It follows that $c:=2k+i-j-\la_j$ is at most $k$. The box 
  $(i-1,c)$ lies in $\la$; otherwise $\la_{i-1}\lequ c-1$ contradicts 
  our choice of $i$. If (I2) is false, then $\mu_i < c$. Since
  $(i-1,c)$ is $k$-related to both $x$ and $y$, we again contradict
  Definition \ref{D:pieriarrow}. 
\end{proof}

Suppose now that $P$ and $Q$ are index sets for $\IG(n-k,2n)$ with 
$P\to Q$. Recall that a {\em cut} is an integer $c\in [0,2n]$ such that
$p_i\lequ c < q_{i+1}$ for some $i$. We call an element $c\in [0,2n]$
which is not a cut a {\em crossing}. It is easy to see that the set of
crossings is equal to $\cup_i[q_i,p_i)$. Indeed, $c$ is a crossing if and only
if $c \gequ p_j$ implies $c \gequ q_{j+1}$ for all $j$. If $q_i\lequ c <p_i$
then if $c \gequ p_j$ we have $j+1\lequ i$, so $c\gequ q_i\gequ q_{j+1}$. 
Conversely, if $c$ is a crossing then set $i=\max(j\ |\ 
q_j\lequ c)$; then we have $c < p_i$.

Let $J(P,Q)= \{ e\in [0,n-1]\ |\ n-e \text{ and } n+e \text{ are both
crossings}\}$; then $N(P,Q)= \#\{e \in [1,n-1]\ |\ e\in J(P,Q) \text{
and } e-1\notin J(P,Q)\}$.  Recall that $\A$ is the set of boxes of
$\mu\ssm \la$ in columns $k+1$ through $k+n$ which are {\em not}
mentioned in (1) or (2) of Definition \ref{D:pieriarrow}.

\begin{lemma}
\label{fundbij}
For $0\lequ e \lequ n-1$, we have $e\in J(P,Q)$ if and only if there 
exists a box in column $k+1+e$ of $\A$. 
\end{lemma}
\begin{proof}
Suppose that there exists a box $B$ in position $(i,k+1+e)$ of $\mu\ssm\la$. 
Then $\mu_i\gequ k+1+e > \la_i$, hence $q_i\lequ n-e< p_i$ and thus
$n-e$ is a crossing. Suppose that $B$ is in $\A$; we will show that 
$n+e$ is then also a crossing. Now either (i) $B$ is $k$-related to a box
$(r,c)$ with $1\lequ c \lequ k$ such that $\mu\ssm\la$ has a box in 
position $(r+1,c)$, or else (ii) $B$ is not related to any box $(r,c)$ 
with $1\lequ c \lequ k$ which is the bottom box of $\lambda$ in column $c$. 

In case (i), we have $r+k+1-c=e+i$, and claim that $q_{r+1} \lequ n+e
< p_{r+1}$. Indeed, since $\mu_i\gequ k+1+e$ and $\mu_{r+1}\gequ c$,
we have $\mu_i+\mu_{r+1} > c+e+k=r+2k+1-i$. Therefore
$\#\{h < r+1\ |\ \mu_h+\mu_{r+1} \lequ 2k+r+1-h\} \lequ r-i$, hence
\begin{align*}
q_{r+1} & = n+k+1 - \mu_{r+1}+ 
\#\{h < r+1\ |\ \mu_h+\mu_{r+1} \lequ 2k+r+1-h\} \\
& \lequ n+k+1-c+r-i= n+e.
\end{align*}
We also have $\la_{r+1}<c$ and $\la_i< k+1+e$, hence
$\la_i+\la_{r+1}\lequ 2k+r-i$, and by similar reasoning
$p_{r+1} \gequ n+k+1-(c-1)+(r+1-i)\gequ n+e+2$. 

In case (ii), we claim that there exists a box $(r,c)$ of $\la$ with
$\la_r =c$ such that $k-\la_r+r=e+i$. Indeed, choose $r$ minimal such
that $k+1-\la_r+r > e+i$. If $k+1-\la_r+r \gequ e+i+2$ then
$\la_{r-1}>\la_r$, and the box in position $(r-1, k+r-e-i)$ is the
bottom box in its column. It follows that $k-c+r = e+i$, where
$c=\la_r$. 

We claim that $q_r \lequ n+e < p_r$. To see this, note that if 
$\mu_r< \la_r$ then the box $(i,k+1+e)$ does not lie in $\A$; hence 
$\mu_r\gequ \la_r$. Also, $\mu_i\gequ k+1+e$ and therefore
$\mu_r+\mu_i\gequ \la_r+k+1+e > 2k+r-i$. It follows that 
\[
q_r = n+k+1 - \mu_r + \#\{h<r \ |\ \mu_h+\mu_r\lequ 2k+r-h\}
\lequ n+k+1-\la_r+ r-1-i = n+e.
\]
We also have $\la_i+\la_r = \la_i + (k+r-e-i) \lequ 
(e+k)+(k+r-e-i)=2k+r-i$. Hence
\[
p_r= n+k+1-\la_r+\#\{h<r\ |\ \la_h+\la_r\lequ 2k+r-h\}
> n+e.
\]

It remains to show the converse. If $n-e$ is a crossing, then there exists
a box $B=(i,k+1+e)$
of $\mu\ssm\la$.
Now if $B$ does not lie in $\A$, then
we need to show that $n+e$ is a
cut, i.e., there exists a $j$ such that $p_j \lequ n+e < q_{j+1}$.
Either the box $(r,c)$ is a bottom box of $\la$ in column $c$
$k$-related to $B$ which is also a bottom box of $\mu$ in column $c$,
or some $d$ boxes $(r-d+1,c),\ldots, (r,c)$ are removed from $\la$ and
$B$ is $k$-related to box $(r-s,c)$, with $0\lequ s \lequ d$.  In the
former case, set $s=0$. We have $k+1-c+r-s=e+i$ and $\la_{r-s}\gequ
c$, while $\mu_{r-s+1}<c$. Since the part of $\mu\ssm\la$ in columns
$k+1$ through $k+n$ is a
horizontal strip containing $B$, we have $\la_{i-1}\gequ k+1+e$ and
$\mu_{i+1}<k+1+e$.  Now
\[
\la_{i-1}+\la_{r-s}\gequ k+1+e+c=2k+2+r-s-i,
\]
hence $\la_{i-1}+\la_{r-s}> 2k+(r-s)-(i-1)$, so 
\[
p_{r-s}\lequ n+k+1-\la_{r-s}+ \#\{h< r-s\ |\ 
\la_h+\la_{r-s}\lequ 2k+r-s-h\} \lequ n+e.
\]
Furthermore, $\mu_{r-s+1}+\mu_{i+1}\lequ 
c-1+e+k = 2k + (r-s+1)-(i+1)$, thus
\[
q_{r-s+1}=n+k+1-\mu_{r-s+1} + \#\{h \lequ r-s\,|\,
\mu_h+\mu_{r-s+1} \lequ 2k + r-s+1-h \} \gequ n+e+1. \qedhere
\]
\end{proof}
Lemma \ref{fundbij} immediately implies the following result.
\begin{prop}
\label{multpqlm}
If $P\to Q$ are index sets for $\IG(n-k,2n)$ and $\lambda
\to\mu$ are the corresponding $k$-strict partitions, then 
we have $N(P,Q)=N(\la,\mu)$. 
\end{prop}

We next discuss the analogous proof of Theorem~\ref{thm.ogevenpieri}.

\begin{prop}
\label{pqlmD}
  Let $P$ and $Q$ be index sets for $\OG(n+1-k,2n+2)$ and let $\lambda$
  and $\mu$ be the corresponding elements of $\wt{P}(k,n)$.  Then $P \to
  Q$ if and only if $\lambda \to \mu$. In this situation we also have
  $N'(P,Q) = N'(\lambda,\mu)$, and if $|\mu|=|\la|+k$ and 
  $N'(P,Q)=N'(\la,\mu)=0$, then $h(P,Q)\equiv h(\la,\mu)\ \text{mod}\ 2$.
\end{prop}
\begin{proof}
The equivalence of the relations $P \to Q$ and $\la \to \mu$ is proved
in a similar way to Proposition \ref{pqlm}. The condition in the
definition of $P\to Q$ that $D(P,Q)$ cannot have exactly three boxes
in columns $n+1$ and $n+2$ is equivalent to the condition
$\type(\la)+\type(\mu)\neq 3$ in the definition of $\la \to\mu$. This
follows because the type of an index set $P$, when non-zero, equals 1
plus the parity of the number of integers in $[1,n+1]\ssm P$. We thus
see that when $\type(P)$ and $\type(Q)$ are
both positive, we must have $\type(P)=\type(Q)$.

One also checks as in Lemma \ref{fundbij} that in this
situation, we have $N'(P,Q) = N'(\lambda,\mu)$, and we will assume
this in the following.  To complete the proof, we show that whenever
$P \to Q$ (equivalently $\la \to \mu$), $|\mu|=|\la|+k$, and
$N'(P,Q)=N'(\la,\mu)=0$, we have $h(P,Q)\equiv h(\la,\mu)\ \text{mod}\
2$. We need to verify that
\[
g(\la,\mu)+\max(\type(P),\type(Q))+|S|+|S'|+n \equiv 0\ \text{mod}\ 2,
\]
where the sets $S$ and $S'$ have been defined before Theorem
\ref{indexpieriD}.

Since $|\mu|=|\la|+k$ and $\mu\ssm\la$ can have at most one box in
each of the first $k$ columns, we deduce that for each 
column among these where $\mu$ has the same number of
boxes as the corresponding column of $\la$, the bottom box of $\mu$ in
this column is $k'$-related to exactly one box of $\mu\ssm\la$. It 
follows that $g(\la,\mu)$ is equal to the number of boxes of 
$\mu\ssm\la$ in columns $k+1$ through $k+n$ minus the number of boxes
in $\la\ssm\mu$. 

We claim that at least one of $\type(\la)$, $\type(\mu)$ must be
positive. Indeed, if $\type(\la)=0$ then the Pieri move $\la\to\mu$
must add a box to column $k$, which implies that $\type(\mu)>0$.  This
fact has the following consequence for the diagram $D(P,Q)$: exactly
one row of $D(P,Q)$ has boxes in columns $n+1$ or $n+2$, and this row
begins or ends in these two columns. Moreover, if this row has boxes
in both central columns, then it must contain at least one more box. 
It also follows from the definitions that no two boxes of $D(P,Q)$ can
lie in the same column.

Using Proposition \ref{P:trnslD} it is easy to translate 
between parts $\la_i>k$ and elements of $p_i$ of $P$ less 
than $n+1$: they are related by the equation $\la_i+p_i=n+k+1$. 
We deduce that the number of boxes of $\mu\ssm\la$ in columns
$k+1$ through $k+n$ equals $|S|$ minus the number $r$ of nonempty 
rows in the left half of $D(P,Q)$ (the part of $D(P,Q)$ which lies
in the first $n+1$ columns).

Let $S'' = \{p\in S'\ :\ p > n+2\}$. In the proof of Theorem
\ref{indexpieriD} we saw that $q_j=p_j$, for all $p_j\in S''$;
moreover, since $N'(\la,\mu)=0$, the set $\A$ is empty. Using these
facts, it is straightforward to check that the elements of $S''$ are
in 1-1 correspondence with the {\em removed} boxes from $\la$, that
is, with $\la\ssm\mu$. Combining this with the previous analysis, we see
that $g(\la,\mu)+|S''|$ has the same parity as $|S|+r$. We are reduced
then to showing that
\[
n+r+\max(\type(P), \type(Q)) \equiv
\begin{cases}
0\ \text{mod}\ 2 & \text{if $n+2 \notin S'$},\\
1\ \text{mod}\ 2 & \text{if $n+2 \in S'$}.
\end{cases}
\]
The above is proved by a case by case analysis, according to three
possibilities for $(\type(P),\type(Q))$. If
$\type(P)=0$ and $\type(Q)>0$, then $n+2 \notin S'$.  Then, if $q_j=n+1$ for
some $j$ we have $j=r$ and $\type(Q) \equiv 1+ (n+1) - j \equiv n+r \
\text{mod}\ 2$, and if $q_j=n+2$ for some $j$ then $j=r+1$ and $\type(Q)
\equiv 1+ (n+1) - (j-1) \equiv n+r \ \text{mod}\ 2$. If
$\type(P)>0$ and $\type(Q)=0$,
then there are two possibilities.
Either $p_j=n+1$ for some $j$, so
$n+2 \notin S'$, $j=r$, and $\type(P) \equiv 1+ (n+1) - j \equiv n+r \
\text{mod}\ 2$, or else $p_j=n+2$ for some $j$, so $q_j<n+1$ and $n+2\in
S'$. Then $j=r$ and $\type(P) \equiv 1+ (n+1) - (j-1) \equiv n+r+1 \
\text{mod}\ 2$, as required. Finally if
$\type(P)=\type(Q)>0$ we must have $q_j=p_j\in \{n+1,n+2\}$ for some
$j$ while $n+2 \notin S'$, and the result again follows.
\end{proof}

\appendix
\section{Quantum cohomology of $\OG(n,2n+2)$}
\label{appendix}

In this section $\OG'$ will denote the Grassmannian
$\OG(n,2n+2)$. A presentation and Pieri rule for the classical
cohomology ring of $\OG'$ were obtained in Section \ref{QCOGeven}. 

Given nonnegative integers $d_1$ and $d_2$, a rational map of degree 
$(d_1,d_2)$ to $\OG'$ is a morphism $f: \bP^1 \to \OG'$ such that 
\[
\int_{\OG'} f_*[\bP^1] \cdot \ta_1 = d_1 
\ \ \ \mathrm{and} \ \ \
\int_{\OG'} f_*[\bP^1] \cdot \ta'_1 = d_2.
\]
For three elements 
$\la$, $\mu$, and $\nu$ in $\wt{\cP}(1,n)$ such that 
$|\lambda| + |\mu| + |\nu| = \dim(\OG')+(d_1+d_2)(n+1)$, the 
Gromov-Witten invariant $\langle \ta_\lambda, \ta_\mu, \ta_\nu
\rangle_{d_1,d_2}$ is defined to be the number of rational maps 
$f\colon \bP^1\to \OG'$ of
degree $(d_1,d_2)$ such that $f(0)\in X_\lambda(E_\bull)$, $f(1)\in
X_\mu(F_\bull)$, and $f(\infty)\in X_\nu(G_\bull)$, for given
isotropic flags $E_\bull$, $F_\bull$, and $G_\bull$ in general
position. 

The parameter space of lines on $\OG'$ is the space $Z$ of pairs $(A,B)$
where $A$ and $B$ are isotropic subspaces of $\C^{2n+2}$ of respective
dimensions $n-1$ and $n+1$.
Observe that $Z$ consists of two connected components, each isomorphic
to the isotropic two-step flag variety
$\OF'=\OF(n-1,n+1;2n+2)$.
One component, which we call $Z_1$, parametrizes lines of degree $(1,0)$
on $\OG'$, and the other component, which we call $Z_2$, parametrizes
lines of degree $(0,1)$.
It follows that 
\[
\langle \ta_{\lambda},\ta_\mu, \ta_{\nu}\rangle_{1,0} = 
\int_{Z_1} [Z_{\lambda}]\cdot [Z_\mu] \cdot [Z_\nu]
\]
where $Z_{\lambda}$, $Z_\mu$, and $Z_\nu$ are the associated Schubert
varieties in $Z_1$, defined as usual.  For any
$\lambda\in \wt{\cP}(1,n)$, let $\ov{\la}$ denote the strict partition
(without type) obtained by deleting the leftmost column of
$\lambda$. With this convention, we have
\begin{prop}
\label{dprop0n}
For any integer $p\in [1,n+1]$ and $\lambda$, $\mu\in \wt{\cP}(1,n)$
with $|\lambda|+|\mu|+p=\dim \OG' + n + 1$, we have
\[
\langle \ta_\lambda, \ta_{\mu}, \ta_p\rangle_{1,0} =
\int_{\OG(n+1,2n+2)} \ta_{\ov{\lambda}} \cdot \ta_{\ov{\mu}} \cdot
\ta_{p-1}
\]
if $\lambda$ and $\mu$ are both of type $0$ or $1$, and $\langle
\ta_\lambda, \ta_{\mu}, \ta_p\rangle_{1,0} =0$ if $\lambda$ or $\mu$
has type $2$.
\end{prop}
\noindent
A corresponding analysis applies to the degree $(0,1)$ Gromov-Witten
invariants. 

The quantum cohomology ring $\QH^*(\OG')$ is a $\Z[q_1,q_2]$-algebra
which is isomorphic to $\HH^*(\OG',\Z)\otimes_{\Z}\Z[q_1,q_2]$ as a
module over $\Z[q_1,q_2]$. Here both $q_1$ and $q_2$ are formal
variables of degree $n+1$.  The ring structure on $\QH^*(\OG')$ is
determined by the relation
\[
  \ta_\lambda \cdot \ta_\mu = 
  \sum \gw{\ta_\lambda, \ta_\mu, \ta_{\nu^\vee}}{d_1,d_2} 
\, \ta_\nu \, {q_1}^{d_1}{q_2}^{d_2},
\]
the sum over $d_1,d_2\gequ 0$ and $\nu\in \wt{\cP}(1,n)$ with
$|\nu| = |\lambda| + |\mu| - (n+1)(d_1+d_2)$.

With the definitions of $\wt{\cP}'(1,n+1)$, $\wt{\nu}$, and $\la^*$ as
in Section \ref{qcOGeven}, and $\text{t}(\nu)=\type(\nu)$, we have the
following analogue of Theorem \ref{ogevenqupieri}.

\begin{thm}[Quantum Pieri rule for $\OG'$] 
\label{ogevenqupierin}
For any $1$-strict partition $\la\in\wt{\cP}(1,n)$ and integer 
$p\in [1,n+1]$, we have 
\[
\ta_p \cdot \ta_\la =
\sum_{\la\to\mu} \delta_{\la\mu} \, 2^{N'(\la,\mu)}\,\tau_\mu +
\sum_{\la\to\nu} 
\delta_{\la\nu} \, 2^{N'(\la,\nu)}\,\tau_{\wt{\nu}}\, 
q_{\text{\em t}(\nu)} \, +\, 
\sum_{\la^*\to\rho} \delta_{\la^*\rho} \,
2^{N'(\la^*,\rho)} \,\ta_{\rho^*}\, q_1q_2
\]
in the quantum cohomology ring $\QH^*(\OG(n,2n+2))$. Here (i) the
first sum is classical, as in (\ref{ogevenclass}), (ii) the second sum
is over $\nu\in \wt{\cP}'(1,n+1)$ with $\la\to\nu$ and
$|\nu|=|\la|+p$, and (iii) the third sum is empty unless $\la_1=n+1$,
and over $\rho\in\wt{\cP}(1,n)$ such that $\rho_1=n+1$,
$\la^*\to\rho$, and $|\rho|=|\la|-n-1+p$.  Furthermore, 
the product
$\tau'_1\cdot\ta_{\la}$ is obtained by replacing 
$\delta$ with $\delta'$ throughout.
\end{thm}

Observe that if we set $q_1=q_2=q$ in Theorem \ref{ogevenqupierin},
then we obtain exactly the statement of Theorem \ref{ogevenqupieri}
with $k=1$.  
It should therefore not come as a surprise that 
the proof
of Theorem \ref{ogevenqupierin} is along the same lines as that of
Theorem \ref{ogevenqupieri}. There are no
linear $q$ terms in a quantum Pieri product $\ta_p \ta_\la$ unless
$\ell(\la)=n$.
All $\nu\in \wt{\cP}'(1,n+1)$ have
positive type, since $\nu_1 < n+2$, i.e., $\text{t}(\nu)\in\{1,2\}$
for all $\nu$ in the statement of the theorem.  In
addition, the quadratic $q$ terms are handled as in the proof of
Theorem \ref{ogoddqupieri}, using the relation $\ta_{n+1}^2=q_1q_2$,
which is easily checked directly. 
The assertions of Lemma \ref{L:fiberD} are true when
$k=1$, $d>2$, and $\ell(\la)=1$, with unaltered proof, and the
argument of the proof of Theorem \ref{T:qclasD}(d) then shows that there
are no cubic or higher-degree $q$ terms.

\begin{thm}[Ring presentation] 
\label{ogevenqpresn}
The quantum cohomology ring
$\QH^*(\OG(n,2n+2))$ is presented as a quotient of the polynomial
ring $\Z[\ta_1,\ta_1',\ta_2,\ldots,\ta_{n+1},
q_1,q_2]$ modulo the relations 
\begin{gather*}
\label{ogevenQR1'n}
\ta_1\Delta_n-q_1=
\ta'_1 \Delta_n-q_2=
\sum_{p=2}^{n+1}(-1)^p\tau_p
\Delta_{n+1-p}, \\
\ta_r^2 + \sum_{i=1}^r(-1)^i \ta_{r+i}c_{r-i}= 0,
 \ \  \ \ 2 \lequ r \lequ n,
\end{gather*}
and
\[
\ta_1\ta'_1-\ta_2=0,
\]
where the variables $c_p$ are defined by (\ref{ctotau}).
\end{thm}
\begin{proof}
Let $(1^r)$ denote the partition $(1,\ldots,1)$ of length $r$,
and $\ta_{(1^r)}$ and $\ta'_{(1^r)}$ denote the corresponding Schubert 
classes of types $1$ and $2$, respectively. The quantum Pieri rule
gives the relations
\begin{gather}
\label{qurels1}
\ta_1\ta_{(1^n)} = \ta_{(2,1^{n-1})}+q_1, \ \ \ \ 
\ta'_1\ta'_{(1^n)} = \ta'_{(2,1^{n-1})}+q_2, \\
\label{qurels2}
\ta_1\ta'_{(1^n)}= \ta'_{(2,1^{n-1})}+\ta_{n+1}, \ \ \ \ 
\ta'_1\ta_{(1^n)}= \ta_{(2,1^{n-1})}+\ta_{n+1}
\end{gather}
in $\QH^*(\OG')$. In contrast to the case when $k>1$, it is not true
here that the Schur determinant $\Delta_r$, for $1\lequ r\lequ n$, is equal
to a Schubert class in $\HH^*(\OG',\Z)$. However, by applying the 
Pieri rule to the monomials in the expansion of $\Delta_r$, noting 
that $\ta_1\ta'_1=\ta_2$, we deduce that
\[
\Delta_r = \ta_{(1^r)} + \ta'_{(1^r)} + \sum_{\mu} c_{\mu}\,\ta_{\mu},
\]
where the sum is over typed partitions $\mu$ with $\ell(\mu)<r$. The 
quantum Pieri rule and equations (\ref{qurels1}), (\ref{qurels2}) 
now easily imply all of the required relations.
\end{proof}

\end{document}